\documentclass[11pt]{article}

\pdfoutput=1

\usepackage{spur}

\title{Flip Graphs on Self-Complementary Ideals of Chain Products}
\author{Serena An and Holden Mui}
\date{January 2, 2024}

\begin{document}

\begin{asydef}
    void rectangle(real x, real y, real w, real h, pen c = rgb("97ffff"), pair o = (0, 0)) {
        pen outline = 0.8*c + 0.2*black;
        pair p00 = (x, y)+o;
        pair p01 = (x, y+h)+o;
        pair p10 = (x+w, y)+o;
        pair p11 = (x+w, y+h)+o;
        fill(p00 -- p01 -- p11 -- p10 -- cycle, c);
        draw(p00 -- p01 -- p11 -- p10 -- cycle, outline);
    }
    
    void darkrectangle(real x, real y, real w, real h, pair o = (0, 0)) {
        rectangle(x, y, w, h, rgb("79cdcd"), o);
    }

    void backperimeter(real w, real h, pair o = (0, 0)) {
        pen outline = black + opacity(0.05);
        for (int x = 1; x < w; ++x) {
            draw(shift(o)*((x, 0) -- (x, h)), outline);
        }
        for (int y = 1; y < h; ++y) {
            draw(shift(o)*((0, y) -- (w, y)), outline);
        }
    }

    void frontperimeter(real w, real h, pair o = (0, 0)) {
        pen outline = black;
        draw(shift(o)*((0, 0) -- (0, h) -- (w, h) -- (w, 0) -- cycle), outline);
    }
    
    real r = 0.75; // adjustable, between 0 and 1
    pair v1 = r*dir(210) + (1-r)*1/2*dir(210);
    pair v2 = r*dir(330) + (1-r)*dir(0);
    pair v3 = dir(90);
    
    pair p(real x, real y, real z, pair o = (0, 0)) {
        return o+x*v1+y*v2+z*v3;
    }
    
    void box(real x, real y, real z, real l1, real l2, real l3, bool l = false, pair o = (0, 0)) {
        pen p1 = rgb("97ffff");
        pen p2 = rgb("8deeee");
        pen p3 = rgb("79cdcd");
        pen outline = black + opacity(0.1);
        pair p000 = p(x, y, z, o);
        pair p001 = p(x, y, z+l3, o);
        pair p010 = p(x, y+l2, z, o);
        pair p011 = p(x, y+l2, z+l3, o);
        pair p100 = p(x+l1, y, z, o);
        pair p101 = p(x+l1, y, z+l3, o);
        pair p110 = p(x+l1, y+l2, z, o);
        pair p111 = p(x+l1, y+l2, z+l3, o);
        fill(p100 -- p110 -- p111 -- p101 -- cycle, p3);
        fill(p010 -- p110 -- p111 -- p011 -- cycle, p2);
        fill(p001 -- p101 -- p111 -- p011 -- cycle, p1);
        draw(p001 -- p011 -- p010 -- p110 -- p100 -- p101 -- cycle, outline);
        draw(p011 -- p111 ^^ p101 -- p111 ^^ p110 -- p111, outline);
        if (l) {
            pair c = (p001 + p111)/2;
            transform t = (0, 0, v1.x, v2.x, v1.y, v2.y)*rotate(135);
            label(t*string(z+1), c);
        }
    }

    void ideal(int[][] heights, bool labels = false, pair o = (0, 0)) {
        for (int i = 0; i < heights.length; ++i) {
            for (int j = 0; j < heights[i].length; ++j) {
                for (int h = 0; h < heights[i][j]; ++h) {
                    bool l;
                    if (h == heights[i][j] - 1) {
                        l = true;
                    } else {
                        l = false;
                    }
                    box(i, j, h, 1, 1, 1, l && labels, o);
                }
            }
        }
    }

    void backframe(real l1, real l2, real l3, pair o = (0, 0)) {
        pen o1 = black + opacity(0.1);
        pen o2 = black + opacity(0.05);
        pair p000 = p(0, 0, 0, o);
        pair p001 = p(0, 0, l3, o);
        pair p010 = p(0, l2, 0, o);
        pair p100 = p(l1, 0, 0, o);
        draw(p000--p001 ^^ p000--p010 ^^ p000--p100, o1);
        for (int x = 1; x < l1; ++x) {
            draw(shift(x*v1)*(p010--p000--p001), o2);
        }
        for (int y = 1; y < l2; ++y) {
            draw(shift(y*v2)*(p100--p000--p001), o2);
        }
        for (int z = 1; z < l3; ++z) {
            draw(shift(z*v3)*(p100--p000--p010), o2);
        }
    }
    
    void frontframe(real l1, real l2, real l3, pair o = (0, 0)) {
        pen outline = black + opacity(0.5);
        pair p001 = p(0, 0, l3, o);
        pair p010 = p(0, l2, 0, o);
        pair p011 = p(0, l2, l3, o);
        pair p100 = p(l1, 0, 0, o);
        pair p101 = p(l1, 0, l3, o);
        pair p110 = p(l1, l2, 0, o);
        pair p111 = p(l1, l2, l3, o);
        draw(p111--p110 ^^ p111--p101 ^^ p111--p011, outline);
        draw(p001--p011--p010--p110--p100--p101--cycle, outline);
    }
\end{asydef}

\maketitle
\begin{abstract}
    In this paper, we introduce a flip operation on self-complementary ideals of chain product posets and study the resulting flip graphs. We give asymptotics for the number of vertices in these graphs, compute their diameters, and give bounds for their radii. We also define similar flip operations on self-complementary ideals of the chain product $[2r]\times [2r]\times [2r]$ satisfying additional symmetries, and we achieve similar results for the resulting flip graphs. 
\end{abstract}

\tableofcontents

\section{Introduction}\label{section:introduction}

Self-complementary ideals of chain products originate from a generalization of maximal intersecting families. The graph of maximal intersecting families of a fixed set $S$ is formed by taking all maximal intersecting families as vertices, and connecting two maximal intersecting families if they differ in exactly one set. Asymptotics for the number of maximal intersecting families are given in \cite{brouwer2013counting}; this paper also computes this vertex count exactly for all $\abs{S} \leq 9$. The structure of this graph is studied in \cite{loeb1997graph}, which computes bounds and some exact values for its diameter and radius. Another related graph is studied in \cite{meyerowitz1995maximal}, where maximal intersecting families that are the same under a permutation of $S$ are mapped to the same vertex.

Because the poset of all subsets of a finite set of size $n$ is isomorphic to an $n$-fold product of the two-element chain, this setup admits a natural generalization to chains of arbitrary length. In this setting, maximal intersecting families correspond to self-complementary ideals. The corresponding graph, which we refer to as a \emph{flip graph}, is formed by taking all self-complementary ideals of a chain product as vertices and connecting two ideals if they differ in the minimum possible number of elements. Questions about the structure of the flip graph include vertex count bounds, diameter bounds, and radius bounds.

Additional types of flip graphs can be created by imposing symmetry restrictions on the self-complementary ideals that form the vertices of the flip graph. One such example is the flip graph on all cyclically symmetric self-complementary (CSSC) ideals of a product of three chains of equal length; two ideals are connected if they differ in exactly three elements. Another such example considers the set of all totally symmetric self-complementary (TSSC) ideals; in this setting, some edges need to be weighted differently to account for the fact that some vertices of the poset lie on axes of symmetry. The total number of vertices in these graphs are computed in \cite{kuperberg1994symmetries} and \cite{andrews1994plane}, respectively, via connections to plane partitions.

The main results can be summarized in the table below. The vertex count for the flip graphs on cyclically symmetric and totally symmetric self-complementary ideals are known exactly; asymptotics are given for all three cases. The diameter of the flip graph is also known exactly in all three cases. The radius is known exactly only for the flip graph on cyclically symmetric self-complementary ideals; bounds are given for the other two cases.

\begin{table}[htbp]
\centering
\begin{tabular}{c|c||c|c|c}
poset & symmetry & vertex count & diameter & radius \\ \hline
$[\ell_1] \times \dots \times [\ell_d]$ & self-complementary & \ref{theorem:scvertex1}, \ref{theorem:scvertex2}, \ref{theorem:scvertex3} & \ref{theorem:scdiameter} & \ref{theorem:scradius}, \ref{theorem:scradius3d} \\
$[2r] \times [2r] \times [2r]$ & CSSC & \ref{theorem:csscvertex} & \ref{theorem:csscdiameter} & \ref{theorem:csscradius} \\
$[2r] \times [2r] \times [2r]$ & TSSC & \ref{theorem:tsscvertex} & \ref{theorem:tsscdiameter} & \ref{conjecture:tsscradius} 
\end{tabular}
\end{table}

\subsection{Organization}

In \cref{section:preliminaries}, the relevant terms and definitions are stated. In \cref{section:sc}, results for the vertex count, diameter, and radius of flip graphs on self-complementary ideals are given. In \cref{section:cssc}, results for the vertex count, diameter, and radius of flip graphs on cyclically symmetric self-complementary ideals are given. In \cref{section:tssc}, results for the vertex count, diameter, and radius of flip graphs on totally symmetric self-complementary ideals are given. In \cref{section:code}, we provide links to the code we wrote to compute graph properties for various small cases, and we provide some sample code output. In \cref{section:futuredirections}, future directions for research are proposed. In \cref{section:acknowledgments}, we give our acknowledgments.

The structure of this paper encourages jumping around. In fact, each section after the preliminaries section can be read independently from the others.

\section{Preliminaries}\label{section:preliminaries}
\subsection{Posets}

A \emph{poset} $(P, \leq_P)$ is a set $P$ with a binary relation $\leq_P$ that satisfies the following three conditions for all $a, b, c \in P$:
\begin{itemize}
    \item $a \leq_P a$ (reflexivity),
    \item $a \leq_P b$ and $b \leq_P a$ imply $a = b$ (antisymmetry), and
    \item $a \leq_P b$ and $b \leq_P c$ imply $a \leq_P c$ (transitivity).
\end{itemize}
When the binary relation $\leq_P$ is unambiguous, it is denoted with the symbol $\leq$. Similarly, posets are denoted with their base set for brevity.

A poset is \emph{self-dual} if there exists an involution $\varphi: P \to P$ such that
\[a \leq_P  b \iff \varphi(b) \leq_P \varphi(a).\]
When referring to a self-dual poset, it is assumed that the map $\varphi$ is given and fixed.

The \emph{product} $P_1 \times \dots \times P_n$ of a finite number of posets $P_1, \dots, P_n$ is the poset on all tuples $(a_1, \dots, a_n)$ with $a_1 \in P_1, \dots, a_n \in P_n$ and the relation \[(a_1, \dots, a_n) \leq_{P_1 \times \dots\times P_n} (b_1, \dots, b_n) \iff \text{$a_1 \leq_{P_1} b_1$, $\dots$, $a_n \leq_{P_n} b_n$.}\]
Observe that a product of self-dual posets is self-dual; if $P_1, \dots, P_n$ come with involutions $\varphi_1, \dots, \varphi_n$, then the involution on $P_1 \times \dots \times P_n$ is given by
\[(a_1, \dots, a_n) \mapsto (\varphi_1(a_1), \dots, \varphi_n(a_n)).\]

One example of a poset, and also the primary poset of importance in this paper, is a chain. Given a positive integer $\ell$, the \emph{chain} with $\ell$ elements is the poset with elements $\{1, 2, \dots, \ell\}$ and the binary relation given by the natural ordering on the integers. This poset is denoted with the notation $[\ell]$. Chains are self-dual since the involution $\varphi: \{1, \dots, \ell\} \to \{1, \dots, \ell\}$ sending $i$ to $\ell+1-i$ for each $i \in \{1, \dots, \ell\}$ reverses the natural ordering on the integers.

\subsection{Ideals}

Let $P$ be a poset. An \emph{ideal} of $P$ is a subset $I \subseteq P$ with the property that \[\text{$a \in I$ and $b \leq a$} \implies b \in I\]
for all $a\in I, b\in P$. Given an ideal $I$, an element $a\in I$ is \emph{maximal} if there does not exist $b\in I$ with $b \neq a$ and $a \leq b$.

\begin{remark}
    The standard definition of an ideal states that ideals are nonempty, but for our purposes, we will allow empty ideals.
\end{remark}

Given a self-dual poset $P$, an ideal $I\subset P$ is \emph{self-complementary} if 
\[a\in I\iff \varphi(a)\notin I.\] 
Observe that if $I$ is self-complementary, then $I \sqcup \varphi(I) = P$, where $\varphi(I) = \{\varphi(a) \mid a \in I\}$. In particular, only self-dual posets with an even number of elements can have self-complementary ideals.

If the poset $P$ is a product of chains $[\ell_1] \times \dots \times [\ell_d]$ for a sequence of positive integers $\ell_1, \dots, \ell_d$, an ideal $I \subset P$ is self-complementary if 
\[(a_1, \dots, a_d) \in I \iff (\ell_1 + 1 - a_1, \dots, \ell_d + 1 - a_d)\notin I.\]

If the chains that comprise the poset all have the same number of elements, it is possible for ideals to have additional types of symmetries. Let $\ell$ be a positive integer; an ideal $I \subseteq [\ell] \times [\ell] \times [\ell]$ is \emph{cyclically symmetric} if
\[(a_1, a_2, a_3) \in I \implies \left(a_{\sigma(1)}, a_{\sigma(2)}, a_{\sigma(3)}\right) \in I\]
for all 3-cycles $\sigma \in S_3$. Similarly, an ideal $I \subseteq [\ell] \times [\ell] \times [\ell]$ is \emph{totally symmetric} if
\[(a_1, a_2, a_3) \in I \implies \left(a_{\sigma(1)}, a_{\sigma(2)}, a_{\sigma(3)}\right) \in I\]
for all permutations $\sigma \in S_3$.

\cref{fig:scexample} is a visualization of a self-complementary ideal of the poset $[6] \times [6] \times [6]$. Elements of the ideal are represented as cubes, and the poset itself is represented by the rectangular skeleton.  

\cref{fig:csscexample} and \cref{fig:tsscexample} are examples of a cyclically symmetric self-complementary ideal and a totally symmetric self-complementary ideal, respectively, of the poset $[6] \times [6] \times [6]$.

\begin{figure}[p]
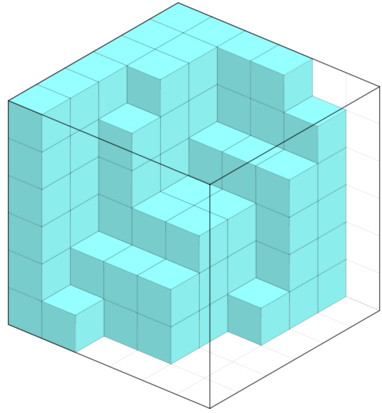

    \centering
    \begin{diagram}[0.3]
        backframe(6, 6, 6);
        ideal(new int[][] {{6, 6, 6, 6, 5, 0},
                           {6, 6, 4, 4, 4, 0}, 
                           {6, 6, 3, 3, 1, 0},
                           {6, 5, 3, 3, 0, 0},
                           {6, 2, 2, 2, 0, 0},
                           {6, 1, 0, 0, 0, 0}});
        frontframe(6, 6, 6);
    \end{diagram}
    \caption{A self-complementary ideal of the poset $[6] \times [6] \times [6]$.}
    \label{fig:scexample}
\end{figure}

\begin{figure}[p]
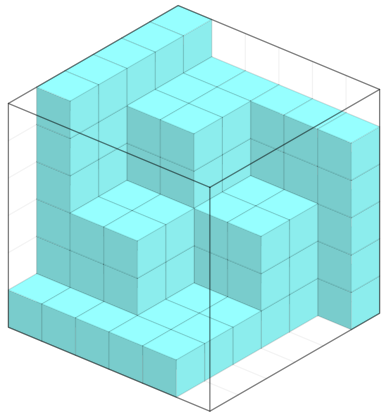

    \centering
    \begin{diagram}[0.3]
        backframe(6, 6, 6);
        ideal(new int[][] {{6, 5, 5, 5, 5, 5},
                           {6, 5, 5, 3, 3, 0},
                           {6, 5, 5, 3, 3, 0},
                           {6, 3, 3, 1, 1, 0},
                           {6, 3, 3, 1, 1, 0},
                           {1, 1, 1, 1, 1, 0}});
        frontframe(6, 6, 6);
    \end{diagram}
    \caption{A cyclically symmetric self-complementary ideal of the poset $[6] \times [6] \times [6]$.}
    \label{fig:csscexample}
\end{figure}

\begin{figure}[p]
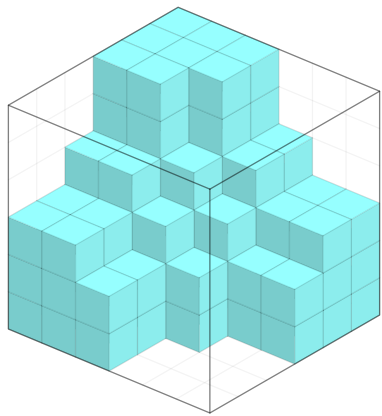

    \centering
    \begin{diagram}[0.3]
        backframe(6, 6, 6);
        ideal(new int[][] {{6, 6, 6, 4, 3, 3},
                           {6, 6, 6, 4, 3, 3},
                           {6, 6, 4, 3, 2, 2},
                           {4, 4, 3, 2, 0, 0},
                           {3, 3, 2, 0, 0, 0},
                           {3, 3, 2, 0, 0, 0}});
        frontframe(6, 6, 6);
    \end{diagram}
    \caption{A totally symmetric self-complementary ideal of the poset $[6] \times [6] \times [6]$.}
    \label{fig:tsscexample}
\end{figure}

\subsection{Graphs}
Given a graph $G$, let let $V(G)$ denote its vertex set and let $v(G) = \abs{V(G)}$ denote its vertex count.

Let $G$ be a connected, unweighted graph. The \emph{distance} between two vertices $v, v' \in V(G)$ is the minimum possible number of edges in any path connecting $v$ and $v'$. Similarly, let $G$ be a connected, weighted graph with nonnegative edge weights. The \emph{distance} between two vertices $v, v' \in V(G)$ is the minimum possible sum of the edge weights along any path connecting $v$ and $v'$. Observe that the distance in a weighted graph with edge weights all equal to 1 coincides with the distance in the corresponding unweighted graph. The distance between two vertices $u, v \in V(G)$ is denoted $\dist(u, v)$.

Let $G$ be a connected graph that can be weighted or unweighted. The \emph{eccentricity} of a vertex $v \in G$ is the maximum distance between $v$ and any vertex in $G$. The \emph{diameter} of $G$ is the maximum eccentricity of any vertex $v \in G$, and the \emph{radius} of $G$ is the minimum eccentricity of any vertex $v \in G$. The \emph{center} of $G$ is the set of all vertices with eccentricity equal to the radius. The \emph{perimeter} of $G$ is the set of all vertices with eccentricity equal to the diameter. 

\section{Self-Complementary Ideals}\label{section:sc}
Given a self-dual poset $P$, the \emph{flip graph on self-complementary ideals of $P$} is the unweighted graph whose vertices are the self-complementary ideals of $P$, and whose edges connect self-complementary ideals $I, J \subset P$ for which
\[\abs{I \setminus J} = \abs{J \setminus I} = 1.\]
When the context is clear, this graph will simply be referred to as the \emph{flip graph}. The graph is named this way because traversing an edge of the graph can be thought of as a ``flip'': taking a maximal element $a$ of some ideal and replacing it with its dual $\varphi(a)$. \Cref{fig:scgraph} shows the flip graph structure for the specific case $P = [2] \times [3] \times [4]$.

\begin{figure}[htbp]
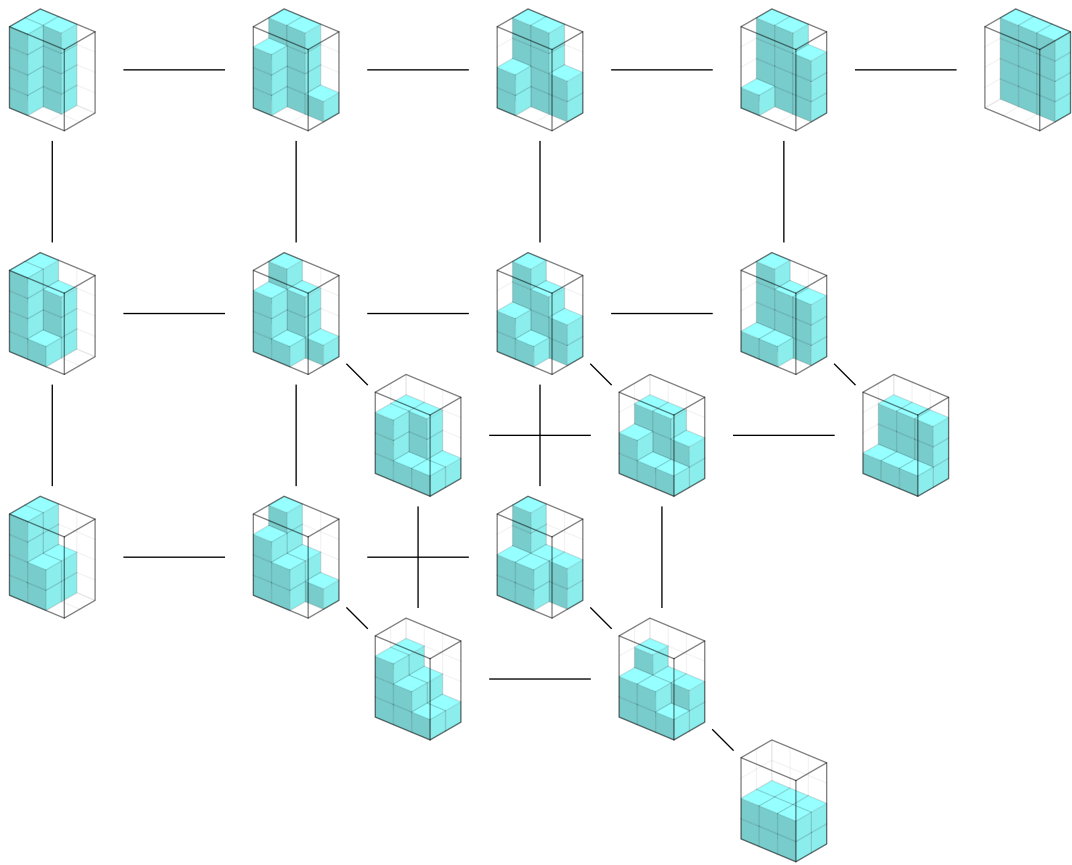

    \centering
    \begin{diagram}[0.9]
        void v(int[][] heights, pair o) {
            pair s = (2*v1+3*v2+4*v3)/2;
            fill(circle(o, 3.5), white);
            backframe(2, 3, 4, o-s);
            ideal(heights, o-s);
            frontframe(2, 3, 4, o-s);
        }
        int[][][] ideals = new int[][][] {
            {{4, 4, 4}, {0, 0, 0}},
            {{4, 4, 3}, {1, 0, 0}}, 
            {{4, 4, 2}, {2, 0, 0}}, 
            {{4, 4, 1}, {3, 0, 0}}, 
            {{4, 4, 0}, {4, 0, 0}}, 
            {{4, 3, 3}, {1, 1, 0}}, 
            {{4, 3, 2}, {2, 1, 0}}, 
            {{4, 3, 1}, {3, 1, 0}}, 
            {{4, 3, 0}, {4, 1, 0}}, 
            {{4, 2, 2}, {2, 2, 0}}, 
            {{4, 2, 1}, {3, 2, 0}}, 
            {{4, 2, 0}, {4, 2, 0}}, 
            {{3, 3, 3}, {1, 1, 1}}, 
            {{3, 3, 2}, {2, 1, 1}}, 
            {{3, 3, 1}, {3, 1, 1}}, 
            {{3, 2, 2}, {2, 2, 1}}, 
            {{3, 2, 1}, {3, 2, 1}}, 
            {{2, 2, 2}, {2, 2, 2}}};
        real s = 12;
        pair[] locations = new pair[] {
            (4*s, 2*s),
            (3*s, 2*s),
            (2*s, 2*s),
            (1*s, 2*s),
            (0*s, 2*s),
            (3*s, 1*s),
            (2*s, 1*s),
            (1*s, 1*s),
            (0*s, 1*s),
            (2*s, 0*s),
            (1*s, 0*s),
            (0*s, 0*s),
            (3.5*s, 0.5*s),
            (2.5*s, 0.5*s),
            (1.5*s, 0.5*s),
            (2.5*s, -0.5*s),
            (1.5*s, -0.5*s),
            (3*s, -1*s)};
        int[][] edges = new int[][] {
            {0, 1},
            {1, 2},
            {1, 5},
            {2, 3},
            {2, 6},
            {3, 4},
            {3, 7},
            {4, 8},
            {5, 6},
            {5, 12},
            {6, 7},
            {6, 9},
            {6, 13},
            {7, 8},
            {7, 10},
            {7, 14},
            {8, 11},
            {9, 10},
            {9, 15},
            {10, 11},
            {10, 16},
            {12, 13},
            {13, 14},
            {13, 15},
            {14, 16},
            {15, 16},
            {15, 17}
        };
        for (int i = 0; i < edges.length; ++i) {
            int v1 = edges[i][0];
            int v2 = edges[i][1];
            draw(locations[v1] -- locations[v2]);
        }
        for (int i = 0; i < ideals.length; ++i) {
            v(ideals[i], locations[i]);
        }
    \end{diagram}
    \caption{The flip graph on self-complementary ideals of $[2] \times [3] \times [4]$.}
    \label{fig:scgraph}
\end{figure}

The following lemma, although simple to prove, is important for establishing bounds on the distance between self-complementary ideals in the flip graph.

\begin{lemma}\label{lemma:scdistance}
    Let $I$ and $J$ be self-complementary ideals of a self-dual poset $P$. Then the distance from $I$ to $J$ in the flip graph of $P$ is $\abs{I \setminus J}$. In particular, the flip graph of any self-dual poset is connected.
\end{lemma}

\begin{proof}
    The distance from $I$ to $J$ must be at least $\abs{I \setminus J}$ since each flip decreases $\abs{I \setminus J}$ by at most 1. Thus, it suffices to give a path from $I$ to $J$ using exactly $\abs{I \setminus J}$ flips. Indeed, repeatedly flipping a maximal element in $I \setminus J$ works.
\end{proof}

Given a positive integer $d$ and a sequence of positive integers $\ell_1, \dots, \ell_d$, let $G(\ell_1, \dots, \ell_d)$ denote the flip graph on self-complementary ideals of $[\ell_1] \times \dots \times [\ell_d]$.

\subsection{Vertex Count}

Consider the graph $G(\ell_1, \dots, \ell_d)$ for some sequence of positive integers $\ell_1, \dots, \ell_d$. When $d \leq 3$, the exact vertex count of $G(\ell_1, \dots, \ell_d)$ is known.

\begin{theorem}\label{theorem:scvertex1}
    Let $d \leq 3$ be a positive integer, and let $\ell_1, \dots, \ell_d$ be a sequence of positive integers with even product.
    \begin{itemize}
        \item If $d=1$, the number of vertices in $G(\ell_1)$ is 1.
        \item If $d=2$, the number of vertices in $G(\ell_1, \ell_2)$ is $\displaystyle\binom{\floor{\ell_1/2}+\floor{\ell_2/2}}{\floor{\ell_1/2}}$.
        \item If $d=3$, the number of vertices in $G(\ell_1, \ell_2, \ell_3)$ is \[ \left(\prod_{i_1=1}^{\floor{\ell_1/2}}\prod_{i_2=1}^{\ceil{\ell_2/2}}\prod_{i_3=1}^{{\ell_3/2}} \frac{i_1+i_2+i_3-1}{i_1+i_2+i_3-2}\right)\left(\prod_{i_1=1}^{\ceil{\ell_1/2}}\prod_{i_2=1}^{\floor{\ell_2/2}}\prod_{i_3=1}^{{\ell_3/2}} \frac{i_1+i_2+i_3-1}{i_1+i_2+i_3-2}\right),\]
        assuming $\ell_3$ is even.
    \end{itemize}
\end{theorem}

\begin{proof}
    If $d=1$ and $\ell_1$ is even, there is only one self-complementary ideal.

    If $d=2$, there is a bijection between self-complementary ideals of $[\ell_1] \times [\ell_2]$ and rotationally symmetric up-left lattice paths from $(\ell_1, 0)$ to $(0, \ell_2)$. Observing that each rotationally symmetric path bijects to a up-left lattice path from $(\ell_1, 0)$ to $(\ceil{\ell_1/2}, \floor{\ell_2/2})$ gives the desired result. (An example of the bijection is given in \cref{fig:sc2d}.)

    If $d=3$, the result is given in \cite{stanley1986symmetries}.
\end{proof}

\begin{figure}[htbp]
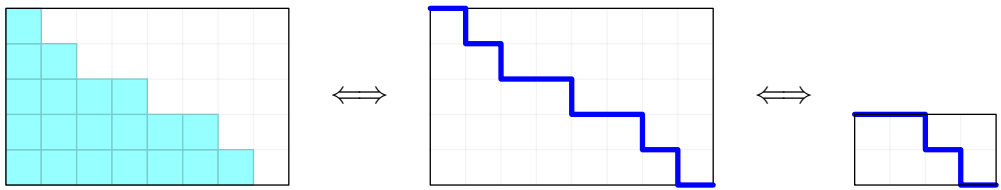

    \centering
    \begin{diagram}[0.8]
        backperimeter(8, 5); 
        rectangle(0, 0, 1, 1);
        rectangle(1, 0, 1, 1);
        rectangle(2, 0, 1, 1);
        rectangle(3, 0, 1, 1);
        rectangle(4, 0, 1, 1);
        rectangle(5, 0, 1, 1);
        rectangle(6, 0, 1, 1);
        rectangle(0, 1, 1, 1);
        rectangle(1, 1, 1, 1);
        rectangle(2, 1, 1, 1);
        rectangle(3, 1, 1, 1);
        rectangle(4, 1, 1, 1);
        rectangle(5, 1, 1, 1);
        rectangle(0, 2, 1, 1);
        rectangle(1, 2, 1, 1);
        rectangle(2, 2, 1, 1);
        rectangle(3, 2, 1, 1);
        rectangle(0, 3, 1, 1);
        rectangle(1, 3, 1, 1);
        rectangle(0, 4, 1, 1);
        frontperimeter(8, 5);
        pair s = (12, 0);
        backperimeter(8, 5, s);
        frontperimeter(8, 5, s);
        draw(shift(s)*((0,5)--(1,5)--(1,4)--(2,4)--(2,3)--(4,3)--(4,2)--(6,2)--(6,1)--(7,1)--(7,0)--(8,0)), blue+linewidth(2));
        pair s = (24, 0);
        backperimeter(4, 2, s);
        draw(shift(s)*((0,2)--(2,2)--(2,1)--(3,1)--(3,0)--(4,0)), blue+linewidth(2));
        frontperimeter(4, 2, s);
        label("$\iff$", (10, 2.5));
        label("$\iff$", (22, 2.5));
    \end{diagram}
    \caption{An example of the bijection in the case $d=2$.}
    \label{fig:sc2d}
\end{figure}

Determining exact answers for $d \geq 4$ in general seems to be a difficult problem; even getting asymptotically tight results appears to be difficult due to the intractability of counting solid partitions. However, some statements about asymptotics can be made, depending on which parameters are fixed and which are varied.

\begin{theorem}\label{theorem:scvertex2}
    Let $d$ be a positive integer, and let $\ell_1, \dots, \ell_{d-1}$ be a fixed sequence of positive integers. Then as $n$ tends to infinity, we have
    \[v(G(\ell_1, \dots, \ell_{d-1}, n)) \sim Cn^{\floor{\ell_1 \cdots \ell_{d-1} / 2}}\]
    for some constant $C$. (If $\ell_1 \cdots \ell_{d-1}$ is even, the asymptotic runs over all positive integers $n$. If $\ell_1 \cdots \ell_{d-1}$ is odd, the asymptotic runs over all even positive integers $n$.)
\end{theorem}

\begin{proof}
    Let $P$ denote the poset
    \[P = \begin{cases}
        [\ell_1] \times \dots \times [\ell_{d-1}] & \text{if $\ell_1 \cdots \ell_{d-1}$ even} \\ 
        \left([\ell_1] \times \dots \times [\ell_{d-1}]\right) \setminus (\tfrac{1}{2}(\ell_1+1), \dots, \tfrac{1}{2}(\ell_{d-1}+1)) & \text{if $\ell_1 \cdots \ell_{d-1}$ odd.}
    \end{cases}\]
    It suffices to find asymptotics for the number of order-preserving functions $f: P \to \{0, \dots, n\}$ satisfying $\varphi_{\{0, \dots, n\}}(f(a)) = f(\varphi_P(a))$ for all $a \in P$, where $\varphi_{\{0, \dots, n\}}$ and $\varphi_P$ denote the self-dual involution on $\{0, \dots, n\}$ and $P$, respectively; this condition is equivalent to the corresponding ideal in $P \times [n]$ being self-complementary. To do this, pick any self-complementary ideal $I \subset P$. Observe that $\abs{I} = \floor{\tfrac{1}{2} \ell_1 \cdots \ell_{d-1}}$ and that any function $g: I \to \{0, \dots, n\}$ extends uniquely to a function $\tilde{g}: P \to \{0, \dots, n\}$ satisfying $\varphi_{\{0, \dots, n\}}(\tilde{g}(a)) = \tilde{g}(\varphi_P(a))$ for all $a \in P$. (If $\ell_1\cdots\ell_{d-1}$ is odd, then $\tilde{g}((\frac{1}{2}(\ell_1+1), \dots, \frac{1}{2}(\ell_{d-1}+1))) = \frac{1}{2}n$ is forced.) An example of this extension is given in \cref{fig:scasymptoticcorrespondence}.

    Construct a function $h: I \to [0, 1)$ such that each of its $\abs{I}$ outputs are chosen uniformly and independently, and define $g: I \to \{0, \dots, n\}$ via $a \mapsto \floor{(n+1)h(a)}$; note that this uniformly selects a function from $I \to \{0, \dots, n\}$. In the limit of large $n$, the probability that $\tilde{g}$ is order-preserving is asymptotically equal to the probability that $h$ is order-preserving, since the probability that two outputs of $g$ have the same floor tends to zero. If $C$ denotes the probability that $h$ is order-preserving, the total number of order-preserving functions $f: P \to \{0, \dots, n\}$ satisfying $\varphi_{\{0, \dots, n\}}(\tilde{g}(a)) = \tilde{g}(\varphi_P(a))$ for all $a \in P$ is asymptotic to $Cn^{\abs{I}} = Cn^{\floor{\ell_1\cdots\ell_{d-1}/2}}$, as desired.
\end{proof}

\begin{figure}[htbp]
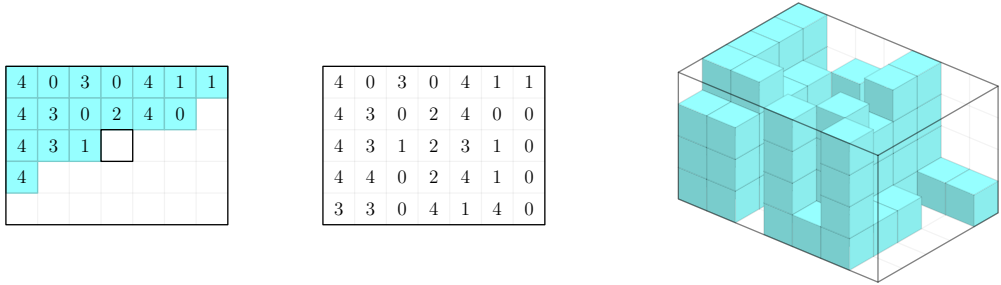

    \centering
    \begin{diagram}[0.8]
        backperimeter(7, 5); 
        rectangle(0, 4, 1, 1);
        rectangle(1, 4, 1, 1);
        rectangle(2, 4, 1, 1);
        rectangle(3, 4, 1, 1);
        rectangle(4, 4, 1, 1);
        rectangle(5, 4, 1, 1);
        rectangle(6, 4, 1, 1);
        rectangle(0, 3, 1, 1);
        rectangle(1, 3, 1, 1);
        rectangle(2, 3, 1, 1);
        rectangle(3, 3, 1, 1);
        rectangle(4, 3, 1, 1);
        rectangle(5, 3, 1, 1);
        rectangle(0, 2, 1, 1);
        rectangle(1, 2, 1, 1);
        rectangle(2, 2, 1, 1);
        rectangle(0, 1, 1, 1);
        frontperimeter(7, 5);
        draw((3, 3) -- (3, 2) -- (4, 2) -- (4, 3) -- cycle);

        pair s = (10, 0);
        backperimeter(7, 5, s);
        frontperimeter(7, 5, s);
        int[][] outputs1 = new int[][] {
            {4, 0, 3, 0, 4, 1, 1},
            {4, 3, 0, 2, 4, 0},
            {4, 3, 1},
            {4},
        };
        int[][] outputs2 = new int[][] {
            {4, 0, 3, 0, 4, 1, 1},
            {4, 3, 0, 2, 4, 0, 0},
            {4, 3, 1, 2, 3, 1, 0},
            {4, 4, 0, 2, 4, 1, 0},
            {3, 3, 0, 4, 1, 4, 0}
        };
        for (int i = 0; i < outputs1.length; ++i) {
            for (int j = 0; j < outputs1[i].length; ++j) {
                label(scale(0.6)*("$"+(string(outputs1[i][j]))+"$"), (j, 4-i)+ (1/2, 1/2));
            }
        }
        for (int i = 0; i < outputs2.length; ++i) {
            for (int j = 0; j < outputs2[i].length; ++j) {
                label(scale(0.6)*("$"+(string(outputs2[i][j]))+"$"), shift(s)*(j, 4-i)+ (1/2, 1/2));
            }
        }

        pair s = (25, 3);
        backframe(5, 7, 4, s);
        ideal(outputs2, s);
        frontframe(5, 7, 4, s);
    \end{diagram}
    \caption{A sample function $g: I \to \{0, \dots, n\}$ (that is not order-preserving) and its corresponding $\tilde{g}: I \to \{0, \dots, n\}$ for $(\ell_1, \ell_2, n) = (5, 7, 4)$. The corresponding subset of $[5] \times [7] \times [4]$ is also included.}
    \label{fig:scasymptoticcorrespondence}
\end{figure}

For the next asymptotic result, the following crude bound will be needed.

\begin{lemma}\label{lemma:idealupperbound}
    Let $n$ and $d$ be positive integers with $d \geq 2$. Then there are at most $4^{n^{d-1}}$ ideals (not necessarily self-complementary) in $[n]^d$.
\end{lemma}

\begin{proof}
    Induct on $d$. For the base case of $d = 2$, there are \smash{$\binom{2n}{n} \leq 4^n$} such ideals. For the inductive step, observe that each layer of an ideal in $[n]^d$ is an ideal in $[n]^{d-1}$. There are $d$ such layers and at most $4^{n^{d-2}}$ choices for each layer by the inductive step, so there are at most $\bigl(4^{n^{d-2}}\bigr)^n = 4^{n^{d-1}}$ ideals in $[n]^d$.
\end{proof}

\begin{theorem}\label{theorem:scvertex3}
    Let $d$ be a fixed positive integer. As $n$ tends to infinity, we have
    \[\log v(G(\underbrace{n, \dots, n}_{\text{$d$ $n$'s}})) = \Theta(n^{d-1}).\]
    (The asymptotic runs over all even positive integers $n$).
\end{theorem}

\begin{proof}
    The upper bound is given by \Cref{lemma:idealupperbound}. For the lower bound, it suffices to construct $2^{\Theta(n^{d-1})}$ self-complementary ideals of $[n]^d$. 
    
    Consider the set
    \[S = \left\{(a_1, \dots, a_d) \in [n^d] \left\vert\medspace
    \begin{aligned}
        a_d &\leq \tfrac{1}{2}n \text{ and } a_1 + \dots + a_d \leq d(\tfrac{1}{2}n+1)-1, \text{ or} \\ 
        a_d &> \tfrac{1}{2}n \text{ and } a_1 + \dots + a_d \leq d(\tfrac{1}{2}n)+1
    \end{aligned}\right.\right\}.\]
    An example of this set is given in \cref{fig:maximums}.

    \begin{figure}[htbp]
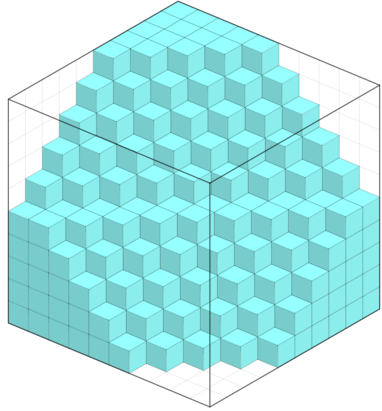

        \centering
        \begin{diagram}[0.3]
            int n = 10;
            backframe(n, n, n);
            for (int x = 0; x < n; ++x) {
                for (int y = 0; y < n; ++y) {
                    for (int z = 0; z < n; ++z) {
                        if (x + y + z <= 3*n/2-1-floor(2*z/n)) {
                            box(x, y, z, 1, 1, 1);
                        }
                    }
                }
            }
            frontframe(n, n, n);
        \end{diagram}
        \caption{$S$ for the case $n=10$, $d=3$.}
        \label{fig:maximums}
    \end{figure}

    Let $M$ be the set of maximal elements in $S$; observe that $\varphi(m) \in M$ for each $m \in M$, where $\varphi$ is the self-dual involution on $[n]^d$. Therefore, every subset $M'\subset M$ such that $M = M' \sqcup \varphi(M')$ gives rise to a self-complementary ideal $S \setminus M'$ of $[n]^d$.
    Bounding the size of $\frac{1}{2}\abs{M}$ gives
    \[\tfrac{1}{2} \abs{M} \geq \lvert\{(b_1, \dots, b_d) \in \bbZ_{\geq_0}^d \mid b_1 + \dots + b_d = \tfrac{1}{2}n-1\}\rvert = \binom{\tfrac{1}{2}n-1+d-1}{d-1} = \Theta(n^{d-1}),\]
    since a suitable translation of the set $\{(b_1, \dots, b_d) \in \bbZ_{\geq_0}^d \mid b_1 + \dots + b_d = \tfrac{1}{2}n-1\}$ is a subset of $M$.
    This gives
    $2^{\abs{M}/2} \geq 2^{\Theta(n^{d-1})}$
    such self-complementary ideals, as desired.
\end{proof}

\begin{remark}
    For $d=2$, the asymptotic is exactly
    \[v(G(n, n)) = \binom{n}{\frac{1}{2}n} \sim C\frac{2^n}{\sqrt{n}}.\]
    for $C=\frac{2}{\sqrt{2\pi}} \approx 0.798$. The logarithmic asymptotic is thus $\log v(G(n, n)) \sim \log(2) n \approx 0.693n$. For $d=3$, the asymptotic is exactly
    \[v(G(n, n, n)) = \left(\prod_{i_1=1}^{n/2}\prod_{i_2=1}^{n/2}\prod_{i_3=1}^{n/2} \frac{i_1+i_2+i_3-1}{i_1+i_2+i_3-2}\right)^2 \sim Cn^{-1/6}e^{9n}\left(\frac{9\sqrt[4]{3}}{8}\right)^{n^2}\]
    for $C = 2^{1/3}3^{-1/6}e^{2 \zeta'(-1)-18}$. The logarithmic asymptotic is thus $\log v(G(n, n, n)) \sim \log(9 \sqrt[4]{3}/8)n^2 \approx 0.392 n^2$. 
    For $d \geq 4$, the exact logarithmic asymptotic appears to be open.
\end{remark}

\begin{conjecture}\label{conjecture:scvertex4}
    Let $\ell$ be an fixed even positive integer. As $n$ tends to infinity, we have
    \[\log v(G(\underbrace{\ell, \dots, \ell}_{\text{$n$ $\ell$'s}})) \sim C\frac{(2\ell)^n }{\sqrt{n}},\]
    where $C=\frac{\log 2}{\sqrt{2\pi(2\ell-1)(2\ell+1)/3}}$.
\end{conjecture}

This is a generalization of \cite{brouwer2013counting}, which proves this asymptotic for $n=2$. A key result used is \cite{kleitman1969dedekind}, which bounds the total number of ideals in $[2]^n$ as $n$ tends to infinity.

\subsection{Diameter}
This section will compute the exact value of the diameter of $G(\ell_1, \dots, \ell_d)$ for every sequence of positive integers $\ell_1, \dots, \ell_d$.

\begin{lemma}\label{lemma:inequality}
    Let $n$ be a positive integer, and let $x_1 \geq \dots \geq x_n$ and $y_1 \geq \dots \geq y_n$ be nonincreasing sequences of real numbers. Then
    \[\frac{1}{n} \sum_{i=1}^n x_iy_i \ge \left(\frac{1}{n} \sum_{i=1}^n x_i\right) \left(\frac{1}{n} \sum_{i=1}^n y_i\right).\]
\end{lemma}

\begin{proof}
    This is an application of the rearrangement inequality. Alternatively, observe that
    \[\frac{1}{n^2} \sum_{1 \leq i < j \leq n} (x_i - x_j)(y_i - y_j)\ge 0\]
    is the difference between the left hand side and the right hand side of the original inequality.
\end{proof}

Now, given a poset $P$ and an ideal $I \subseteq P$, let $\mu(I)$ denote the \emph{density} $\mu(I) = \frac{\abs{I}}{\abs{P}}$. The following lemma is effectively a ``correlation inequality'' and is a generalization of the Harris-Kleitman inequality \cite{harris1960lower, kleitman1966families}.

\begin{lemma}\label{lemma:idealcorrelation}
    Let $\ell_1, \dots, \ell_d$ be a sequence of positive integers, and let $I$ and $J$ be ideals of $[\ell_1] \times \dots \times [\ell_d]$. Then
    $\mu(I \cap J) \geq \mu(I) \mu(J)$.
\end{lemma}

\begin{proof}
    Induct on $d$. The base case of $d=1$ follows from the inequality $\frac{\min(\abs{I}, \abs{J})}{\ell_1}\ge \frac{\abs{I}}{\ell_1}\cdot \frac{\abs{J}}{\ell_1}$, since $\abs{I\cap J} = \min(\abs{I}, \abs{J})$ in this case.

    For the inductive step, let $I$ and $J$ be ideals of $[\ell_1] \times \dots \times [\ell_d]$, and for each index $i\in \{1, \dots, \ell_d\}$ define 
    \begin{align*}
        I_i &= \{(a_1, \dots, a_{d-1}) \in [\ell_1] \times \dots \times [\ell_{d-1}] \mid (a_1, \dots, a_{d-1}, i) \in I\}, \\
        J_i &= \{(a_1, \dots, a_{d-1}) \in [\ell_1] \times \dots \times [\ell_{d-1}] \mid (a_1, \dots, a_{d-1}, i) \in J\}.
    \end{align*}
    Observe that both $I_i$ and $J_i$ are ideals of $[\ell_1] \times \dots \times [\ell_{d-1}]$ for each index $i\in \{1, \dots, \ell_d\}$, and that $\mu(I_1) \geq \dots \geq \mu(I_{\ell_d})$ and $\mu(J_1) \geq \dots \geq \mu(J_{\ell_d})$.
    By the inductive hypothesis and \cref{lemma:inequality}, we have
    \begin{align*}
        \mu(I \cap J) = \frac{1}{\ell_d} \sum_{i=1}^{\ell_d} \mu(I_i \cap J_i) &\geq \frac{1}{\ell_d}\sum_{i=1}^{\ell_d} \mu(I_i) \mu(J_i) \\
        &\geq \left(\frac{1}{\ell_d}\sum_{i=1}^{\ell_d} \mu(I_i)\right) \left(\frac{1}{\ell_d}\sum_{i=1}^{\ell_d} \mu(J_i)\right) = \mu(I)\mu(J),
    \end{align*}
    completing the induction.
\end{proof}

\begin{corollary}\label{corollary:scidealcorrelation1}
    Let $\ell_1, \dots, \ell_d$ be a sequence of positive integers with $\ell_1\cdots\ell_d$ even. Then for any self-complementary ideals $I, J \subset [\ell_1] \times \dots \times [\ell_d]$, we have
    $\abs{I \cap J} \geq \frac{1}{4}\ell_1 \cdots \ell_d.$
\end{corollary}

\begin{proof}
    Since $\mu(I) = \mu(J) = \frac{1}{2}$, we have
    \[\frac{\abs{I \cap J}}{\ell_1 \cdots \ell_d} = \mu(I \cap J) \geq \mu(I) d(J) = \frac{1}{4}\]
    by \cref{lemma:idealcorrelation}.
\end{proof}

\begin{corollary}\label{corollary:scidealcorrelation2}
    Let $\ell_1, \dots, \ell_d$ be a sequence of positive integers, all of which are odd except for $\ell_k$. Then for any self-complementary ideals $I, J\subset [\ell_1] \times \dots \times [\ell_d]$, we have
    $\abs{I \cap J} \geq \frac{1}{4}(\ell_1 \cdots \ell_d + \ell_k)$.
\end{corollary}

\begin{proof}
    Induct on $d$. The base case of $d=1$ is true because the flip graph on $[\ell_1]$ has only one self-complementary ideal: the ideal with $\frac{1}{2}\ell_1$ elements.

    For the inductive step, suppose without loss of generality that $\ell_d$ is odd. Consider the sets 
    \begin{align*}
        I' &= \{(a_1, \dots, a_d)\in I\mid a_d = \tfrac{1}{2}(\ell_d + 1)\}\\
        J' &= \{(a_1, \dots, a_d)\in J\mid a_d = \tfrac{1}{2}(\ell_d + 1)\}.
    \end{align*} 
    and
    \begin{align*}
        I'' &= \left\{(a_1, \dots, a_{d-1}) \in [\ell_1] \times \dots \times [\ell_{d-1}] \mid (a_1, \dots, a_{d-1}, \tfrac{1}{2}(\ell_d+1)) \in I' \right\} \\
        J'' &= \left\{(a_1, \dots, a_{d-1}) \in [\ell_1] \times \dots \times [\ell_{d-1}] \mid (a_1, \dots, a_{d-1}, \tfrac{1}{2}(\ell_d+1)) \in J' \right\}.
    \end{align*}
    Observe that $I''$ and $J''$ are self-complementary ideals in $[\ell_1] \times \dots \times [\ell_{d-1}]$, since the involution on the self-dual poset $[\ell_1] \times \dots \times [\ell_d]$ maps elements with last coordinate $\frac{1}{2}(\ell_d+1)$ to elements with last coordinate $\frac{1}{2}(\ell_d+1)$. Similarly, $I \setminus I'$ and $J \setminus J'$ are self-complementary ideals in $[\ell_1] \times \dots \times [\ell_{d-1}] \times ([\ell_d] \setminus \{\frac{1}{2}(\ell_d+1)\})$. Therefore,
    \begin{align*}
        \abs{I \cap J} &= \abs{I'' \cap J''} + \abs{(I \setminus I') \cap (J \setminus J')} \\
        &\geq \tfrac{1}{4}(\ell_1 \cdots \ell_{d-1} + \ell_k) + \tfrac{1}{4}(\ell_1 \cdots \ell_{d-1} (\ell_d-1)) \\
        &= \tfrac{1}{4}(\ell_1 \cdots \ell_d + \ell_k)
    \end{align*}
    by the inductive hypothesis and \cref{corollary:scidealcorrelation1}.
\end{proof}

\begin{theorem}\label{theorem:scdiameter}
    Let $\ell_1, \dots, \ell_d$ be a sequence of positive integers, and let $V$ denote their product. The diameter of the flip graph on self-complementary ideals of $[\ell_1] \times \dots \times [\ell_d]$ is
    \[
        \begin{cases}
            0 &\text{if all of $\ell_1, \dots, \ell_d$ are odd,} \\
            \frac{1}{4}V & \text{if at least two of $\ell_1, \dots, \ell_d$ are even, and} \\
            \frac{1}{4}(V - \ell_k) & \text{if $\ell_k$ is even and the rest are odd.}
        \end{cases}
    \]
\end{theorem}

\begin{proof}
    If all of $\ell_1, \dots, \ell_d$ are odd, then no self-complementary ideals exist. Thus, the flip graph is empty, and the diameter is zero.
    
    Now, suppose that at least two of $\ell_1, \dots, \ell_d$ are even.
    \begin{itemize}
        \item To prove that the distance between any two self-complementary ideals is at most $\frac{1}{4}V$, let $I$ and $J$ be self-complementary ideals in $[\ell_1] \times \dots \times [\ell_d]$. By \cref{lemma:scdistance} and \cref{corollary:scidealcorrelation1}, the distance between $I$ and $J$ is
        \[\abs{I \setminus J} = \abs{I} - \abs{I \cap J} \leq \tfrac{1}{2}V - \tfrac{1}{4}V = \tfrac{1}{4}V.\]
        \item Suppose that $i$ and $j$ are distinct indices for which both $\ell_i$ and $\ell_j$ are even. Then the self-complementary ideals
        \begin{align*}
            I &= \{(a_1, \dots, a_d) \in [\ell_1] \times \dots \times [\ell_d] \mid a_i \leq \tfrac{1}{2}\ell_i\} \\
            J &= \{(a_1, \dots, a_d) \in [\ell_1] \times \dots \times [\ell_d] \mid a_j \leq \tfrac{1}{2}\ell_j\}
        \end{align*}
        (see \cref{fig:scdiameter1} for an example) satisfy
        \[\abs{I \setminus J} = \abs{\left\{(a_1, \dots, a_d) \in [\ell_1] \times \dots \times [\ell_d] \mid \text{$a_i \leq \tfrac{1}{2}\ell_i$ and $a_j > \tfrac{1}{2}\ell_j$}\right\}} = \tfrac{1}{4}V.\]
        By \cref{lemma:scdistance}, the distance between $I$ and $J$ is $\frac{1}{4}V$.
    \end{itemize}
    Therefore the diameter of the flip graph of $[\ell_1] \times \dots \times [\ell_d]$ is $\frac{1}{4}V$ when at least two of $\ell_1, \dots, \ell_d$ are even.
    
    Finally, suppose that $\ell_k$ is even and the rest are odd. If $d=1$, the result follows because the diameter of a one-vertex graph is zero; hence, suppose $d \geq 2$.
    \begin{itemize}
        \item To prove that the distance between any two self-complementary ideals is at most $\frac{1}{4}(V-\ell_k)$, let $I$ and $J$ be self-complementary ideals in $[\ell_1] \times \dots \times [\ell_d]$. By \cref{lemma:scdistance} and \cref{corollary:scidealcorrelation2}, the distance between $I$ and $J$ is
        \[\abs{I \setminus J} = \abs{I} - \abs{I \cap J} \leq \tfrac{1}{2}V - \tfrac{1}{4}(V+\ell_k) = \tfrac{1}{4}(V-\ell_k).\]
        \item Without loss of generality, assume $\ell_d$ is even, and let $I'$ be a self-complementary ideal of $[\ell_1] \times \dots \times [\ell_{d-1}] \setminus \left\{\left(\frac{1}{2}(\ell_1+1), \dots, \frac{1}{2}(\ell_{d-1}+1)\right)\right\}$. Then the self-complementary ideals
        \begin{align*}
            I &= \left(I' \times [\ell_d]\right) \cup \left\{\left(\tfrac{1}{2}(\ell_1+1), \dots, \tfrac{1}{2}(\ell_{d-1}+1), i\right) \mid i \in \{1, \dots, \tfrac{1}{2} \ell_d\}\right\} \\
            J &= \{(a_1, \dots, a_d) \in [\ell_1] \times \dots \times [\ell_d] \mid a_d \leq \tfrac{1}{2} \ell_d\}
        \end{align*}
        (see \cref{fig:scdiameter2} for an example) satisfy
        \begin{align*}
            \abs{I \setminus J} = \abs{I' \times \{ \tfrac{1}{2}\ell_{d} + 1, \dots, \ell_d\}} = \tfrac{1}{2}(\ell_1 \dots \ell_{d-1}-1) \cdot \tfrac{1}{2}\ell_d = \tfrac{1}{4}(V - \ell_d).
        \end{align*}
        By \cref{lemma:scdistance}, the distance between $I$ and $J$ is $\frac{1}{4}(V-\ell_d)$. \hfill\mbox{\qedhere}
    \end{itemize}
\end{proof}

\begin{figure}[htbp]
    \centering
    \begin{diagram}[0.5]
        pair s = (0, 0);
        backframe(5, 6, 4, s);
        ideal(new int[][] {{4, 4, 4, 0, 0, 0},
                           {4, 4, 4, 0, 0, 0}, 
                           {4, 4, 4, 0, 0, 0},
                           {4, 4, 4, 0, 0, 0},
                           {4, 4, 4, 0, 0, 0}}, s);
        frontframe(5, 6, 4, s);
        pair s = (15, 0);
        backframe(5, 6, 4, s);
        ideal(new int[][] {{2, 2, 2, 2, 2, 2},
                           {2, 2, 2, 2, 2, 2}, 
                           {2, 2, 2, 2, 2, 2},
                           {2, 2, 2, 2, 2, 2},
                           {2, 2, 2, 2, 2, 2}}, s);
        frontframe(5, 6, 4, s);
    \end{diagram}
    \caption{An example of $I$ and $J$ for the case $(\ell_1, \ell_2, \ell_3) = (5, 6, 4)$.}
    \label{fig:scdiameter1}
\end{figure}

\begin{figure}[htbp]
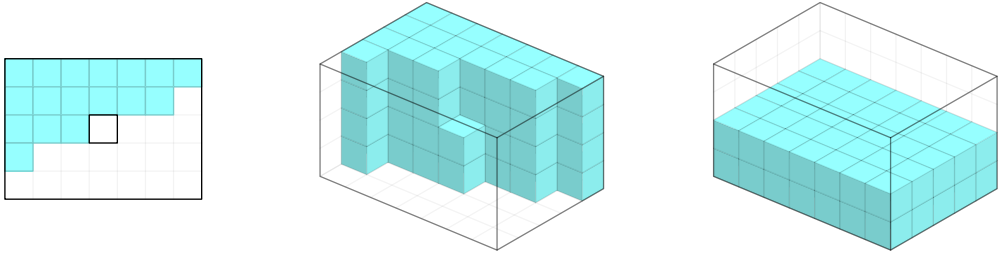

    \centering
    \begin{diagram}[0.8]
        backperimeter(7, 5); 
        rectangle(0, 4, 1, 1);
        rectangle(1, 4, 1, 1);
        rectangle(2, 4, 1, 1);
        rectangle(3, 4, 1, 1);
        rectangle(4, 4, 1, 1);
        rectangle(5, 4, 1, 1);
        rectangle(6, 4, 1, 1);
        rectangle(0, 3, 1, 1);
        rectangle(1, 3, 1, 1);
        rectangle(2, 3, 1, 1);
        rectangle(3, 3, 1, 1);
        rectangle(4, 3, 1, 1);
        rectangle(5, 3, 1, 1);
        rectangle(0, 2, 1, 1);
        rectangle(1, 2, 1, 1);
        rectangle(2, 2, 1, 1);
        rectangle(0, 1, 1, 1);
        frontperimeter(7, 5);
        draw((3, 2) -- (3, 3) -- (4, 3) -- (4, 2) -- cycle);
        pair s = (15, 3);
        backframe(5, 7, 4, s);
        ideal(new int[][] {{4, 4, 4, 4, 4, 4, 4},
                           {4, 4, 4, 4, 4, 4, 0}, 
                           {4, 4, 4, 2, 0, 0, 0},
                           {4, 0, 0, 0, 0, 0, 0},
                           {0, 0, 0, 0, 0, 0, 0}}, s);
        frontframe(5, 7, 4, s);
        pair s = (29, 3);
        backframe(5, 7, 4, s);
        ideal(new int[][] {{2, 2, 2, 2, 2, 2, 2},
                           {2, 2, 2, 2, 2, 2, 2}, 
                           {2, 2, 2, 2, 2, 2, 2},
                           {2, 2, 2, 2, 2, 2, 2},
                           {2, 2, 2, 2, 2, 2, 2}}, s);
        frontframe(5, 7, 4, s);
    \end{diagram}
    \caption{An example of $I'$, $I$, and $J$ for the case $(\ell_1, \ell_2, \ell_3) = (5, 7, 4)$.}
    \label{fig:scdiameter2}
\end{figure}

\subsection{Radius}
This section will compute the radius of $G(\ell_1, \dots, \ell_d)$ for certain sequences of positive integers $\ell_1, \dots, \ell_d$. The answer is known exactly in the specific case when both of the following conditions hold:
\begin{itemize}
    \item all of $\ell_1, \dots, \ell_d$ are even, and
    \item $d$ is odd or some $\ell_i$ is divisible by 4.
\end{itemize}

Some of these results are generalizations of \cite{meyerowitz1995maximal}, which gives a tight radius bound for $[2]^d$ when $d$ is odd.

\begin{lemma}\label{lemma:scradiuslowerbound1}
    Let $d$ be a positive integer. Let $\ell_1, \dots, \ell_d$ be a sequence of even positive integers, and let $V$ denote their product. Then the radius of $G(\ell_1, \dots, \ell_d)$ is at least
    \[{\left(\frac{1}{4} - \frac{1}{2^{d+1}}\binom{d-1}{\floor{\frac{1}{2}(d-1)}}\right)V}.\]
\end{lemma}

\begin{proof}
    Let $r$ denote the radius of $G(\ell_1, \dots, \ell_d)$, and let $C$ be any self-complementary ideal of $[\ell_1] \times \dots \times [\ell_d]$ contained in the center of $G(\ell_1, \dots, \ell_d)$. Additionally, for each index $i \in \{1, \dots, d\}$ let $I_i$ be the self-complementary ideal
    \[I_i = \{(a_1, \dots, a_d) \in [\ell_1] \times \dots \times [\ell_d] \mid a_i \leq \tfrac{1}{2}\ell_i\}.\]
    Then
    \[r \geq \max_{i \in \{1, \dots, d\}} \dist(C, I_i) \geq \frac{1}{d} \sum_{i=1}^d \dist(C, I_i) = \frac{1}{d} \sum_{i=1}^d \abs{C \setminus I_i} = \frac{1}{2}V - \frac{1}{d} \sum_{i=1}^d \abs{C \cap I_i}.\]
    Now, for each every tuple $(t_1, \dots, t_d) \in \{0, 1\}^d$, define
    \[c_{(t_1, \dots, t_d)} = \abs{\{(a_1, \dots, a_d) \in C \mid \bbmone_{a_1 > \ell_1/2} = t_1, \dots, \bbmone_{a_d > \ell_d/2} = t_d\}}.\]
    Since $C$ is self-complementary, we have
    \[c_{(t_1, \dots, t_d)} + c_{(1-t_1, \dots, 1-t_d)} = \frac{V}{2^d}\]
    for every tuple $(t_1, \dots, t_d) \in \{0, 1\}^d$.
    Using this fact gives
    \[\sum_{i=1}^d \abs{C \cap I_i} = \sum_{i=1}^d \sum_{\substack{(t_1, \dots, t_d) \in \{0, 1\}^d \\ t_i = 0}} c_{(t_1, \dots, t_d)}
    = \sum_{s = 0}^d \sum_{\substack{(t_1, \dots, t_d) \in \{0, 1\}^d \\ t_1 + \dots + t_d = s}} (d-s)c_{(t_1, \dots, t_d)}.\]
    Using the equalities $c_{(t_1, \dots, t_d)} + c_{(1-t_1, \dots, 1-t_d)} = \frac{V}{2^d}$ for all $(t_1, \dots, t_d) \in \{0, 1\}^d$ and the inequalities $c_{(t_1, \dots, t_d)} \leq \frac{V}{2^d}$ for $t_1+\dots+t_d < \frac{1}{2}d$ gives a bound of 
    \begin{align*}
        \sum_{i=1}^d \abs{C \cap I_i}
        &\leq 
        \begin{cases}
            d\left(\frac{1}{4} + \frac{1}{2^{d+1}} \binom{d-1}{(d-1)/2}\right) V & \text{if $d$ is odd} \\
            d\left(\frac{1}{4} + \frac{1}{2^{d+1}} \binom{d-1}{(d-2)/2}\right) V & \text{if $d$ is even}
        \end{cases}\\
        &= d\left(\frac{1}{4} + \frac{1}{2^{d+1}} \binom{d-1}{\floor{\frac{1}{2}(d-1)}}\right) V
    \end{align*}
    after collecting the binomial sums.
    Therefore,
    \[r \geq \frac{1}{2}V - \frac{1}{d} \sum_{i=1}^d \abs{C \cap I_i} \geq \frac{1}{2}V - \left(\frac{1}{4} + \frac{1}{2^{d+1}} \binom{d-1}{\floor{\frac{1}{2}(d-1)}}\right) V = \left(\frac{1}{4} - \frac{1}{2^{d+1}} \binom{d-1}{\floor{\frac{1}{2}(d-1)}} \right)V\]
    for all positive integers $d$.
\end{proof}

\begin{remark}
    Equality in the above lemma can only occur when the inequality $c_{(t_1, \dots, t_d)} \leq \frac{V}{2^d}$ is tight for all $(t_1, \dots, t_d) \in \{0, 1\}^d$ satisfying $t_1 + \dots + t_d < \frac{1}{2}d$. An example of the equality case for the poset $[6] \times [8] \times [4]$ is given in \cref{fig:sccenter}.
\end{remark}

The next theorem, known as the Erd\H{o}s-Ko-Rado theorem, is a famous result bounding the size of an intersecting family of $r$-element sets. It will give a tight bound for the radius in the case when all of the side lengths $\ell_1, \dots, \ell_d$ are even and the dimension $d$ is odd. 

\begin{theorem}[Erd\H os-Ko-Rado]\label{theorem:erdoskorado}
    Let $n$ and $r$ be nonnegative integers with $n \geq 2r$. Let $S$ be an $n$-element set, and let $\mcF$ be a set of $r$-element subsets of $S$ for which every pair of sets in $\mcF$ share an element. Then $\abs{\mcF} \leq \binom{n-1}{r-1}$. 
\end{theorem}

\begin{proof}
    A beautiful proof by Katona is given in \cite{katona1972simple}.
\end{proof}

\begin{lemma}\label{lemma:scradiusupperbound1}
    Let $d$ be an odd positive integer, and let $\ell_1, \dots, \ell_d$ be a sequence of even positive integers with product $V$. Then the radius of the flip graph on $[\ell_1] \times \dots \times [\ell_d]$ is at most
    \[\left(\frac{1}{4} - \frac{1}{2^{d+1}} \binom{d-1}{\frac{1}{2}(d-1)}\right)V.\]
    Additionally, the self-complementary ideal
    \[\{(a_1, \dots, a_d) \in [\ell_1] \times \dots \times [\ell_d] \mid \bbmone_{a_1 \leq \ell_1/2} + \dots + \bbmone_{a_d \leq \ell_d/2} > \tfrac{1}{2}d\}\]
    achieves this bound.
\end{lemma}

\begin{proof}
    It suffices to prove the result for $\ell_1 = \dots = \ell_d = 2$, since $[\ell_1] \times \dots \times [\ell_d]$ can be partitioned into $\frac{V}{2^d}$ copies of $[2]^d$, where the center of symmetry of each copy of $[2]^d$ is the same as the center of symmetry of $[\ell_1] \times \dots \times [\ell_d]$. An example of this partition can be found in \cref{fig:cspartition1}.
    \begin{figure}[htbp]
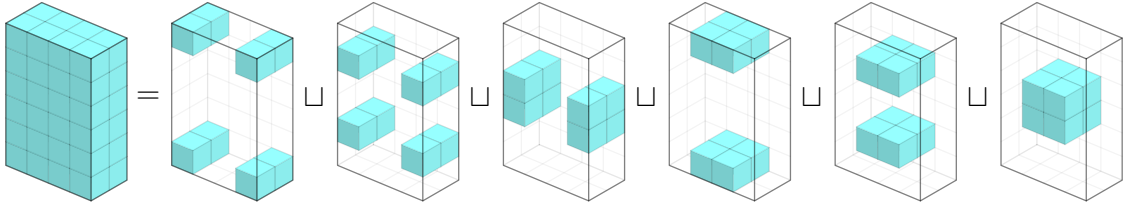

        \centering
        \begin{diagram}[0.9]
            backframe(2, 4, 6);
            ideal(new int[][] {{6, 6, 6, 6}, {6, 6, 6, 6}});
            frontframe(2, 4, 6);
            int[][] blocks = new int[][] {
                {0, 0, 0},
                {0, 0, 1},
                {0, 0, 2},
                {0, 1, 0},
                {0, 1, 1},
                {0, 1, 2}
            };
            for (int i = 1; i <= 6; ++i) {
                pair t = (7*i, 0);
                if (i == 1) {
                    label("$=$", (-2.5, 2)+t);
                } else {
                    label("$\sqcup$", (-2.5, 2)+t);
                }
                int x = blocks[i-1][0];
                int y = blocks[i-1][1];
                int z = blocks[i-1][2];
                backframe(2, 4, 6, t);
                box(x, y, z, 1, 1, 1, t);
                box(x, y, 5-z, 1, 1, 1, t);
                box(x, 3-y, z, 1, 1, 1, t);
                box(x, 3-y, 5-z, 1, 1, 1, t);
                box(1-x, y, z, 1, 1, 1, t);
                box(1-x, y, 5-z, 1, 1, 1, t);
                box(1-x, 3-y, z, 1, 1, 1, t);
                box(1-x, 3-y, 5-z, 1, 1, 1, t);
                frontframe(2, 4, 6, t);
            }
        \end{diagram}
        \caption{The partition of $[2] \times [4] \times [6]$ into six copies of $[2]^3$.}
        \label{fig:cspartition1}
    \end{figure}
    
    Let
    \[C = \{(a_1, \dots, a_d) \in [2]^d \mid \bbmone_{a_1 = 1} + \dots + \bbmone_{a_d = 1} > \tfrac{1}{2}d\}.\]
    be a self-complementary ideal in $[2]^d$, and let $I$ be any self-complementary ideal in $[2]^d$. By \cref{lemma:scdistance}, it suffices to show that
    \[\abs{I \setminus C} \leq \left(\frac{1}{4} - \frac{1}{2^{d+1}} \binom{d-1}{\frac{1}{2}(d-1)}\right)2^d.\]
    Observe that if $(a_1, \dots, a_d)$ and $(b_1, \dots, b_d)$ are elements in $I$, then there is some index $i$ for which $a_i = b_i = 1$, since $I$ is self-complementary. This means $I \setminus C$ can be interpreted as an intersecting family of subsets of $\{1, \dots, d\}$ by mapping each tuple $(a_1, \dots, a_d)$ to the set $\{i \in [d] \mid a_i = 1\}$.
    By \cref{theorem:erdoskorado}, we have 
    \begin{align*}
    \abs{I \setminus C} &\leq \sum_{i = 0}^{(d-3)/2} \binom{d-1}{i} = \frac{1}{2}\left(2^{d-1} - \binom{d-1}{\frac{1}{2}(d-1)}\right) = \left(\frac{1}{4} - \frac{1}{2^{d+1}} \binom{d-1}{\floor{\frac{1}{2}(d-1)}}\right)2^d,
    \end{align*}
    since every tuple in $I \setminus C$ has fewer than $\frac{d}{2}$ 1's. Lastly, the self-complementary ideal \[\{(a_1, \dots, a_d) \in [\ell_1] \times \dots \times [\ell_d] \mid \bbmone_{a_1 \leq \ell_1/2} + \dots + \bbmone_{a_d \leq \ell_d/2} > \tfrac{1}{2}d\}\]
    achieves this bound because it corresponds to the equality cases in \cref{theorem:erdoskorado}.
\end{proof}

An example of a self-complementary ideal in the center of the flip graph of $[6] \times [8] \times [4]$ is depicted in \cref{fig:sccenter}.

\begin{figure}[htbp]
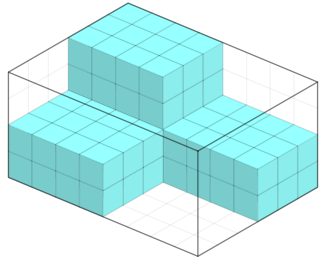

    \centering
    \begin{diagram}[0.25]
        backframe(6, 8, 4);
        ideal(new int[][] {{4, 4, 4, 4, 2, 2, 2, 2},
                           {4, 4, 4, 4, 2, 2, 2, 2},
                           {4, 4, 4, 4, 2, 2, 2, 2},
                           {2, 2, 2, 2, 0, 0, 0, 0},
                           {2, 2, 2, 2, 0, 0, 0, 0},
                           {2, 2, 2, 2, 0, 0, 0, 0}});
        frontframe(6, 8, 4);
    \end{diagram}
    \caption{The unique equality case for the case $(\ell_1, \ell_2, \ell_3) = (6, 8, 4)$.}
    \label{fig:sccenter}
\end{figure}

\begin{lemma}\label{lemma:scradiusupperbound2}
    Let $d$ be a positive integer, and let $\ell_1, \dots, \ell_d$ be a sequence of even positive integers with product $V$. Suppose $\ell_d$ is divisible by 4. Then the radius of $G(\ell_1, \dots, \ell_d)$ is at most
    \[\left(\frac{1}{4} - \frac{1}{2^{d+1}} \binom{d-1}{\floor{\frac{1}{2}(d-1)}}\right)V.\]
    Additionally, if $d$ is odd, then the self-complementary ideal
    \[\{(a_1, \dots, a_d) \in [\ell_1] \times \dots \times [\ell_d] \mid \bbmone_{a_1 \leq \ell_1/2} + \dots + \bbmone_{a_d \leq \ell_d/2} > \tfrac{1}{2}d\}\]
    achieves this bound. Similarly, if $d$ is even, then the self-complementary ideal 
    \[\{(a_1, \dots, a_d) \in [\ell_1] \times \dots \times [\ell_d] \mid \bbmone_{a_1 \leq \ell_1/2} + \dots + \bbmone_{a_{d-1} \leq \ell_{d-1}/2} + \tfrac{1}{2}\bigl\lceil\tfrac{4a_d}{\ell_d}-1\bigr\rceil > \tfrac{1}{2}d\}\]
    achieves this bound.
\end{lemma}

\begin{proof}
    If $d$ is odd, the result follows from \cref{lemma:scradiusupperbound1}, so assume henceforth that $d$ is even and $a_d$ is divisible by 4. It suffices to prove the result for $\ell_1 = \dots = \ell_{d-1} = 2$ and $\ell_d = 4$, since $[\ell_1] \times \dots \times [\ell_d]$ can be partitioned into $\frac{V}{2^{d+1}}$ copies of $[2]^{d-1} \times [4]$, where the center of symmetry of each copy of $[2]^{d-1} \times [4]$ is the same as the center of symmetry of $[\ell_1] \times \dots \times [\ell_d]$, and if $(a_1, \dots, a_{d-1}, a_d)$ lies in a copy, then so does $(a_1, \dots, a_{d-1}, \frac{1}{2}\ell_d+1 - a_d)$, $(a_1, \dots, a_{d-1}, \frac{1}{2}\ell_d + a_d)$, and $(a_1, \dots, a_{d-1}, \ell_d+1-a_d)$.

    \begin{figure}[htbp]
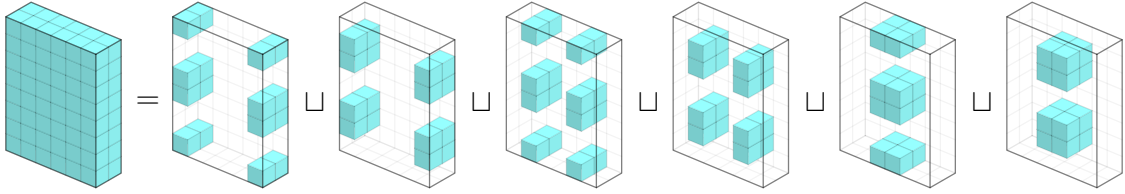

        \centering
        \begin{diagram}[0.9]
            backframe(2, 6, 8);
            ideal(new int[][] {{8, 8, 8, 8, 8, 8}, {8, 8, 8, 8, 8, 8}});
            frontframe(2, 6, 8);
            int[][] blocks = new int[][] {
                {0, 0, 0},
                {0, 0, 1},
                {0, 1, 0},
                {0, 1, 1},
                {0, 2, 0},
                {0, 2, 1}
            };
            for (int i = 1; i <= 6; ++i) {
                pair t = (10*i, 0);
                if (i == 1) {
                    label("$=$", (-3, 2)+t);
                } else {
                    label("$\sqcup$", (-3, 2)+t);
                }
                int x = blocks[i-1][0];
                int y = blocks[i-1][1];
                int z = blocks[i-1][2];
                backframe(2, 6, 8, t);
                box(x, y, z, 1, 1, 1, t);
                box(x, y, 3-z, 1, 1, 1, t);
                box(x, y, 4+z, 1, 1, 1, t);
                box(x, y, 7-z, 1, 1, 1, t);
                box(x, 5-y, z, 1, 1, 1, t);
                box(x, 5-y, 3-z, 1, 1, 1, t);
                box(x, 5-y, 4+z, 1, 1, 1, t);
                box(x, 5-y, 7-z, 1, 1, 1, t);
                box(1-x, y, z, 1, 1, 1, t);
                box(1-x, y, 3-z, 1, 1, 1, t);
                box(1-x, y, 4+z, 1, 1, 1, t);
                box(1-x, y, 7-z, 1, 1, 1, t);
                box(1-x, 5-y, z, 1, 1, 1, t);
                box(1-x, 5-y, 3-z, 1, 1, 1, t);
                box(1-x, 5-y, 4+z, 1, 1, 1, t);
                box(1-x, 5-y, 7-z, 1, 1, 1, t);
                frontframe(2, 6, 8, t);
            }
        \end{diagram}
        \caption{The partition of $[2] \times [6] \times [8]$ into six copies of $[2] \times [2] \times [4]$.}
        \label{fig:cspartition2}
    \end{figure}

    Let
    \[C = \{(a_1, \dots, a_d) \in [2]^{d-1} \times [4] \mid \bbmone_{a_1=1} + \dots + \bbmone_{a_{d-1}=1} + \tfrac{4-a_d}{2} \geq \tfrac{d}{2}\}\]
    be a self-complementary ideal in $[2]^{d-1} \times [4]$, and let $I$ be any self-complementary ideal in $[2]^{d-1} \times [4]$. It suffices to show that
    \[\abs{I \setminus C} \leq \left(\frac{1}{4} - \frac{1}{2^{d+1}} \binom{d-1}{\frac{1}{2}(d-2)}\right)2^{d+1}.\]
    Let $\varphi: [2]^{d-1} \times [4] \to [2]^{d+1}$ be the map given by
    \[\textstyle(a_1, \dots, a_{d-1}, a_d) \mapsto
    \begin{cases}
        (a_1, \dots, a_{d-1}, 1, 1) & \text{if } a_d = 1 \\ 
        (a_1, \dots, a_{d-1}, 1, 2) & \text{if } a_d = 2 \\ 
        (a_1, \dots, a_{d-1}, 2, 1) & \text{if } a_d = 3 \\ 
        (a_1, \dots, a_{d-1}, 2, 2) & \text{if } a_d = 4 \\ 
    \end{cases}.\]
    Observe that $\varphi$ maps self-complementary ideals of $[2]^{d-1} \times [4]$ to self-complementary ideals of $[2]^{d+1}$. Since $\varphi$ maps $C$ to an element in the center of $[2]^{d+1}$, we apply \cref{lemma:scradiusupperbound1} to obtain
    \[\abs{I \setminus C} = \abs{\varphi(I) \setminus \varphi(C)} \leq \left(\frac{1}{4} - \frac{1}{2^{d+2}} \binom{d}{\frac{1}{2}d}\right) 2^{d+1} = \left(\frac{1}{4} - \frac{1}{2^{d+1}} \binom{d-1}{\frac{1}{2}(d-2)}\right)2^{d+1}. \qedhere\]
\end{proof}

\begin{conjecture}\label{conjecture:radiusinevendimension}
    Let $d$ be an even positive integer. Then the radius of the flip graph of $[2]^d$ is at most
    \[\ceil{\left(\frac{1}{4} - \frac{1}{2^{d+1}} \binom{d-1}{\frac{1}{2}(d-2)}\right) 2^d}.\]
    Additionally, there exists a self-complementary ideal achieving this bound containing 
    \[\{(a_1, \dots, a_d) \in [2]^d \mid \bbmone_{a_1=1} + \dots + \bbmone_{a_d=1} > \tfrac{1}{2}d\}\]
    as a subset.
\end{conjecture}

\begin{remark}
    The truth of a specific case (\cref{conjecture:chvatal3}) of Chv\'atal's conjecture (\cref{conjecture:chvatal1}) implies the truth of this conjecture. To see why, let $\mcH$ be a uniform set of subsets of $\{1, \dots, d\}$, satisfying the conditions of \cref{conjecture:chvatal3}.
    Let $\varphi$ be the map from elements of $[2]^d$ to subsets of $\{1, \dots, d\}$ given by
    \[(a_1, \dots, a_d) \mapsto \{i \in [d] \mid a_i = 1\},\]
    and let $C = [2]^d \setminus \{\varphi^{-1}(A) \mid A \in \mcH\}$. Observe that $C$ is an self-complementary ideal of $[2]^d$, since
    \begin{itemize}
        \item for each subset $A \subseteq \{1, \dots, d\}$, exactly one of $A$ or $\{1, \dots, d\} \setminus A$ is contained in $\mcH$, and
        \item $\varphi$ maps the order structure on $[2]^d$ to the reverse inclusion structure on subsets of $\{1, \dots, d\}$ (that is, $a \leq_{[2]^d} b \iff \varphi(a) \supseteq \varphi(b)$), so $C$ is an ideal.
    \end{itemize}
    Now by \cref{lemma:scdistance}, it suffices to show that \[\abs{I \setminus C} \leq \ceil{\left(\frac{1}{4} - \frac{1}{2^{d+1}} \binom{d-1}{\frac{1}{2}(d-2)}\right)2^d}\] for any self-complementary ideal $I\subset [2]^d$. Indeed,
    \[\abs{I \setminus C} = \abs{\varphi(I) \setminus \varphi(C)} = \abs{\varphi(I) \cap \mcH}.\]
    Since $I$ is a self-complementary ideal, every pair of sets in $\varphi(I)$ must share an element. By the hypothesis of the conjecture, we have
    \[\abs{\varphi(I) \cap \mcH} \leq \ceil{\frac{1}{4}\binom{d}{\frac{1}{2}d}} + \sum_{i=1}^{\frac{1}{2}d-1} \binom{d-1}{i-1} = \ceil{\left(\frac{1}{4} - \frac{1}{2^{d+1}} \binom{d-1}{\frac{1}{2}(d-2)}\right)2^d}.\]
    Finally, observe that $C$ contains
    \[\{(a_1, \dots, a_d) \in [2]^d \mid \bbmone_{a_1=1} + \dots + \bbmone_{a_d=1} > \tfrac{1}{2}d\}\]
    as a subset, as desired. 
\end{remark}

\begin{corollary}\label{corollary:scradiusupperbound3}
    Assume \cref{conjecture:radiusinevendimension}. Let $d$ be a positive integer, and let $\ell_1, \dots, \ell_d$ be a sequence of even positive integers with product $V$. Suppose none of $\ell_1, \dots, \ell_d$ are divisible by 4. Then the radius of $G(\ell_1, \dots, \ell_d)$ is at most
    \[\ceil{\left(\frac{1}{4} - \frac{1}{2^{d+1}} \binom{d-1}{\floor{\frac{1}{2}(d-1)}}\right)V}.\]
    Additionally, there exists a self-complementary ideal achieving this bound containing 
    \[\{(a_1, \dots, a_d) \in [\ell_1] \times \dots \times [\ell_d] \mid \bbmone_{a_1 \leq \ell_1/2} + \dots + \bbmone_{a_d \leq \ell_d/2} > \tfrac{1}{2}d\}\]
    as a subset.
\end{corollary}

\begin{proof}
    Let $P$ denote the poset $[\ell_1] \times \dots \times [\ell_d]$. For each index $i \in \{1, \dots, d\}$, define $c_i = \{\tfrac{1}{4}\ell_i+\tfrac{1}{2}, \tfrac{3}{4}\ell_i+\tfrac{1}{2}\}$. Define the following sequence of subsets of $P$:
    \begin{align*}
        P_0 &= c_1 \times c_2 \times \dots \times c_{d-1} \times c_d \\
        P_1 &= \{[\ell_1] \setminus c_1\} \times c_2 \times \dots \times c_{d-1} \times c_d \\
        P_2 &= [\ell_1] \times \{[\ell_2] \setminus c_2\} \times \dots \times c_{d-1} \times c_d \\
         & \vdots \\
        P_{d-1} &= [\ell_1] \times [\ell_2] \times \dots \times \{[\ell_{d-1}] \setminus c_{d-1}\} \times c_d \\
        P_d &= [\ell_1] \times [\ell_2] \times \dots \times [\ell_{d-1}] \times \{[\ell_d] \setminus c_d\}.
    \end{align*}
    Observe that $P = P_0 \sqcup \dots \sqcup P_d$. (An example of this partition is given in \cref{fig:sc2mod4partition}.)
    
    \begin{figure}[htbp]
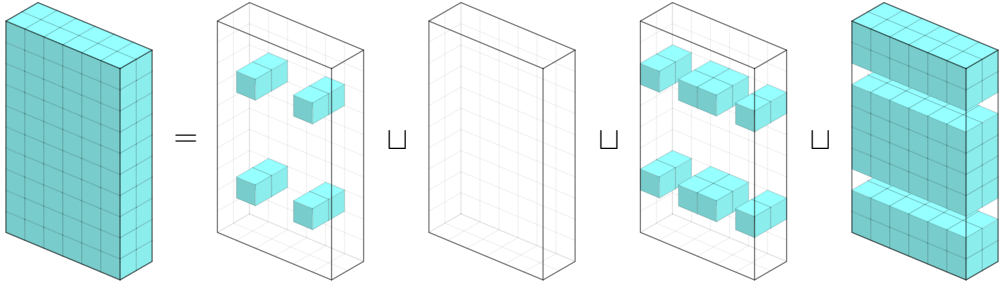

        \centering
        \begin{diagram}[0.8]
            backframe(2, 6, 10);
            ideal(new int[][] {{10, 10, 10, 10, 10, 10},
                               {10, 10, 10, 10, 10, 10}});
            frontframe(2, 6, 10);
            
            pair s = (10, 0);
            backframe(2, 6, 10, s);
            box(0, 1, 2, 1, 1, 1, s);
            box(0, 1, 7, 1, 1, 1, s);
            box(0, 4, 2, 1, 1, 1, s);
            box(0, 4, 7, 1, 1, 1, s);
            box(1, 1, 2, 1, 1, 1, s);
            box(1, 1, 7, 1, 1, 1, s);
            box(1, 4, 2, 1, 1, 1, s);
            box(1, 4, 7, 1, 1, 1, s);
            frontframe(2, 6, 10, s);
    
            pair s = (20, 0);
            backframe(2, 6, 10, s);
            frontframe(2, 6, 10, s);
            
            pair s = (30, 0);
            backframe(2, 6, 10, s);
    
            for (int x = 0; x < 2; ++x) {
                for (int y = 0; y < 6; ++y) {
                    for (int z = 0; z < 10; ++z) {
                        if (z == 2 || z == 7) {
                            if (y != 1 && y != 4) {
                                box(x, y, z, 1, 1, 1, s);
                            }
                        }
                    }
                }
            }
            
            frontframe(2, 6, 10, s);
            pair s = (40, 0);
            backframe(2, 6, 10, s);
            for (int x = 0; x < 2; ++x) {
                for (int y = 0; y < 6; ++y) {
                    for (int z = 0; z < 10; ++z) {
                        if (z != 2 && z != 7) {
                            box(x, y, z, 1, 1, 1, s);
                        }
                    }
                }
            }
            frontframe(2, 6, 10, s);
            label("$=$", (7, 3.5));
            label("$\sqcup$", (17, 3.5));
            label("$\sqcup$", (27, 3.5));
            label("$\sqcup$", (37, 3.5));
        \end{diagram}
        \caption{The partition of the poset $[2] \times [6] \times [10]$.}
        \label{fig:sc2mod4partition}
    \end{figure}
    
    Let $C_0 \subset P_0$ be an ideal satisfying \cref{conjecture:radiusinevendimension}. Similarly, let $C_1 \subset P_1, \dots, C_d \subset P_d$ be ideals satisfying \cref{lemma:scradiusupperbound2}, which exist since each of $P_1, \dots, P_d$ have at least one dimension divisible by 4. Let $C = C_0 \sqcup \cdots \sqcup C_d$.

    $C$ is an ideal of $P$ because the construction given in \cref{lemma:scradiusupperbound2} implies that for every maximal element $a \in C_i$, the set $\{b \in P \mid b \leq a\} \subseteq I$ and $\{b \in P \mid b > a\} \subseteq P \setminus I$. Furthermore, it is self-complementary because $C_0, \dots, C_d$ are all self-complementary. Therefore, for any self-complementary ideal $I \subset P$,
    \begin{align*}
        \dist(C, I) &\leq \dist(C_0, I \cap C_0) + \dots + \dist(C_d, I \cap C_d) \\
        &\leq \ceil{\left(\frac{1}{4} - \frac{1}{2^{d+1}} \binom{d-1}{\frac{1}{2}(d-2)}\right) \abs{P_0}} + \sum_{i=1}^d \left(\frac{1}{4} - \frac{1}{2^{d+1}} \binom{d-1}{\frac{1}{2}(d-2)}\right)\abs{P_i} \\
        &= \ceil{\left(\frac{1}{4} - \frac{1}{2^{d+1}} \binom{d-1}{\frac{1}{2}(d-2)}\right) \left(\abs{P_0} + \dots + \abs{P_d}\right)} \\
        &= \ceil{\left(\frac{1}{4} - \frac{1}{2^{d+1}} \binom{d-1}{\frac{1}{2}(d-2)}\right) V}
    \end{align*}
    by \cref{conjecture:radiusinevendimension} and \cref{lemma:scradiusupperbound2}, since $I \cap C_i$ is an self-complementary ideal of $C_i$ for each $i \in \{0, \dots, d\}$ by construction.
\end{proof}

Combining all previous results gives the following theorem.

\begin{theorem}\label{theorem:scradius}
    Let $d$ be a positive integer, and let $\ell_1, \dots, \ell_d$ be a sequence of even positive integers with product $V$. Suppose that $d$ is odd or that at least one $\ell_i$ is divisible by 4. Then the radius of $G(\ell_1, \dots, \ell_d)$ is
    \[\left(\frac{1}{4} - \frac{1}{2^{d+1}} \binom{d-1}{\floor{\frac{1}{2}(d-1)}}\right)V.\]
    Additionally, if \cref{conjecture:radiusinevendimension} is true and there are no assumptions on $d$ or any $\ell_i$ being divisible by 4, then the radius of $G(\ell_1, \dots, \ell_d)$ is
    \[\ceil{\left(\frac{1}{4} - \frac{1}{2^{d+1}} \binom{d-1}{\floor{\frac{1}{2}(d-1)}}\right)V} = \begin{cases} \left(\frac{1}{4} - \frac{1}{2^{d+1}} \binom{d-1}{\floor{(d-1)/2}}\right)V & \text{if $d$ is not a power of 2} \\ 
    \left(\frac{1}{4} - \frac{1}{2^{d+1}} \binom{d-1}{\floor{(d-1)/2}}\right)V + \frac{1}{2} & \text{if $d$ is a power of 2,} \end{cases}.\]
\end{theorem}

\begin{proof}
    The lower bound in all cases is given by \cref{lemma:scradiuslowerbound1}. The upper bound, assuming $d$ is odd, is given by \cref{lemma:scradiusupperbound1}. The upper bound, assuming one dimension is divisible by 4, is given by \cref{lemma:scradiusupperbound2}. The upper bound, assuming nothing aside from the evenness of $\ell_1, \dots, \ell_d$, is given by \cref{corollary:scradiusupperbound3}, which is true because \cref{conjecture:chvatal3} implies \cref{conjecture:radiusinevendimension}, which implies \cref{corollary:scradiusupperbound3}.
\end{proof}

Now, we turn our attention towards the case when at least one of $\ell_1, \dots, \ell_d$ is odd.

\begin{lemma}\label{lemma:scradiusupperbound4}
    Assume \cref{conjecture:radiusinevendimension}. Given a sequence of positive integers $\ell_1, \dots, \ell_d$, let $r(\ell_1, \dots, \ell_d)$ denote the radius of $G(\ell_1, \dots, \ell_d)$. Suppose $\ell_1, \dots, \ell_{d'}$ are odd, and $\ell_{d'+1}, \dots, \ell_d$ are even. Then for every positive integer $k$, we have
    \[r(\ell_1, \dots, \ell_{d'}) \leq \sum_{S \subseteq \{1, \dots, d\}} r(\{\ell_i-1 \mid i \in S\} \cup \{\ell_{d'+1}, \dots, \ell_d\}).\]
\end{lemma}

\begin{proof}
    Let $P$ denote $[\ell_1]\times\cdots\times [\ell_d]$.
    Without loss of generality, assume $\ell_1, \dots, \ell_{d'}$ are odd, and $\ell_{d'+1}, \dots, \ell_d$ are even. For each index $i \in \{1, \dots, d'\}$, let $m_i = \{\frac{1}{2}(\ell_i+1)\}$. For each subset $S \subseteq \{1, \dots, d'\}$, define
    \[P_S = \left(\begin{cases}m_1 & \text{if $1 \in S$} \\ [\ell_1] \setminus m_1 & \text{if $1 \notin S$} \end{cases}\right) \times \dots \times \left(\begin{cases}m_{d'} & \text{if $d' \in S$} \\ [\ell_{d'}] \setminus m_{d'} & \text{if $d' \notin S$} \end{cases}\right) \times [\ell_{d'+1}] \times \dots \times [\ell_d].\]
    Observe that
    \[[\ell_1] \times \dots \times [\ell_d] = \bigsqcup_{S \subseteq \{1, \dots, d'\}} P_S.\] (An example of this partition is given in \cref{fig:cspartition3}.)

    \begin{figure}[htbp]
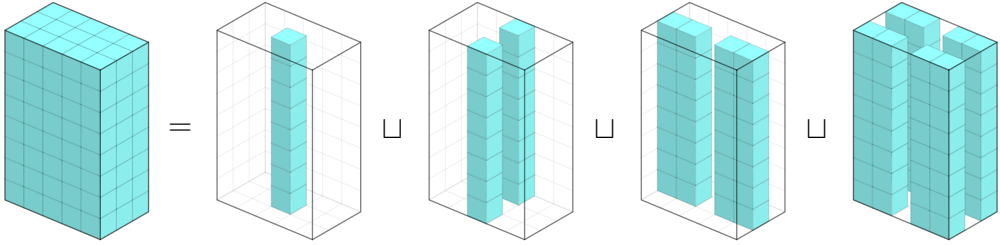

        \centering
        \begin{diagram}[0.8]
            backframe(3, 5, 8);
            ideal(new int[][] {{8, 8, 8, 8, 8}, {8, 8, 8, 8, 8}, {8, 8, 8, 8, 8}});
            frontframe(3, 5, 8);
    
            pair s = (10, 0);
            backframe(3, 5, 8, s);
            for (int x = 0; x < 3; ++x) {
                for (int y = 0; y < 5; ++y) {
                    for (int z = 0; z < 8; ++z) {
                        if (x == 1 && y == 2) {
                            box(x, y, z, 1, 1, 1, s);
                        }
                    }
                }
            }
            frontframe(3, 5, 8, s);
    
            pair s = (20, 0);
            backframe(3, 5, 8, s);
            for (int x = 0; x < 3; ++x) {
                for (int y = 0; y < 5; ++y) {
                    for (int z = 0; z < 8; ++z) {
                        if (x != 1 && y == 2) {
                            box(x, y, z, 1, 1, 1, s);
                        }
                    }
                }
            }
            frontframe(3, 5, 8, s);
    
            pair s = (30, 0);
            backframe(3, 5, 8, s);
            for (int x = 0; x < 3; ++x) {
                for (int y = 0; y < 5; ++y) {
                    for (int z = 0; z < 8; ++z) {
                        if (x == 1 && y != 2) {
                            box(x, y, z, 1, 1, 1, s);
                        }
                    }
                }
            }
            frontframe(3, 5, 8, s);
    
            pair s = (40, 0);
            backframe(3, 5, 8, s);
            for (int x = 0; x < 3; ++x) {
                for (int y = 0; y < 5; ++y) {
                    for (int z = 0; z < 8; ++z) {
                        if (x != 1 && y != 2) {
                            box(x, y, z, 1, 1, 1, s);
                        }
                    }
                }
            }
            frontframe(3, 5, 8, s);
    
            label("$=$", (6, 2));
            label("$\sqcup$", (16, 2));
            label("$\sqcup$", (26, 2));
            label("$\sqcup$", (36, 2));
        \end{diagram}
        \caption{The partition of the poset $[3] \times [5] \times [8]$.}
        \label{fig:cspartition3}
    \end{figure}
    
    For each poset $P_S$, construct an element $C_S$ of the center according to \cref{lemma:scradiusupperbound1}, \cref{lemma:scradiusupperbound2}, or \cref{corollary:scradiusupperbound3}.
    Let $C = \sqcup_{S \subseteq \{1, \dots, d'\}} C_S$. Then $C$ is an ideal because every $C_S$ contains \[\{(a_1, \dots, a_d) \in P_S \mid \bbmone_{a_1 \leq \ell_1/2} + \dots + \bbmone_{a_d \leq \ell_d/2} > \tfrac{1}{2}d\}\]
    as a subset, and dually does not contain any element of \[\{(a_1, \dots, a_d) \in P_S \mid \bbmone_{a_1 \leq \ell_1/2} + \dots + \bbmone_{a_d \leq \ell_d/2} < \tfrac{1}{2}d\}.\]
    Thus for any self-complementary ideal $I \subset P$, we have
    \[\abs{I \setminus P} = \sum_{S \subseteq \{1, \dots, d'\}} (I \cap P_S) \setminus C_S \leq \sum_{S \subseteq \{1, \dots, d\}} r(\{\ell_i-1 \mid i \in S\} \cup \{\ell_{d'+1}, \dots, \ell_d\})\]
    by \cref{lemma:scdistance}, since each $I \cap P_S$ is self-complementary in $P_S$.
\end{proof}

An almost-exact answer is obtainable in the $d=3$ case. In fact, the following is true:
\begin{theorem}\label{theorem:scradius3d}
    Let $d \leq 3$ be a positive integer, and let $\ell_1, \dots, \ell_d$ be positive integers. Let $r(\ell_1, \dots, \ell_d)$ and $d(\ell_1, \dots, \ell_d)$ denote the radius and diameter of $G(\ell_1, \dots, \ell_d)$. Then \[r(\ell_1, \dots, \ell_d) = \ceil{\tfrac{1}{2}d(\ell_1, \dots, \ell_d)},\]
    unless $(\ell_1, \ell_2, \ell_3) \equiv (2, 3, 3) \pmod{4}$ or permutations, in which case it is possible that 
    \[r(\ell_1, \dots, \ell_d) = \tfrac{1}{2}d(\ell_1, \dots, \ell_d)+1.\]
\end{theorem}

\begin{proof}
    To prove the upper bound, we use \cref{theorem:scdiameter} and \cref{lemma:scradiusupperbound4} and split into cases.
    \begin{itemize}
        \item If $\ell_1 \cdots \ell_d$ is odd, then $r(\ell_1, \dots, \ell_d) = d(\ell_1, \dots, \ell_d) = 0$. 
        \item If $d=1$ and $\ell_1$ is even, then $r(\ell_1) = d(\ell_1) = 0$.
        \item If $d=2$ and $\ell_1, \ell_2$ are both even, then
        \[r(\ell_1, \ell_2) = \ceil{(\tfrac{1}{4} - \tfrac{1}{8}) \ell_1 \ell_2} = \ceil{\tfrac{1}{8} \ell_1\ell_2} = \ceil{\tfrac{1}{2} d(\ell_1, \ell_2)}.\]
        \item If $d=2$, $\ell_1$ is even, and $\ell_2 =2k+1$ is odd, then
        \[r(\ell_1, \ell_2) \leq r(\ell_1, 2k) + r(\ell_1) = \ceil{\tfrac{1}{8}\ell_1 \cdot 2k} = \ceil{\tfrac{1}{2} \cdot \tfrac{1}{4}(\ell_1\ell_2 - \ell_1)} = \ceil{\tfrac{1}{2}d(\ell_1, \ell_2)}.\]
        \item If $d=3$ and $\ell_1, \ell_2, \ell_3$ are all even, then
        \[r(\ell_1, \ell_2, \ell_3) = (\tfrac{1}{4} - \tfrac{1}{8}) \ell_1 \ell_2 \ell_3 = \tfrac{1}{8} \ell_1 \ell_2 \ell_3 = \tfrac{1}{2} d(\ell_1, \ell_2, \ell_3).\]
        \item If $d=3$, $\ell_1, \ell_2$ are even, and $\ell_3 = 2k+1$ is odd, then
        \[
            r(\ell_1, \ell_2, \ell_3) \leq r(\ell_1, \ell_2, 2k) + r(\ell_1, \ell_2)
            = \tfrac{1}{8}\ell_1 \ell_2 \cdot 2k + \ceil{\tfrac{1}{8}\ell_1 \ell_2}
            = \ceil{\tfrac{1}{8}\ell_1\ell_2\ell_3} = \ceil{\tfrac{1}{2} d(\ell_1, \ell_2, \ell_3)}.
        \]
        \item If $d=3$, $\ell_1 = 2k_1+1$ and $\ell_2 = 2k_2+1$ are odd, and $\ell_3$ is even, then
        \begin{align*}
            r(\ell_1, \ell_2, \ell_3) &\leq r(2k_1, 2k_2, \ell_3) + r(2k_1, \ell_3) + r(2k_2, \ell_3) + r(\ell_3) \\
            &= \tfrac{1}{8} 2k_1 \cdot 2k_2 \cdot \ell_3 + \ceil{\tfrac{1}{8} \cdot 2k_1 \ell_3} + \ceil{\tfrac{1}{8} \cdot 2k_2 \ell_3} \\
            &= \ceil{\tfrac{1}{8} (2k_1 \cdot 2k_2 + 2k_1 + 2k_2) \ell_3} + \bbmone_{\text{$k_1$, $k_2$, $\ell_3/2$ odd}} \\
            &= \ceil{\tfrac{1}{8} (\ell_1 \ell_2 \ell_3 - \ell_3)} + \bbmone_{\text{$k_1$, $k_2, \ell_3/2$ odd}}.
        \end{align*}
    \end{itemize}
    To prove the lower bound, observe that $r(\ell_1, \dots, \ell_d) \geq \ceil{\frac{1}{2}d(\ell_1, \dots, \ell_d)}$.
\end{proof}

This result illustrates the difficulty of obtaining a ``nice'' formula for the radius when $\ell_1, \dots, \ell_d$ can be any sequence of positive integers.

\subsection{Relation to Chv\'atal's Conjecture}

This section is essentially an addendum to the previous section about radius bounds, since some results from the previous section rely on the truth of conjectures stated here.

The following conjecture, known as Chv\'atal's conjecture, can be thought of as an extension of the Erd\H{o}s-Ko-Rado theorem (\cref{theorem:erdoskorado}) to sets of subsets closed under taking subsets. If proven, it would imply tight radius bounds for flip graphs of all products of even-length chains.

\begin{conjecture}[Chv\'atal]\label{conjecture:chvatal1}
    Let $S$ be a finite set, and let $\mcD$ be a set of subsets of $S$ closed under taking subsets. Let $\mcF$ be a subset of $\mcD$ for which every pair of sets in $\mcF$ share an element. Then
    \[\abs{\mcF} \leq \max_{s \in S} \abs{\{A \in \mathcal{F} \mid s \in A\}}.\]
\end{conjecture}

\begin{remark}
    This conjecture seems to be open in all but a small number of cases. According to \cite{friedgut2018chvatal},
    the intractability of this conjecture stems from the inability to choose a suitable $s$ in the maximum, and that current progress is limited.
    For the purposes of this paper, the conjecture only needs to be resolved in the case given in \cref{conjecture:chvatal3}.
\end{remark}

Before the next conjecture is stated, consider the following definition. Let $d$ be an even positive integer, and call a set $\mcH$ of subsets of $S = \{1, \dots, d\}$ \emph{uniform} if
\begin{itemize}
    \item every subset has size $\frac{1}{2}d$,
    \item for every subset $A \subset S$ of size $\frac{1}{2}d$, exactly one of $A$ and $S \setminus A$ is in $\mcH$, and
    \item for every element $s \in S$, $s$ appears in at most $\ceil{\frac{1}{4} \binom{2d}{d}}$ subsets of $\mcH$.
\end{itemize}
Observe that if $d$ is not a power of 2, then $\frac{1}{4} \binom{2d}{d}$ is an integer, so every element must appear exactly $\frac{1}{4} \binom{2d}{d}$ times. The existence of uniform subsets is given by \cite{meyerowitz1995maximal} in Lemma 13.

\begin{conjecture}\label{conjecture:chvatal2}
    Let $d$ be an even positive integer. There exists a uniform set of subsets $\mcH$ such that if
    \[\mcD = \mcH \cup \{A \subset \{1, \dots, d\} \mid \abs{A} = \tfrac{1}{2} d - 1\},\]
    then every subset $\mcF$ of $\mcD$ for which every pair of sets in $\mcF$ share an element satisfies
    \[\mcF \leq \ceil{\frac{1}{4}\binom{d}{\frac{1}{2}d}}+ \binom{d-1}{\tfrac{1}{2}d-2}.\]
\end{conjecture}

\begin{remark}
    This conjecture is true for $d \leq 6$. For $d=2$, select $\mcH = \{\{1\}\}$. For $d=4$, select $\mcH = \{\{1, 2\}, \{1, 3\}, \{2, 3\}\}$. For $d=6$, select
    \[\mcH=\{\{1, 2, 3\}, \{1, 2, 4\},  \{1, 3, 6\}, \{1, 4, 5\}, \{1, 5, 6\}, \{2, 3, 5\}, \{2, 4, 6\}, \{2, 5, 6\}, \{3, 4, 6\}, \{3, 4, 5\}\}.\]
    The fact that these choices of $\mcH$ work for $d = 2$ and $d=4$ is easy to verify. For $d=6$, we present two interesting proofs for $\abs{F} \leq \frac{1}{4} \binom{6}{3} + \binom{5}{1} = 10$. However, these proofs seem difficult to generalize to larger values of $d$, because it appears that for large $d$ the structure of uniform $\mcH$ is not well-behaved. 
\end{remark}

\begin{proof}[Proof 1]
    Define
    \begin{align*}
        \mcD_1 &= \{\{4, 5\}, \{1, 2, 3\}, \{5, 6\}, \{1, 2, 4\}, \{3, 6\}\}, \\
        \mcD_2 &= \{\{2, 3\}, \{1, 4, 5\}, \{2, 6\}, \{3, 4, 5\}, \{1, 6\}\}, \\
        \mcD_3 &= \{\{4, 6\}, \{2, 3, 5\}, \{1, 4\}, \{2, 5, 6\}, \{1, 3\}\}, \\
        \mcD_4 &= \{\{2, 5\}, \{1, 3, 6\}, \{2, 4\}, \{1, 5, 6\}, \{3, 4\}\}, \\
        \mcD_5 &= \{\{3, 5\}, \{2, 4, 6\}, \{1, 5\}, \{3, 4, 6\}, \{1, 2\}\}.
    \end{align*}
    Observe that $\mcD = \mcD_1 \sqcup \mcD_2 \sqcup \mcD_3 \sqcup \mcD_4 \sqcup \mcD_5$
    so
    \[\abs{\mcF \cap \mcD} = \abs{\mcF \cap \mcD_1} + \abs{\mcF \cap \mcD_2} + \abs{\mcF \cap \mcD_3} + \abs{\mcF \cap \mcD_4} + \abs{\mcF \cap \mcD_5} \leq 2 + 2 + 2 + 2 + 2 = 10. \qedhere\]
\end{proof}

\begin{proof}[Proof 2]
    Let $\mcF$ be a subset of $\mcD$ for which every pair of sets in $\mcF$ share an element. For every subset $A \subseteq \{1, 2, 3, 4, 5, 6\}$, define
    \[\bbmone_A = \begin{cases} 1 & \text{if $A \in \mcF$} \\ 0 & \text{otherwise.}\end{cases}\]
    Observe that
    \[\bbmone_{123} + \tfrac{1}{2}(\bbmone_{45} + \bbmone_{46} + \bbmone_{56}) \leq 1 + \tfrac{1}{2} \bbmone_{45}\bbmone_{46}\bbmone_{56},\]
    and analogous symmetric statements hold.
    Summing this inequality over all sets in $\mcH$ gives
    \[\sum_{A \in \mcH} \bbmone_{A} + \smashoperator{\sum_{\substack{A \subset \{1, \dots, 6\} \\ \abs{A} = 2}}} \bbmone_{A} \leq 10 + \tfrac{1}{2} \smashoperator{\sum_{\substack{a, b, c \in \{1, \dots, 6\} \\ \{1, \dots, 6\} \setminus \{a, b, c\} \in \mcH}}} \bbmone_{ab}\bbmone_{ac}\bbmone_{bc} \leq 10 + \tfrac{1}{2}.\]
    Since the left hand side is an integer, it follows that \[\sum_{A \in \mcH} \bbmone_{A} + \smashoperator{\sum_{\substack{A \subset \{1, \dots, 6\} \\ \abs{A} = 2}}} \bbmone_A = \sum_{A \in \mcD} \bbmone_A \leq 10. \qedhere\]
\end{proof}

\begin{conjecture}\label{conjecture:chvatal3}
    Let $d$ be an even positive integer. There exists a uniform subset $\mcH$ such that if
    \[\mcD = \mcH \cup \{A \subset \{1, \dots, d\} \mid \abs{A} \leq \tfrac{1}{2} d - 1\},\]
    then every subset $\mcF$ of $\mcD$ for which every pair of sets in $\mcF$ share an element satisfies
    \[\mcF \leq \ceil{\frac{1}{4}\binom{d}{\frac{1}{2}d}} + \sum_{i=1}^{\frac{1}{2}d-1} \binom{d-1}{i-1}.\]
\end{conjecture}

\begin{remark}
    This conjecture follows from \cref{conjecture:chvatal2} and \cref{theorem:erdoskorado} after partitioning
    \[\mcD = \left(\mcH \cup \{A \subset \{1, \dots, d\} \mid \abs{A} = \tfrac{1}{2}d-1\}\right) \cup \smash\bigcup_{i \in \{0, \dots, \frac{1}{2}d-2\}} \{A \subset \{1, \dots, d\} \mid \abs{A} = i\}.\]
    This conjecture is also a specific case of Chv\'atal's conjecture (\cref{conjecture:chvatal1}).
\end{remark}

\section{Cyclically Symmetric Self-Complementary Ideals}\label{section:cssc}
The flip graph on self-complementary ideals of a poset can be modified to accommodate additional symmetries the ideals may have. In particular, there are flip graphs whose vertices are the cyclically symmetric self-complementary (CSSC) ideals of a poset $P = [\ell] \times [\ell] \times [\ell] = [2r] \times [2r] \times [2r]$ for an even positive integer $\ell = 2r$.

Given two distinct CSSC ideals $I, J \subset P$, $I$ and $J$ differ by a \emph{flip} if there exists a point $(a_1, a_2, a_3)\in I$ such that the replacements
\begin{align*}
    (a_1, a_2, a_3)&\mapsto (2r + 1 - a_1, 2r + 1 - a_2, 2r + 1 - a_3)\\
    (a_2, a_3, a_1)&\mapsto (2r + 1 - a_2, 2r + 1 - a_3, 2r + 1 - a_1)\\
    (a_3, a_1, a_2)&\mapsto (2r + 1 - a_3, 2r + 1 - a_1, 2r + 1 - a_2)
\end{align*}
yield $J$. The \emph{flip graph on cyclically symmetric self-complementary ideals of $P$} is an unweighted graph whose edges are constructed between CSSC ideals of $P$ that differ by a flip. \cref{fig:csscgraph} shows the flip graph structure for the specific case $P = [4] \times [4] \times [4]$. 

\begin{figure}[htbp]
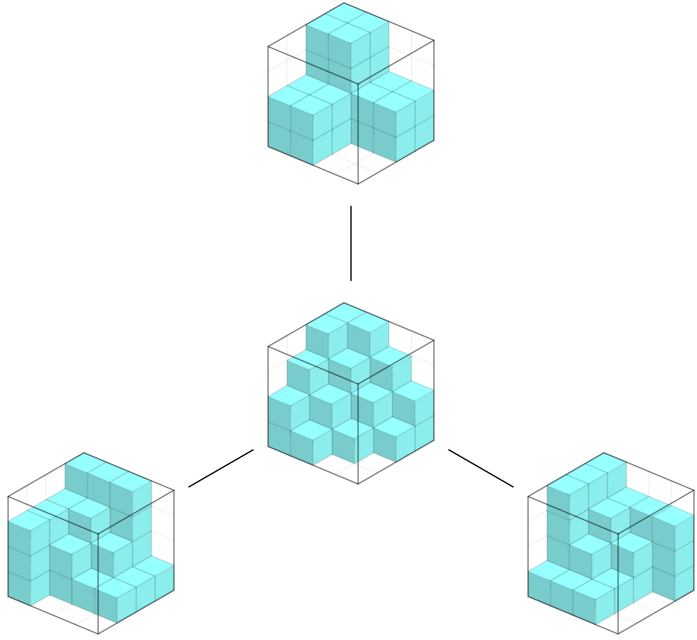

    \centering
    \begin{diagram}[0.6]
        void v(int[][] heights, pair o) {
            pair s = (4*v1+4*v2+4*v3)/2;
            fill(circle(o, 4.5), white);
            backframe(4, 4, 4, o-s);
            ideal(heights, o-s);
            frontframe(4, 4, 4, o-s);
        }
        int[][][] ideals = new int[][][] {
            {{4, 4, 3, 2},
             {4, 3, 2, 1},
             {3, 2, 1, 0},
             {2, 1, 0, 0}},
            {{4, 4, 2, 2},
             {4, 4, 2, 2},
             {2, 2, 0, 0},
             {2, 2, 0, 0}},
            {{4, 4, 4, 1},
             {3, 3, 2, 1},
             {3, 2, 1, 1},
             {3, 0, 0, 0}},
            {{4, 3, 3, 3},
             {4, 3, 2, 0},
             {4, 2, 1, 0},
             {1, 1, 1, 0}}};
        real s = 12;
        pair[] locations = new pair[] {
            (0, 0),
            dir(90)*s,
            dir(210)*s,
            dir(330)*s};
        int[][] edges = new int[][] {
            {0, 1},
            {0, 2},
            {0, 3}};
        for (int i = 0; i < edges.length; ++i) {
            int v1 = edges[i][0];
            int v2 = edges[i][1];
            draw(locations[v1] -- locations[v2]);
        }
        for (int i = 0; i < ideals.length; ++i) {
            v(ideals[i], locations[i]);
        }
    \end{diagram}
    \caption{The flip graph on cyclically symmetric self-complementary ideals of $[4] \times [4] \times [4]$.}
    \label{fig:csscgraph}
\end{figure}

Observe that at least two of $a_1, a_2, a_3$ must be distinct for the above flip to yield another ideal, because
\[\{(1, 1, 1), (2, 2, 2), \dots, (r, r, r)\}\subset I\]
for any self-complementary ideal $I \subset P$. Thus the points $(a_1, a_2, a_3)$, $(a_2, a_3, a_1)$, and $(a_3, a_1, a_2)$ above are distinct, and we have
\[\abs{I \setminus J} = \abs{J \setminus I} = 3\]
for any two CSSC ideals that differ by a flip. Observe that this condition is also sufficient: if CSSC ideals $I$ and $J$ satisfy $\abs{I \setminus J} = \abs{J \setminus I} = 3$ and $(a_1, a_2, a_3)\in I\setminus J$, then flipping $(a_1, a_2, a_3)$, $(a_2, a_3, a_1)$, and $(a_3, a_1, a_2)$ in $I$ yields the ideal $J$. This readily implies the following lemma:

\begin{lemma}\label{lemma:csscdistance}
    Let $I$ and $J$ be cyclically symmetric self-complementary ideals of $[2r]\times[2r]\times[2r]$. Then the distance from $I$ to $J$ in the flip graph is $\frac{1}{3}\abs{I \setminus J}$.
\end{lemma}

\begin{proof}
    This is essentially the same as the proof of \cref{lemma:scdistance}.
\end{proof}

Given a positive integer $r$, let $G_{C_3}(2r, 2r, 2r)$ denote the flip graph on cyclically symmetric self-complementary ideals of $[2r] \times [2r] \times [2r]$. Additionally, for each tuple $(t_1, t_2, t_3) \in \{0, 1\}^3$ define
\[O_{t_1t_2t_3} = \{(a_1, a_2, a_3) \in [2r] \times [2r] \times [2r] \mid t_1 = \bbmone_{a_1 > r}, t_2 = \bbmone_{a_2 > r}, t_3 = \bbmone_{a_3 > r}\}.\]

For ideals $I, J\subseteq P$ let $I \symdif J$ denote their symmetric difference, equal to $(I\setminus J)\cup (J\setminus I)$. Observe that if $I$ and $J$ are CSSC ideals that differ by a flip, then 
\[\abs{I\symdif J} = 6,\] 
and for $t_1t_2t_3\in \{001, 010, 100, 011, 101, 110\}$, we have
\[\abs{(I \symdif J) \cap O_{t_1t_2t_3}} = 1.\]

\subsection{Vertex Count}

The number of vertices in $G_{C_3}(2r, 2r, 2r)$, also equal to the number of cyclically symmetric self-complementary ideals of $[2r] \times [2r] \times [2r]$, was determined by Kuperberg in 1994.

\begin{theorem}\label{theorem:csscvertex}
    Let $r$ be a positive integer. The number of vertices in $G_{C_3}(2r, 2r, 2r)$ is
    \[\left(\prod _{j=0}^{r-1}{\frac {(3j+1)!}{(r+j)!}}\right)^2.\]
\end{theorem}

\begin{proof}
    This is given in \cite{kuperberg1994symmetries}.
\end{proof}

\begin{corollary}\label{corollary:csscasymptotics}
    As $r$ tends to infinity, we have
    \[v(G_{C_3}(2r, 2r, 2r)) \sim Cr^{-5/18}\left(\frac{27}{16}\right)^{r^2}\]
    for some real number $C$. In particular, $\log v(G_{C_3}(2r, 2r, 2r)) \sim \log(27/16) r^2 \approx 0.523r^2$.
\end{corollary}

\begin{proof}
    In fact, \[C=\left(\frac{e^{\zeta'(-1)}2^{1/4}\Gamma(2/3)}{3^{1/12}\Gamma(1/3)}\right)^{2/3} \approx 0.600.\]
    This can be obtained using asymptotic formulas for the superfactorial numbers and triple factorial numbers.
\end{proof}

\subsection{Diameter}
This section will compute the exact value of the diameter of $G_{C_3}(2r, 2r, 2r)$ for all positive integers $r$.

\begin{lemma}\label{lemma:csscdiameterlowerbound}
    Let $r$ be a positive integer. The diameter of $G_{C_3}(2r, 2r, 2r)$ is at least $\frac{1}{3}(r-1)r(r+1)$.
\end{lemma}

\begin{proof}
    Let $I$ denote the CSSC ideal $O_{000}\cup O_{001}\cup O_{010}\cup O_{100}$, and let $J$ denote the CSSC ideal constructed as a union of three ``pyramids" consisting of $1\times 1, 3\times 3, \dots, (2r-1)\times (2r-1)$ squares as layers. \cref{fig:csscdiameter} depicts these ideals in the case $r=5$.
    
    By \cref{lemma:csscdistance}, the distance between $I$ and $J$ in the flip graph is
    \[\tfrac{1}{3}\abs{I \setminus J} = \abs{O_{011} \cap (I \setminus J)} = 1 \cdot 2 + 2 \cdot 3 + \dots + (r-1) \cdot r = \tfrac{1}{3}(r-1)r(r+1).\]
    Therefore, the diameter of the flip graph is at least $\frac{1}{3}(r-1)r(r+1)$.
\end{proof}

\begin{figure}[htbp]
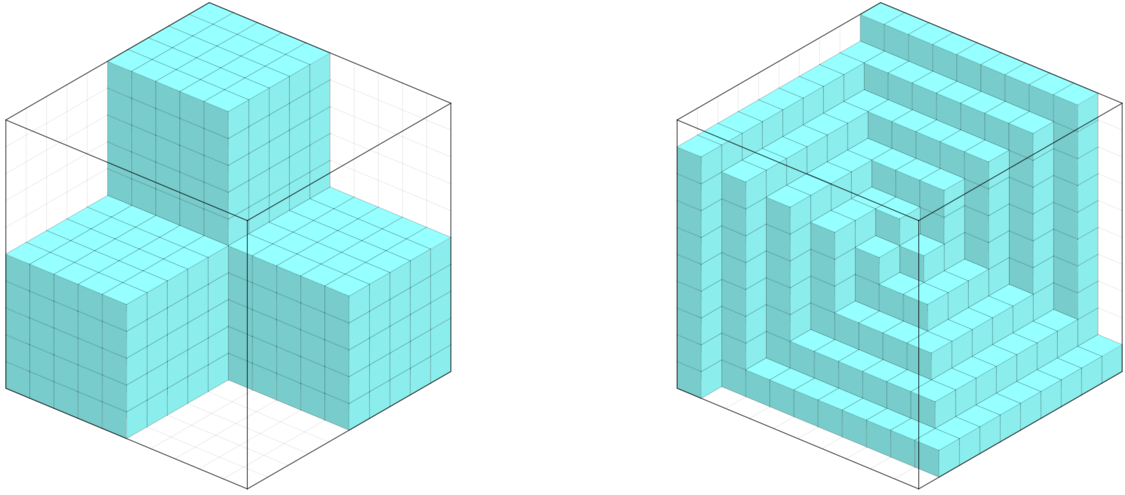

    \centering
    \begin{diagram}[0.9]
        backframe(10, 10, 10);
        ideal(new int[][] {{10, 10, 10, 10, 10, 5, 5, 5, 5, 5},
                           {10, 10, 10, 10, 10, 5, 5, 5, 5, 5}, 
                           {10, 10, 10, 10, 10, 5, 5, 5, 5, 5},
                           {10, 10, 10, 10, 10, 5, 5, 5, 5, 5},
                           {10, 10, 10, 10, 10, 5, 5, 5, 5, 5},
                           {5, 5, 5, 5, 5, 0, 0, 0, 0, 0},
                           {5, 5, 5, 5, 5, 0, 0, 0, 0, 0},
                           {5, 5, 5, 5, 5, 0, 0, 0, 0, 0},
                           {5, 5, 5, 5, 5, 0, 0, 0, 0, 0},
                           {5, 5, 5, 5, 5, 0, 0, 0, 0, 0}});
        frontframe(10, 10, 10);
        pair s = (25, 0);
        backframe(10, 10, 10, s);
        ideal(new int[][] {{10, 10, 10, 10, 10, 10, 10, 10, 10, 1},
                           {9, 9, 9, 9, 9, 9, 9, 9, 2, 1}, 
                           {9, 8, 8, 8, 8, 8, 8, 3, 2, 1},
                           {9, 8, 7, 7, 7, 7, 4, 3, 2, 1},
                           {9, 8, 7, 6, 6, 5, 4, 3, 2, 1},
                           {9, 8, 7, 6, 5, 4, 4, 3, 2, 1},
                           {9, 8, 7, 6, 3, 3, 3, 3, 2, 1},
                           {9, 8, 7, 2, 2, 2, 2, 2, 2, 1},
                           {9, 8, 1, 1, 1, 1, 1, 1, 1, 1},
                           {9, 0, 0, 0, 0, 0, 0, 0, 0, 0}}, s);
        frontframe(10, 10, 10, s);
    \end{diagram}
    \caption{$I$ and $J$ for the case $r=5$. These ideals attain the maximum possible distance in the flip graph on cyclically symmetric self-complementary ideals of $[10] \times [10] \times [10]$.}
    \label{fig:csscdiameter}
\end{figure}

\begin{theorem}\label{theorem:csscdiameter}
    Let $r$ be a positive integer. The diameter of $G_{C_3}(2r, 2r, 2r)$ is $\frac{1}{3}(r-1)r(r+1)$.
\end{theorem}

\begin{proof}
    The lower bound is given by \cref{lemma:csscdiameterlowerbound}. Since the diameter is at most twice the radius, the upper bound is given by \cref{lemma:csscradiusupperbound}.
\end{proof}

\subsection{Radius}

This section will compute the exact value of the radius of $G_{C_3}(2r, 2r, 2r)$ for all positive integers $r$. We first prove two general lemmas about CSSC ideals.

\begin{lemma}\label{lemma:csscmustinclude}
    All cyclically symmetric self-complementary ideals $I\subset [2r]\times [2r]\times [2r]$ must contain the points $(i, i, 2r + 1 - i)$ for each integer $1 \leq i \leq r$.
\end{lemma}

\begin{proof}
    If $(i, i, 2r + 1 - i)\notin I$, then $(2r + 1 - i, 2r + 1 - i, i)\in I$ from $I$ being self-complementary. Since $I$ is cyclically symmetric, we also have $(i, 2r + 1 - i, 2r + 1 - i)\in I$, which is a contradiction because $2r + 1 - i > i$.
\end{proof}

\begin{lemma}\label{lemma:csscmustinclude2}
    All cyclically symmetric self-complementary ideals $I\subset [2r]\times [2r]\times [2r]$ must contain at least one of $(1, r, 2r)$ or $(r, 1, 2r)$.
\end{lemma}

\begin{proof}
    Suppose that $(1, r, 2r)\notin I$, so $(2r, r+1, 1)\in I$, and thus $(r+1, 1, 2r)\in I$. Then $I$ contains $(r, 1, 2r)$, as it is covered by $(r+1, 1, 2r)$.
    The other case is similar.
\end{proof}

Next, we prove a lemma about ideals in two dimensions that will be used when inducting with 3-dimensional ideals.

\begin{lemma}\label{lemma:cssc2didealdifference}
    Let $r\ge 2$ be a positive integer, and let $k\le r-2$ be a nonnegative integer.
    Let $C$ be the self-complementary ideal in $[r-k-1]\times [r+k]$ given by $\{(a_1, a_2)\in [r-k-1]\times [r+k]\mid a_1 + a_2\le r\}$. (An example of $C$ is given in \cref{fig:cssc2didealdifference}.)
    Then for any ideal $I\subseteq [r-k-1]\times [r+k]$, not necessarily self-complementary, we have 
    \[\abs{C\symdif I}\le (k+1) + (k+2) + \cdots + (r-1) = \tfrac{1}{2}(r+k)(r-k-1).\]
    Furthermore, equality holds only when $I$ is the empty ideal or the entire poset.
\end{lemma}

\begin{figure}[htbp]
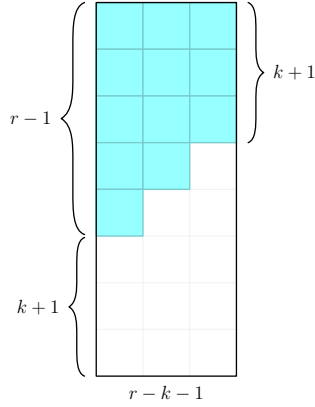

    \centering
    \begin{diagram}[0.25]
        int r = 6;
        int k = 2;
        int w = r-k-1;
        int h = r+k;
        backperimeter(w, h);
        for (int i = 0; i < w; ++i) {
            for (int j = 0; j < r-1-i; ++j) {
                rectangle(i, h-1-j, 1, 1);
            }
        }
        frontperimeter(w, h);

        transform t = scale(0.6);
        real eps = 0.25;
        path p1 = brace((0, k+1) + eps*dir(180)+eps/8*dir(90), (0, h) + eps*dir(180), 1/2);
        draw(p1);
        label(t*"$r-1$", p1, dir(180));
        path p2 = brace((w, h) + eps*dir(0), (w, r-1) + eps*dir(0), 1/3);
        draw(p2);
        label(t*"$k+1$", p2, dir(0));
        path p3 = brace((0, 0) + eps*dir(180), (0, k+1) + eps*dir(180)-eps/8*dir(90), 1/3);
        draw(p3);
        label(t*"$k+1$", p3, dir(180));
        label(t*"$r-k-1$", (0,0)/2 + (w, 0)/2, dir(-90));
    \end{diagram}
    \caption{An example of $C$ for the case $r=6$, $k=2$.}
    \label{fig:cssc2didealdifference}
\end{figure}

\begin{proof}
    For each integer $1\le i\le r-k-1$, define $p_i = \abs{\{(a_1, a_2)\in I\mid a_1 = i\}}$, so $0\le p_{r-k-1}\le \cdots\le p_1\le r+k$.
    Since $C$ contains $r-i$ points with $a_1 = i$ for each integer $1 \leq i \leq r-k-1$, we wish to upper-bound 
    \[\abs{C\symdif I} = \sum_{i = 1}^{r-k-1} |p_i - (r-i)|.\] 
    We will pair up column $i$ with column $r - k -i$ and show that 
    \[|p_i - (r-i)| + |p_{r-k-i} - (k+i)|\le r+k.\] 
    For the following, assume without loss of generality that $i\le r-k-i$, so $0\le p_{r-k-i}\le p_i\le r+k$.  

    For a fixed $p_i$, the maximum value of $|p_i - (r-i)| + |p_{r-k-i} - (k+i)|$ obtained by varying $p_{r-k-i}$ must occur at an extreme: when $p_{r-k-i} = p_i$ or $p_{r-k-i} = 0$. Similarly for a fixed $p_{r-k-i}$, the maximum value obtained by varying $p_i$ occurs when $p_i = p_{r-k-i}$ or when $p_i = r+k$. The maximum value now reduces to a finite case check:
    \begin{itemize}
        \item \textbf{Case 1:} $p_i\neq p_{r-k-i}$. 
        
        Then we must have $p_i = r+k$ and $p_{r-k-i} = 0$, and 
        \[|(r+k) - (r-i)| + |0 - (k+i)| = 2k + 2i\le r+k,\] by the assumption that $i\le r-k-i$. 
        \item \textbf{Case 2:} $p_i = p_{r-k-i}$. 
        
        Setting $x = p_i = p_{r-k-i}$, the maximum value of $|x-(r-i)| + |x - (k+i)|$ must occur at an endpoint of $x = 0$ or $x=r+k$, both of which give $r+k$.
    \end{itemize}

    Summing the inequalities $|p_i - (r-i)| + |p_{r-k-i} - (k+i)|\le r+k$ for $i = 1$ to $r-k-1$ and dividing by 2 yields the desired conclusion of 
    \[\abs{C\symdif I}\le \tfrac{1}{2}(r+k)(r-k-1).\]

    In order for equality to hold, we need $|p_i - (r-i)| + |p_{r-k-i} - (k+i)| = r+k$ for all $1\le i\le \frac{r-k}{2}$. Note that equality cannot actually hold in Case 1 above, because that would require $i = r-k-i$, but then $p_i = p_{r-k-i}$, so Case 2 would apply instead. Thus, to achieve equality, we must apply Case 2 and have $p_i = p_{r-k-i} = 0$ or $p_i = p_{r-k-i} = r+k$ for each integer $1 \leq i \leq \frac{r-k}{2}$.

    In particular, we have that $p_1 = p_{r-k-1} = 0$ or $p_1 = p_{r-k-1} = r+k$. This determines that either $p_i = 0$ for all $i$, or $p_i = r+k$ for all $i$, giving that $I$ is either the empty ideal or the entire poset. 
\end{proof}

\begin{lemma}\label{lemma:csscradiusupperbound}
    The radius of $G_{C_3}(2r, 2r, 2r)$ at most $\frac{1}{6}(r-1)r(r+1)$.
\end{lemma}

\begin{proof}
    Let $C_{2r}$ denote the CSSC ideal of $[2r]\times [2r]\times [2r]$ given by 
    \[\{(a_1, a_2, a_3)\in [2r]\times [2r]\times [2r]\mid a_1 + a_2 + a_3\le 3r + 1\},\]
    and let $I$ denote any CSSC ideal of $[2r] \times [2r] \times [2r]$. 
    $C_{2r}$ for the case $r=5$ is shown in \cref{fig:cssccenter}.

    \begin{figure}[htbp]
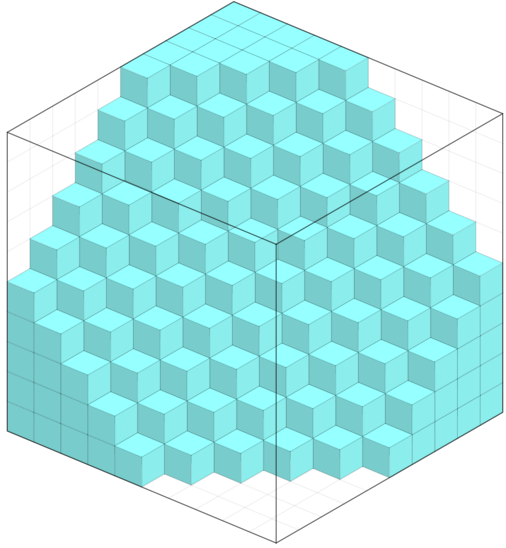

    \centering
    \begin{diagram}[0.4]
        backframe(10, 10, 10);
        ideal(new int[][] {{10, 10, 10, 10, 10, 9, 8, 7, 6, 5},
                           {10, 10, 10, 10, 9, 8, 7, 6, 5, 4}, 
                           {10, 10, 10, 9, 8, 7, 6, 5, 4, 3},
                           {10, 10, 9, 8, 7, 6, 5, 4, 3, 2},
                           {10, 9, 8, 7, 6, 5, 4, 3, 2, 1},
                           {9, 8, 7, 6, 5, 4, 3, 2, 1, 0},
                           {8, 7, 6, 5, 4, 3, 2, 1, 0, 0},
                           {7, 6, 5, 4, 3, 2, 1, 0, 0, 0},
                           {6, 5, 4, 3, 2, 1, 0, 0, 0, 0},
                           {5, 4, 3, 2, 1, 0, 0, 0, 0, 0}});
        frontframe(10, 10, 10);
    \end{diagram}
    \caption{$C_{2r}$ for the case $r=5$. This ideal lies in the center of the flip graph on cyclically symmetric self-complementary ideals of $[10] \times [10] \times [10]$.}
    \label{fig:cssccenter}
\end{figure}
    
    By \cref{lemma:csscdistance}, it suffices to show that
    \[\tfrac{1}{6} \abs{C_{2r} \symdif I} \leq \tfrac{1}{6}(r-1)r(r+1),\]
    or that
    \[\abs{O_{001}\cap (C_{2r}\symdif I)}\leq \tfrac{1}{6}(r-1)r(r+1).\]
    
    We will induct on $r$; the base case of $r=1$ is true because there is only one CSSC ideal, which is the center. Now assuming the result for $r-1$, we will show it for $r$. 
    
    Let the \emph{core} of a CSSC ideal $I\subset [2r]\times [2r]\times [2r]$ refer to its intersection with $\{(a_1, a_2, a_3)\mid 2\le a_i\le 2r-1\}$ (i.e.\ the points in the inner $[2r-2]\times [2r-2]\times [2r-2]$ cube), and let the \emph{shell} refer to the rest of $I$. Observe that the core of a CSSC ideal of $[2r]\times [2r]\times [2r]$ is a CSSC ideal of $[2r-2]\times [2r-2]\times [2r-2]$, and that the core of $C_{2r}$ is $C_{2r-2}$, which is the center in the $r-1$ case by the inductive hypothesis.
    The inductive hypothesis for $r-1$ states that symmetric difference of the cores of $I$ and $C_{2r}$ in $O_{001}$ is at most $\frac{1}{6}(r-2)(r-1)r$. Now it suffices to show that the symmetric difference of the shells of $I$ and $C_{2r}$ in $O_{001}$ is at most $\frac{1}{6}(r-1)r(r+1) - \frac{1}{6}(r-2)(r-1)r = \frac{1}{2}(r-1)r$.  

    From \cref{lemma:csscmustinclude2}, we may assume without loss of generality that $I$ contains $(r, 1, 2r)$, and therefore all points $(a_1, a_2, a_3)$ in the shell in $O_{001}$ with $a_2 = 1$. Note that $C_{2r}$ contains $(r, 1, 2r)$ as well.

    Consider the remaining points in $O_{001}$ that the shells of $C_{2r}$ and $I$ may differ in, which is the union of a vertical layer of points $V = \{(a_1, a_2, a_3)\in O_{001}\mid a_1 = 1, a_2\ge 2, a_3\le 2r-1\}$ and the top layer of points $T = \{(a_1, a_2, a_3)\in O_{001}\mid a_2\ge 2, a_3 = 2r\}$.
    
    Consider swinging the vertical layer $V$ ``behind" the top layer $T$ to form the poset $[2r - 1]\times [r-1]\times [1]$, which we will now view as an ideal in $P' = [2r-1]\times [r-1]$. 
    
    More precisely, the point $(1, a_2, a_3)\in V$ for $2\le a_2\le r$ and $r+1\le a_3\le 2r-1$ corresponds to the point $(a_3 - r, a_2 - 1)$ in $P'$, and the point $(a_1, a_2, 2r)\in T$ for $1\le a_1\le r$ and $2\le a_2\le r$ corresponds to the point $(a_1 + r - 1, a_2 - 1)$ in $P'$. An example of $V$, $T$, and $P'$ is illustrated in \cref{fig:csscvtp'}.  
    
    \begin{figure}[htbp]
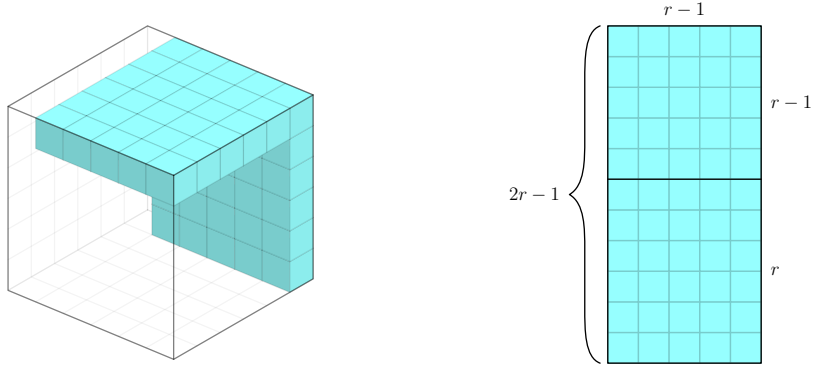

        \centering
        \begin{diagram}[0.65]
            pair s = (-15, 5);
            backframe(6, 6, 6, s);
            for (int a = 1; a < 6; ++a) {
                for (int b = 6; b < 11; ++b) {
                    box(0, a, b-6, 1, 1, 1, s);
                }
            }
            for (int a = 1; a < 6; ++a) {
                for (int b = 0; b < 6; ++b) {
                    box(b, a, 11-6, 1, 1, 1, s);
                }
            }
            frontframe(6, 6, 6, s);
            int r = 6;
            int k = 3;
            int w = r-1;
            int h = k+r+2;
            backperimeter(w, h);
            for (int i = 0; i < 5; ++i) {
                for (int j = 0; j < 11; ++j) {
                    rectangle(i, j, 1, 1);
                }
            }
            draw((0, 6) -- (5, 6));
            frontperimeter(w, h);

            real eps = 0.25;
            transform t = scale(0.6);
            path p1 = brace((0, 0) + eps*dir(180), (0, h) + eps*dir(180), 1);
            draw(p1);
            label(t*"$2r-1$", p1, dir(180));
            label(t*"$r-1$", (0, 11)/2 + (5, 11)/2, dir(90));
            label(t*"$r-1$", (5, 11)/2 + (5, 6)/2, dir(0));
            label(t*"$r$", (5, 6)/2 + (5, 0)/2, dir(0));
        \end{diagram}
        \caption{$V\cup T$ for the case $r=6$ is depicted on the left; note that only $O_{001}$ is drawn. The corresponding $P'$ is shown on the right.}
        \label{fig:csscvtp'}
    \end{figure}

    Let $I'$ and $C_{2r}'$ denote the ideals in $P'$ corresponding to the shells of $I$ and $C_{2r}$ in $V\cup T$.
    It remains to show that $\abs{C_{2r}'\symdif I'}\le \frac{1}{2}(r-1)r$. We split into two cases:

    \begin{itemize}
        \item \textbf{Case 1:} $(1, r, 2r-1)\in I$. 
        
        This corresponds to $I'$ containing the point $(r-2, r-1)$, which we note is also contained in $C_{2r}'$. Then the remainder of $P'$ not covered by the point $(r-2, r-1)$, in which $I'$ and $C_{2r}'$ may differ, is a rectangle with dimensions $r\times (r-1)$. Applying the $k=0$ version of \cref{lemma:cssc2didealdifference} gives that $\abs{C_{2r}'\symdif I'}\le \frac{1}{2}(r-1)r$, as desired. 

        Applying the equality cases of \cref{lemma:cssc2didealdifference} gives that equality here holds when $I' = P'$ or when $I' = \{(a_1, a_2)\in P'\mid a_1\le r-2\}$.
        
        \item \textbf{Case 2:} $(1, r, 2r-1)\notin I$.

        Let $k$ be the largest positive integer such that $(1, r, 2r - k)\notin I$; we have that $k\ge 1$ by assumption and that $k\le r-2$ because $(r, r, r+1)\in I$ by \cref{lemma:csscmustinclude}. Then since $I$ is CSSC, we have that $(2r, r+1, k+1)\in I$, so $(r+1, k+1, 2r)\in I$. This means that $(2r-2, k)\in I'$. In other words, all points in the leftmost $k$ columns of $P'$ are contained in $I'$. In these leftmost $k$ columns, the symmetric difference between $C_{2r}'$ and $I'$ is $1 + 2 + \cdots + k$. 
    
        This leaves us with a $(r+k)\times (r-k-1)$ rectangle in which $C_{2r}'$ and $I'$ may differ. By \cref{lemma:cssc2didealdifference}, the symmetric difference between $C_{2r}'$ and $I'$ in that rectangle is at most $(k+1) + (k+2) + \cdots + (r-1)$. All together, we have that $C_{2r}'$ and $I'$  differ in at most $1 + 2 + \cdots + (r-1) = \frac{1}{2}(r-1)r$ points, as desired. A visual depiction of this logic is given in \cref{fig:cssc2dremainingdifferencewithk}.

        The equality cases of \cref{lemma:cssc2didealdifference} imply that equality here holds when $I' = P'$ or when $I' = \{(a_1, a_2)\in P'\mid a_1\le r-k-2 \text{ or } a_2\le k\}$. However, in this case, we have $I'\neq P'$ by the assumption that $(1, r, 2r-1)\notin I$, so we are left with only the latter equality case.
    \end{itemize}

        \begin{figure}[htbp]
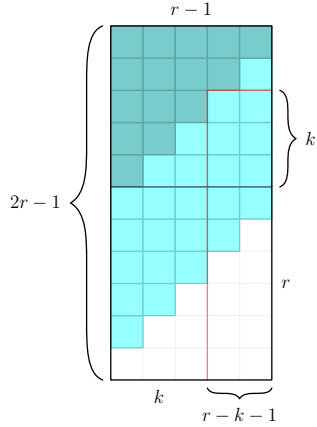

    \centering
        \begin{diagram}[0.25]
            int r = 6;
            int k = 3;
            int w = r-1;
            int h = k+r+2;
            backperimeter(w, h);
            for (int i = 0; i < w; ++i) {
                for (int j = 1; j < r; ++j) {
                    rectangle(i, i+j, 1, 1);
                }
            }
            for (int i = 0; i < w; ++i) {
                for (int j = r; i+j < h; ++j) {
                    darkrectangle(i, i+j, 1, 1);
                }
            }
            draw((0, r) -- (r-1, r), black+opacity(0.5));
            draw((k, 0) -- (k, r+k) -- (r-1, r+k), red+opacity(0.5));
            frontperimeter(w, h);
            transform t = scale(0.6);
            label(t*"$k$", (0, 0)/2 + (k, 0)/2, dir(-90));
            label(t * "$r$", (r-1, 0)/2 + (r-1, r)/2, dir(0));
            label(t*"$r-1$", (0, h)/2 + (w, h)/2, dir(90));

            real eps = 0.25;
            path p1 = brace((0, 0) + eps*dir(180), (0, h) + eps*dir(180), 1);
            draw(p1);
            label(t*"$2r-1$", p1, dir(180));
            path p2 = brace((w, r+k) + eps*dir(0), (w, r) + eps*dir(0), 1/2);
            draw(p2);
            label(t*"$k$", p2, dir(0));
            path p3 = brace((w, 0) + eps*dir(-90), (k, 0) + eps*dir(-90), 1/3);
            draw(p3);
            label(t*"$r-k-1$", p3, dir(-90));
        \end{diagram}
        \caption{An example of $P'$ for the case $r=6$, $k=3$. The highlighted cells are contained in $C_{2r}'$. The darker highlighted cells must also be contained in $I'$ by \cref{lemma:csscmustinclude}.}
        \label{fig:cssc2dremainingdifferencewithk}
    \end{figure}

    This concludes the proof that $\abs{C_{2r}'\symdif I'}\le \frac{1}{2}(r-1)r$. To complete the list of equality cases for the inductive step, we must note that we previously assumed without loss of generality that $(r, 1, 2r)\in I$, since $I$ contains one of $(r, 1, 2r)$ and $(1, r, 2r)$ by \cref{lemma:csscmustinclude2}. The symmetric case of when $(1, r, 2r)\in I$ introduces $r-2$ more equality cases that are analogous to the equality cases in Case 2, but with the first two coordinates swapped. Extending these equality cases in $O_{001}$ to all octants, the $2r-1$ CSSC shells that differ overall from the shell of $C_{2r}$ in the maximum possible $6\cdot \frac{1}{2}(r-1)r$ points are the CSSC shells $S_{k, 2r}$ such that
    \[\{(k, 2r-k, 2r), (2r-k, 2r, k), (2r, k, 2r-k)\}\subset S_{k, 2r}\] for each of $1\le k\le 2r-1$. The shells $S_{k, 2r}$ in the case of $r=5$ are depicted in \cref{fig:csscfurthestshells}. 
\end{proof}

\begin{figure}[htbp]
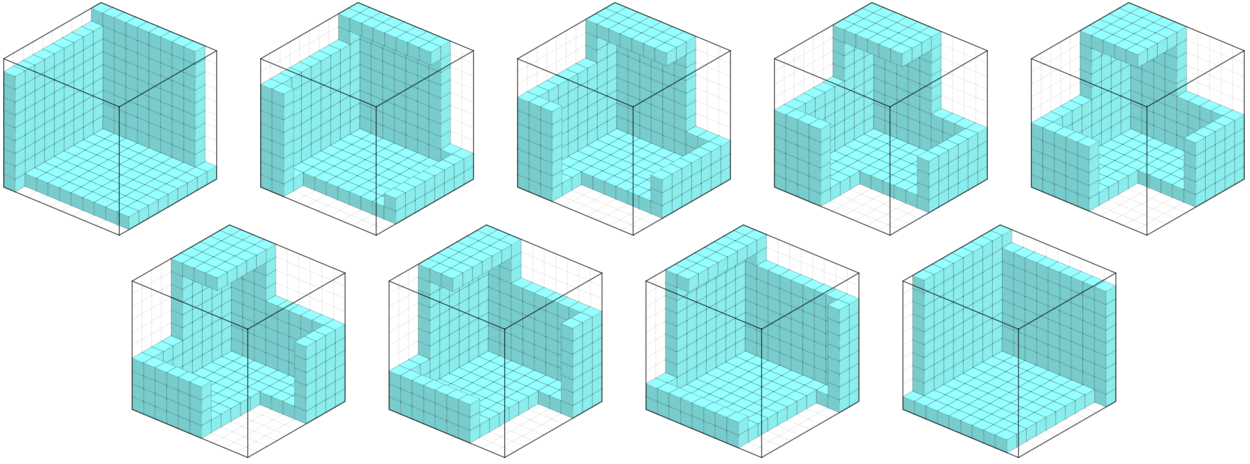

    \centering
    \begin{diagram}[1.0]
        int r = 5;
        int n = 2*r;
        for (int c = 1-r; c < r; ++c) {
            pair s;
            if (c <= 0) {
                s = 20*((c-2.5), 0);
            } else {
                s = 20*((c-7), -sqrt(3)/2);
            }
            backframe(10, 10, 10, s);
            for (int i = 0; i < n; ++i) {
                box(i, 0, 0, 1, 1, 1, s);
                box(0, i, 0, 1, 1, 1, s);
                box(0, 0, i, 1, 1, 1, s);
            }
            for (int a = 1; a < n; ++a) {
                for (int b = 1; b < n; ++b) {
                    if (a >= r-c && b >= r+c) {
                        continue;
                    }
                    box(a, b, 0, 1, 1, 1, s);
                    box(b, 0, a, 1, 1, 1, s);
                    box(0, a, b, 1, 1, 1, s);
                }
            }
            for (int a = 1; a < r+c; ++a) {
                for (int b = 1; b < r-c; ++b) {
                    box(a, b, n-1, 1, 1, 1, s);
                    box(b, n-1, a, 1, 1, 1, s);
                    box(n-1, a, b, 1, 1, 1, s);
                }
            }
            frontframe(10, 10, 10, s);
        }
    \end{diagram}
    \caption{The shells that are equality cases for the inductive step of \cref{lemma:csscradiusupperbound} in the case $r=5$.}
    \label{fig:csscfurthestshells}
\end{figure}

\begin{theorem}\label{theorem:csscradius}
    The radius of $G_{C_3}(2r, 2r, 2r)$ is $\frac{1}{6}(r-1)r(r+1)$.
\end{theorem}

\begin{proof}
    The upper bound is given by \cref{lemma:csscradiusupperbound}. Since the radius is at least half the diameter, the lower bound is given by \cref{lemma:csscdiameterlowerbound}.
\end{proof}

\begin{remark}
    From our code, it appears that the center of $G_{C_3}(2r, 2r, 2r)$ contains exactly 1 element, namely $C_{2r}$. This fact should not be hard to prove with a little more time, applying the work already done in this section.
\end{remark}

Furthermore, we can enumerate nicely the number of ideals that are furthest from $C_{2r}$ in the flip graph.

\begin{theorem}\label{theorem:furthestfromcenter}
    The number of ideals of $[2r]\times [2r]\times [2r]$ that are a distance of $\frac{1}{6}(r-1)r(r+1)$ from $C_{2r}$ in $G_{C_3}(2r, 2r, 2r)$ is $3^{r-1}$.
\end{theorem}

\begin{proof}
    We prove this via induction. The base case of $r=1$ is clear because there is only one CSSC ideal $C_2$ in $[2]\times[2]\times[2]$, of distance 0 from itself. The three ideals of distance 1 from $C_{2r}$ in the case of $r=2$ were illustrated previously in \cref{fig:scgraph} as well.

    Let $P_{2r}$ denote the set of CSSC ideals of $[2r]\times [2r]\times [2r]$ that are a distance of $\frac{1}{6}(r-1)r(r+1)$ from $C_{2r}$ in the flip graph. We will first construct three ways to add on a shell to each $I\in P_{2r-2}$ to form a CSSC ideal in $P_{2r}$. 

    Fix a CSSC ideal $I\in P_{2r-2}$, and recall from the equality cases of the inductive step in \cref{lemma:csscradiusupperbound} that its shell is $S_{k, 2r-2}$ for some $1\le k\le 2r-3$. Then the only three shells $S_{i, 2r}$ that can fit around $I$ to form a CSSC ideal in $P_{2r}$ are 
    \[S_{1, 2r}, S_{2r-1, 2r}, \text{ and } S_{k+1, 2r}.\]
    The first two shells can in fact fit around any ideal of $[2r-2]\times [2r-2]\times [2r-2]$. The third shell is specific to the fact that the shell of the ideal $I\in P_{2r-2}$ is $S_{k, 2r-2}$.
    For the particular $I\in P_{2r-2}$ for $r=5$ shown in \cref{fig:csscshells1}, the three shells are demonstrated in \cref{fig:csscshells2}, and the resulting ideals in $P_{2r}$ are shown in \cref{fig:csscshells3}. 
    Thus, given that $\abs{P_{2r-2}} = 3^{r-2}$, we have constructed exactly $3^{r-1}$ ideals in $P_{2r}$.
    An example of the tree structure arising from this proof is given in \cref{fig:tree}.
\end{proof}

\begin{figure}[htbp]
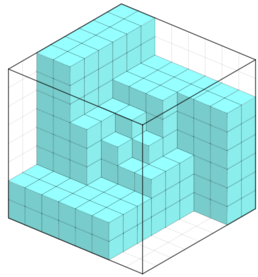

    \centering
    \begin{diagram}[0.2]
        backframe(8, 8, 8);
        ideal(new int[][] {{8, 8, 6, 6, 6, 6, 6, 6}, {8, 8, 6, 6, 6, 6, 6, 6}, {8, 8, 6, 6, 6, 3, 0, 0}, {8, 8, 5, 5, 4, 3, 0, 0}, {8, 8, 5, 4, 3, 3, 0, 0}, {8, 8, 5, 2, 2, 2, 0, 0}, {2, 2, 2, 2, 2, 2, 0, 0}, {2, 2, 2, 2, 2, 2, 0, 0}});
        frontframe(8, 8, 8);
    \end{diagram}
    \caption{An example of $I \in P_8$ with shell $S_{6, 8}$.}
    \label{fig:csscshells1}
\end{figure}

\begin{figure}[htbp]
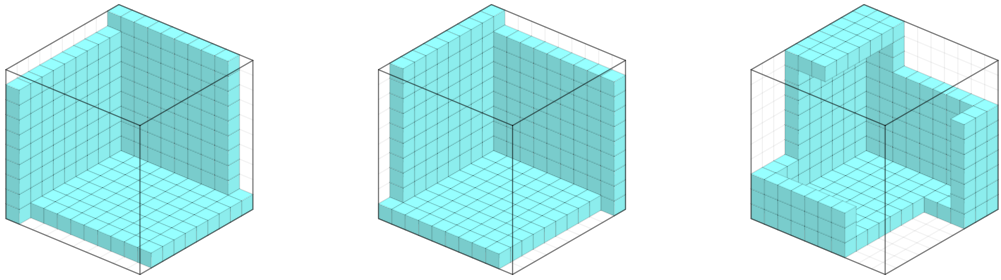

    \centering
    \begin{diagram}[0.8]
        pair s = (-25, -25);
        backframe(10, 10, 10, s);
        ideal(new int[][] {{10, 10, 10, 10, 10, 10, 10, 10, 10, 1},
                           {9, 1, 1, 1, 1, 1, 1, 1, 1, 1},
                           {9, 1, 1, 1, 1, 1, 1, 1, 1, 1},
                           {9, 1, 1, 1, 1, 1, 1, 1, 1, 1},
                           {9, 1, 1, 1, 1, 1, 1, 1, 1, 1},
                           {9, 1, 1, 1, 1, 1, 1, 1, 1, 1},
                           {9, 1, 1, 1, 1, 1, 1, 1, 1, 1},
                           {9, 1, 1, 1, 1, 1, 1, 1, 1, 1},
                           {9, 1, 1, 1, 1, 1, 1, 1, 1, 1},
                           {9, 0, 0, 0, 0, 0, 0, 0, 0, 0}}, s);
        frontframe(10, 10, 10, s);
        pair s = (0, -25);
        backframe(10, 10, 10, s);
        ideal(new int[][] {{10, 9, 9, 9, 9, 9, 9, 9, 9, 9},
                           {10, 1, 1, 1, 1, 1, 1, 1, 1, 0}, 
                           {10, 1, 1, 1, 1, 1, 1, 1, 1, 0},
                           {10, 1, 1, 1, 1, 1, 1, 1, 1, 0},
                           {10, 1, 1, 1, 1, 1, 1, 1, 1, 0},
                           {10, 1, 1, 1, 1, 1, 1, 1, 1, 0},
                           {10, 1, 1, 1, 1, 1, 1, 1, 1, 0},
                           {10, 1, 1, 1, 1, 1, 1, 1, 1, 0},
                           {10, 1, 1, 1, 1, 1, 1, 1, 1, 0},
                           {1, 1, 1, 1, 1, 1, 1, 1, 1, 0},}, s);
        frontframe(10, 10, 10, s);

        pair s = (25, -25);
        backframe(10, 10, 10, s);
        ideal(new int[][] {{10, 10, 10, 7, 7, 7, 7, 7, 7, 7},
                           {10, 1, 1, 1, 1, 1, 1, 1, 1, 7}, 
                           {10, 1, 1, 1, 1, 1, 1, 1, 1, 7},
                           {10, 1, 1, 1, 1, 1, 1, 0, 0, 0},
                           {10, 1, 1, 1, 1, 1, 1, 0, 0, 0},
                           {10, 1, 1, 1, 1, 1, 1, 0, 0, 0},
                           {10, 1, 1, 1, 1, 1, 1, 0, 0, 0},
                           {3, 1, 1, 1, 1, 1, 1, 0, 0, 0},
                           {3, 1, 1, 1, 1, 1, 1, 0, 0, 0},
                           {3, 3, 3, 3, 3, 3, 3, 0, 0, 0},}, s);
        for (int i = 1; i <= 6; ++i) {
            for (int j = 1; j <= 2; ++j) {
                box(i, j, 9, 1, 1, 1, s);
            }
        }
        frontframe(10, 10, 10, s);
    \end{diagram}
    \caption{The three shells $S_{1, 10}, S_{9, 10}, S_{7, 10}$ that can surround $I$ in \cref{fig:csscshells1} to form an ideal in $P_{10}$.}
    \label{fig:csscshells2}
\end{figure}

\begin{figure}[htbp]
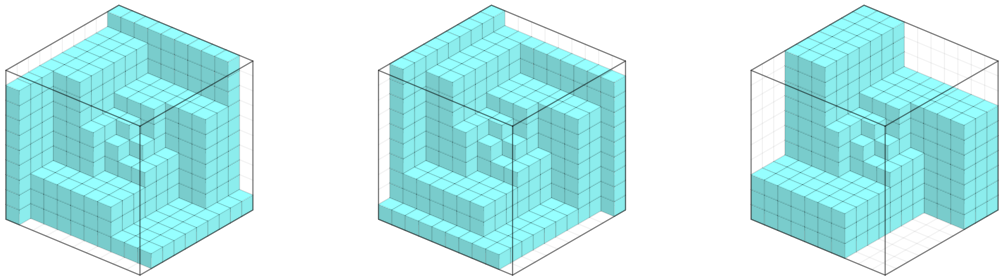

    \centering
    \begin{diagram}[0.8]
        pair s = (-25, -50);
        backframe(10, 10, 10, s);
        ideal(new int[][] {{10, 10, 10, 10, 10, 10, 10, 10, 10, 1},
                           {9, 9, 9, 7, 7, 7, 7, 7, 7, 1},
                           {9, 9, 9, 7, 7, 7, 7, 7, 7, 1},
                           {9, 9, 9, 7, 7, 7, 4, 1, 1, 1},
                           {9, 9, 9, 6, 6, 5, 4, 1, 1, 1},
                           {9, 9, 9, 6, 5, 4, 4, 1, 1, 1},
                           {9, 9, 9, 6, 3, 3, 3, 1, 1, 1},
                           {9, 3, 3, 3, 3, 3, 3, 1, 1, 1},
                           {9, 3, 3, 3, 3, 3, 3, 1, 1, 1},
                           {9, 0, 0, 0, 0, 0, 0, 0, 0, 0}}, s);
        frontframe(10, 10, 10, s);
        pair s = (0, -50);
        backframe(10, 10, 10, s);
        ideal(new int[][] {{10, 9, 9, 9, 9, 9, 9, 9, 9, 9},
                           {10, 9, 9, 7, 7, 7, 7, 7, 7, 0},
                           {10, 9, 9, 7, 7, 7, 7, 7, 7, 0},
                           {10, 9, 9, 7, 7, 7, 4, 1, 1, 0},
                           {10, 9, 9, 6, 6, 5, 4, 1, 1, 0},
                           {10, 9, 9, 6, 5, 4, 4, 1, 1, 0},
                           {10, 9, 9, 6, 3, 3, 3, 1, 1, 0},
                           {10, 3, 3, 3, 3, 3, 3, 1, 1, 0},
                           {10, 3, 3, 3, 3, 3, 3, 1, 1, 0},
                           {1, 1, 1, 1, 1, 1, 1, 1, 1, 0}}, s);
        frontframe(10, 10, 10, s);
        pair s = (25, -50);
        backframe(10, 10, 10, s);
        ideal(new int[][] {{10, 10, 10, 7, 7, 7, 7, 7, 7, 7},
                           {10, 10, 10, 7, 7, 7, 7, 7, 7, 7},
                           {10, 10, 10, 7, 7, 7, 7, 7, 7, 7},
                           {10, 10, 10, 7, 7, 7, 4, 0, 0, 0},
                           {10, 10, 10, 6, 6, 5, 4, 0, 0, 0},
                           {10, 10, 10, 6, 5, 4, 4, 0, 0, 0},
                           {10, 10, 10, 6, 3, 3, 3, 0, 0, 0},
                           {3, 3, 3, 3, 3, 3, 3, 0, 0, 0},
                           {3, 3, 3, 3, 3, 3, 3, 0, 0, 0},
                           {3, 3, 3, 3, 3, 3, 3, 0, 0, 0}}, s);
        frontframe(10, 10, 10, s);
    \end{diagram}
    \caption{Three ideals in $P_{10}$ corresponding to $I$ in \cref{fig:csscshells1} and the shells in \cref{fig:csscshells2}.}
    \label{fig:csscshells3}
\end{figure}

\begin{figure}[htbp]
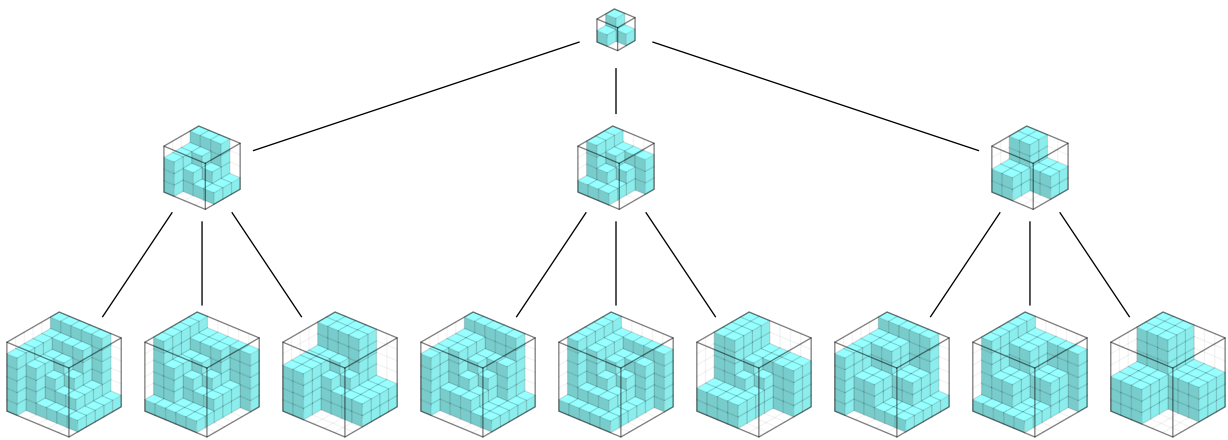

    \centering
    \begin{diagram}[1]
        void v(int[][] heights, pair o) {
            int l = heights.length;
            pair s = (l*v1+l*v2+l*v3)/2;
            fill(circle(o, 2/3*l+2), white);
            backframe(l, l, l, o-s);
            ideal(heights, o-s);
            frontframe(l, l, l, o-s);
        }
        int[][][] ideals = new int[][][] {
            {{2, 1}, {1, 0}},
            {{4, 4, 4, 1}, {3, 3, 2, 1}, {3, 2, 1, 1}, {3, 0, 0, 0}},
            {{4, 3, 3, 3}, {4, 3, 2, 0}, {4, 2, 1, 0}, {1, 1, 1, 0}},
            {{4, 4, 2, 2}, {4, 4, 2, 2}, {2, 2, 0, 0}, {2, 2, 0, 0}},
            {{6, 6, 6, 6, 6, 1}, {5, 5, 5, 5, 2, 1}, {5, 4, 4, 3, 2, 1}, {5, 4, 3, 2, 2, 1}, {5, 4, 1, 1, 1, 1}, {5, 0, 0, 0, 0, 0}},
            {{6, 5, 5, 5, 5, 5}, {6, 5, 5, 5, 2, 0}, {6, 4, 4, 3, 2, 0}, {6, 4, 3, 2, 2, 0}, {6, 4, 1, 1, 1, 0}, {1, 1, 1, 1, 1, 0}},
            {{6, 6, 6, 6, 2, 2}, {6, 6, 6, 6, 2, 2}, {4, 4, 4, 3, 2, 2}, {4, 4, 3, 2, 2, 2}, {4, 4, 0, 0, 0, 0}, {4, 4, 0, 0, 0, 0}},
            {{6, 6, 6, 6, 6, 1}, {5, 5, 4, 4, 4, 1}, {5, 5, 4, 3, 1, 1}, {5, 5, 3, 2, 1, 1}, {5, 2, 2, 2, 1, 1}, {5, 0, 0, 0, 0, 0}},
            {{6, 5, 5, 5, 5, 5}, {6, 5, 4, 4, 4, 0}, {6, 5, 4, 3, 1, 0}, {6, 5, 3, 2, 1, 0}, {6, 2, 2, 2, 1, 0}, {1, 1, 1, 1, 1, 0}},
            {{6, 6, 4, 4, 4, 4}, {6, 6, 4, 4, 4, 4}, {6, 6, 4, 3, 0, 0}, {6, 6, 3, 2, 0, 0}, {2, 2, 2, 2, 0, 0}, {2, 2, 2, 2, 0, 0}},
            {{6, 6, 6, 6, 6, 1}, {5, 5, 5, 3, 3, 1}, {5, 5, 5, 3, 3, 1}, {5, 3, 3, 1, 1, 1}, {5, 3, 3, 1, 1, 1}, {5, 0, 0, 0, 0, 0}},
            {{6, 5, 5, 5, 5, 5}, {6, 5, 5, 3, 3, 0}, {6, 5, 5, 3, 3, 0}, {6, 3, 3, 1, 1, 0}, {6, 3, 3, 1, 1, 0}, {1, 1, 1, 1, 1, 0}},
            {{6, 6, 6, 3, 3, 3}, {6, 6, 6, 3, 3, 3}, {6, 6, 6, 3, 3, 3}, {3, 3, 3, 0, 0, 0}, {3, 3, 3, 0, 0, 0}, {3, 3, 3, 0, 0, 0}}};
        real s = 12;
        pair[] locations = new pair[] {
            s*(0, 2.5),
            s*(-3, 1.5),
            s*(0, 1.5),
            s*(3, 1.5),
            s*(-4, 0),
            s*(-3, 0),
            s*(-2, 0),
            s*(-1, 0),
            s*(0, 0),
            s*(1, 0),
            s*(2, 0),
            s*(3, 0),
            s*(4, 0)};
        int[][] edges = new int[][] {
            {0, 1},
            {0, 2},
            {0, 3},
            {1, 4},
            {1, 5},
            {1, 6},
            {2, 7},
            {2, 8},
            {2, 9},
            {3, 10},
            {3, 11},
            {3, 12}};
        for (int i = 0; i < edges.length; ++i) {
            int v1 = edges[i][0];
            int v2 = edges[i][1];
            draw(locations[v1] -- locations[v2]);
        }
        for (int i = 0; i < ideals.length; ++i) {
            v(ideals[i], locations[i]);
        }
    \end{diagram}
    \caption{The first three layers of the tree structure arising from the proof of \cref{theorem:furthestfromcenter}.}
    \label{fig:tree}
\end{figure}

\begin{remark}
    The perimeter of $G_{C_3}(2r, 2r, 2r)$ is a subset of $P_{2r}$ (and appears to be a strict subset for $r\ge 4$).
\end{remark}

\section{Totally Symmetric Self-Complementary Ideals}\label{section:tssc}
The definition of a flip graph can be modified to accommodate totally symmetric self-complementary (TSSC) ideals of a poset $P = [\ell] \times [\ell] \times [\ell] = [2r] \times [2r] \times [2r]$ for an even positive integer $\ell=2r$.

Given two distinct TSSC ideals $I, J \subset P$, $I$ and $J$ differ by a \emph{flip} if there exists a point $(a_1, a_2, a_3) \in I$ such that the replacements
\[(a_{\sigma(1)}, a_{\sigma(2)}, a_{\sigma(3)})\mapsto (2r + 1 - a_{\sigma(1)}, 2r + 1 - a_{\sigma(2)}, 2r + 1 - a_{\sigma(3)})\]
over all $\sigma\in S_3$ yield $J$. The \emph{flip graph on totally symmetric self-complementary ideals of $P$} is a weighted graph whose edges are constructed between TSSC ideals of $P$ that differ by a flip. If $a_1, a_2, a_3$ are distinct, connect $I$ and $J$ with an edge of weight one. If exactly two of $a_1, a_2, a_3$ are equal, connect $I$ and $J$ with an edge of weight two. Note that, like the CSSC ideal case, it is not possible for $a_1 = a_2 = a_3$.

\cref{fig:tsscgraph} shows the flip graph structure for the specific case $P = [6] \times [6] \times [6]$. Edges of weight one are denoted with a thinner line segment, and edges of weight two are denoted with a thicker line segment.

\begin{figure}[htbp]
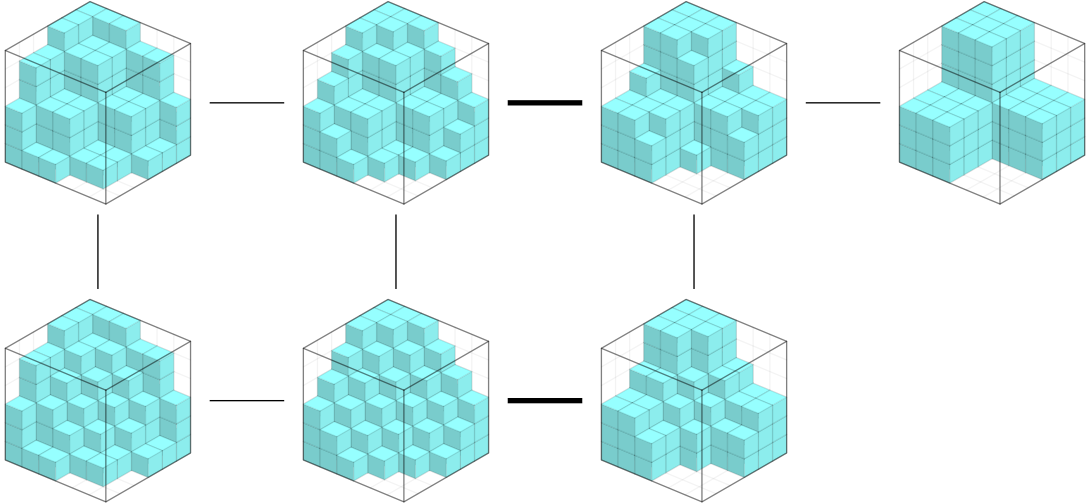

    \centering
    \begin{diagram}[0.9]
        void v(int[][] heights, pair o) {
            pair s = (6*v1+6*v2+6*v3)/2;
            fill(circle(o, 6), white);
            backframe(6, 6, 6, o-s);
            ideal(heights, o-s);
            frontframe(6, 6, 6, o-s);
        }
        int[][][] ideals = new int[][][] {
            {{6, 6, 6, 3, 3, 3}, {6, 6, 6, 3, 3, 3}, {6, 6, 6, 3, 3, 3}, {3, 3, 3, 0, 0, 0}, {3, 3, 3, 0, 0, 0}, {3, 3, 3, 0, 0, 0}},
            {{6, 6, 6, 4, 3, 3}, {6, 6, 6, 3, 3, 3}, {6, 6, 5, 3, 3, 2}, {4, 3, 3, 1, 0, 0}, {3, 3, 3, 0, 0, 0}, {3, 3, 2, 0, 0, 0}},
            {{6, 6, 6, 4, 3, 3}, {6, 6, 6, 4, 3, 3}, {6, 6, 4, 3, 2, 2}, {4, 4, 3, 2, 0, 0}, {3, 3, 2, 0, 0, 0}, {3, 3, 2, 0, 0, 0}},
            {{6, 6, 6, 5, 4, 3}, {6, 6, 5, 3, 3, 2}, {6, 5, 5, 3, 3, 1}, {5, 3, 3, 1, 1, 0}, {4, 3, 3, 1, 0, 0}, {3, 2, 1, 0, 0, 0}},
            {{6, 6, 6, 5, 4, 3}, {6, 6, 5, 4, 3, 2}, {6, 5, 4, 3, 2, 1}, {5, 4, 3, 2, 1, 0}, {4, 3, 2, 1, 0, 0}, {3, 2, 1, 0, 0, 0}},
            {{6, 6, 6, 5, 5, 3}, {6, 5, 5, 3, 3, 1}, {6, 5, 5, 3, 3, 1}, {5, 3, 3, 1, 1, 0}, {5, 3, 3, 1, 1, 0}, {3, 1, 1, 0, 0, 0}},
            {{6, 6, 6, 5, 5, 3}, {6, 5, 5, 4, 3, 1}, {6, 5, 4, 3, 2, 1}, {5, 4, 3, 2, 1, 0}, {5, 3, 2, 1, 1, 0}, {3, 1, 1, 0, 0, 0}}};
        real s = 16;
        pair[] locations = new pair[] {
            (3, 1)*s,
            (2, 1)*s,
            (2, 0)*s,
            (1, 1)*s,
            (1, 0)*s,
            (0, 1)*s,
            (0, 0)*s};
        int[][] edges1 = new int[][] {
            {0, 1},
            {1, 2},
            {3, 4},
            {3, 5},
            {4, 6},
            {5, 6}};
        int[][] edges2 = new int[][] {
            {1, 3},
            {2, 4}};
        for (int i = 0; i < edges1.length; ++i) {
            int v1 = edges1[i][0];
            int v2 = edges1[i][1];
            draw(locations[v1] -- locations[v2]);
        }
        for (int i = 0; i < edges2.length; ++i) {
            int v1 = edges2[i][0];
            int v2 = edges2[i][1];
            draw(locations[v1] -- locations[v2], linewidth(2));
        }
        for (int i = 0; i < ideals.length; ++i) {
            v(ideals[i], locations[i]);
        }
    \end{diagram}
    \caption{The flip graph on totally symmetric self-complementary ideals of $[6] \times [6] \times [6]$.}
    \label{fig:tsscgraph}
\end{figure}

\begin{lemma}
    Let $I$ and $J$ be totally symmetric self-complementary ideals of $[2r]\times [2r]\times [2r]$. Then the distance between $I$ and $J$ in the flip graph is $\frac{1}{3}\abs{I\setminus J}$. 
\end{lemma}

\begin{proof}\label{lemma:tsscdistance}
    This is essentially the same as the proofs of \cref{lemma:scdistance} and \cref{lemma:csscdistance}.
\end{proof}

Given a positive integer $r$, let $G_{S_3}(2r, 2r, 2r)$ denote the flip graph on totally symmetric self-complementary ideals of $[2r] \times [2r] \times [2r]$. Like before, for each tuple $(t_1, t_2, t_3) \in \{0, 1\}^3$ define
\[O_{t_1t_2t_3} = \{(a_1, a_2, a_3) \in [2r] \times [2r] \times [2r] \mid t_1 = \bbmone_{a_1 > r}, t_2 = \bbmone_{a_2 > r}, t_3 = \bbmone_{a_3 > r}\}.\]

\subsection{Vertex Count}

The number of vertices in $G_{S_3}(2r, 2r, 2r)$, or the number of totally symmetric self-complementary ideals of $[2r] \times [2r] \times [2r]$, was determined by Andrews in 1994.

\begin{theorem}\label{theorem:tsscvertex}
    Let $r$ be a positive integer. The number of vertices in $G_{S_3}(2r, 2r, 2r)$ is
    \[\prod _{j=0}^{r-1}{\frac {(3j+1)!}{(r+j)!}}.\]
\end{theorem}

\begin{proof}
    This is given in \cite{andrews1994plane}.
\end{proof}

\begin{corollary}
    As $r$ tends to infinity,
    \[v(G_{S_3}(2r, 2r, 2r)) \sim Cr^{-5/36}\left(\frac{3\sqrt{3}}{4}\right)^{r^2}\]
    for some real number $C$. In particular, $\log v(G_{S_3}(2r, 2r, 2r)) \sim \frac{1}{2} \log(27/16) r^2 \approx 0.262r^2$.
\end{corollary}

\begin{proof}
    One can take \[C=\left(\frac{e^{\zeta'(-1)}2^{1/4}\Gamma(2/3)}{3^{1/12}\Gamma(1/3)}\right)^{1/3} \approx 0.775.\] This is the square root of the result in \cref{corollary:csscasymptotics} because the vertex count of $G_{S_3}(2r, 2r, 2r)$ is the square root of the vertex count of $G_{C_3}(2r, 2r, 2r)$.
\end{proof}

\subsection{Diameter}

The diameter of the $G_{S_3}(2r, 2r, 2r)$ is known exactly. To determine this exact value, first consider the following lemma.

\begin{lemma}\label{lemma:tsscmustinclude}
    All totally symmetric self-complementary ideals $I\subset [2r]\times [2r]\times [2r]$ must contain the points $(a_1, a_2, a_3)$ that satisfy
    \begin{itemize}
        \item $2a_1\le 2r+1$ and $a_2 + a_3\le 2r+1$, or
        \item $2a_2\le 2r+1$ and $a_3 + a_1\le 2r+1$, or
        \item $2a_3\le 2r+1$ and $a_1 + a_2\le 2r+1$.
    \end{itemize}
\end{lemma}

\begin{proof}
    If $(a_1, a_2, a_3)\notin I$,  then $(2r+1-a_1, 2r+1-a_2, 2r+1-a_3)\in I$, as are $(2r+1-a_1, 2r+1-a_3, 2r+1-a_2)$, $(2r+1-a_3, 2r+1-a_2, 2r+1-a_1)$, and $(2r+1-a_2, 2r+1-a_1, 2r+1-a_3)$. If any of the above inequalities hold, then one of these three points covers $(a_1, a_2, a_3)$, a contradiction.
\end{proof}

To visualize the points satisfying the above inequalities, we show the example of $r=5$ in \cref{fig:tsscmustinclude}.

This also means that the TSSC ideal with the fewest points in $O_{001}$ contains precisely those points only and the remaining octants filled in accordingly. On the other hand, the TSSC ideal with the most points in $O_{001}$ is the one that contains the entirety of $O_{001}$. The case of $r=5$ is depicted in \cref{fig:tsscdiameter}.

\begin{figure}[htbp]
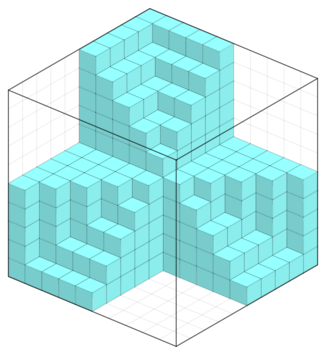

    \centering
    \begin{diagram}[0.25]
        backframe(10, 10, 10);
        ideal(new int[][] {{10, 10, 10, 10, 10, 5, 5, 5, 5, 5},
                           {10, 9, 9, 9, 9, 5, 5, 5, 5, 1}, 
                           {10, 9, 8, 8, 8, 5, 5, 5, 2, 1}, 
                           {10, 9, 8, 7, 7, 5, 5, 3, 2, 1}, 
                           {10, 9, 8, 7, 6, 5, 4, 3, 2, 1}, 
                           {5, 5, 5, 5, 5, 0, 0, 0, 0, 0}, 
                           {5, 5, 5, 5, 4, 0, 0, 0, 0, 0}, 
                           {5, 5, 5, 3, 3, 0, 0, 0, 0, 0}, 
                           {5, 5, 2, 2, 2, 0, 0, 0, 0, 0}, 
                           {5, 1, 1, 1, 1, 0, 0, 0, 0, 0}});
        frontframe(10, 10, 10);
    \end{diagram}
    \caption{Points that must be included in all totally symmetric self-complementary ideals of $[10] \times [10] \times [10]$.}
    \label{fig:tsscmustinclude}
\end{figure}

\begin{figure}[htbp]
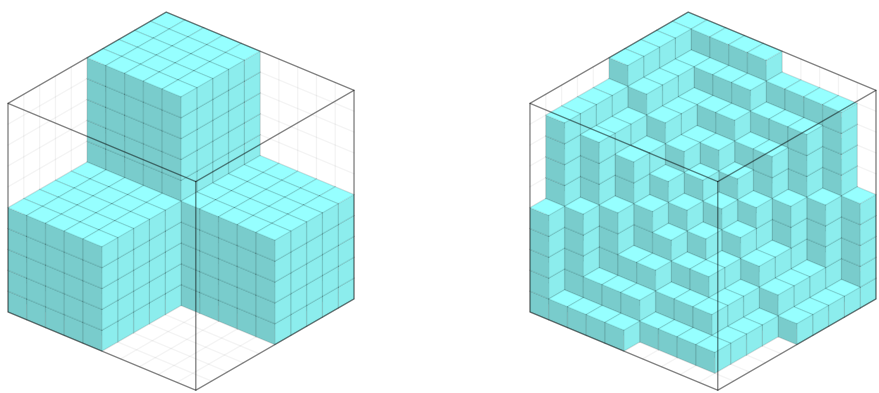

    \centering
    \begin{diagram}[0.7]
        backframe(10, 10, 10);
        ideal(new int[][] {{10, 10, 10, 10, 10, 5, 5, 5, 5, 5},
                           {10, 10, 10, 10, 10, 5, 5, 5, 5, 5}, 
                           {10, 10, 10, 10, 10, 5, 5, 5, 5, 5},
                           {10, 10, 10, 10, 10, 5, 5, 5, 5, 5},
                           {10, 10, 10, 10, 10, 5, 5, 5, 5, 5},
                           {5, 5, 5, 5, 5, 0, 0, 0, 0, 0},
                           {5, 5, 5, 5, 5, 0, 0, 0, 0, 0},
                           {5, 5, 5, 5, 5, 0, 0, 0, 0, 0},
                           {5, 5, 5, 5, 5, 0, 0, 0, 0, 0},
                           {5, 5, 5, 5, 5, 0, 0, 0, 0, 0}});
        frontframe(10, 10, 10);
        pair s = (25, 0);
        backframe(10, 10, 10, s);
        ideal(new int[][] {{10, 10, 10, 10, 10, 9, 9, 9, 9, 5},
                           {10, 9, 9, 9, 9, 8, 8, 8, 5, 1}, 
                           {10, 9, 8, 8, 8, 7, 7, 5, 2, 1},
                           {10, 9, 8, 7, 7, 6, 5, 3, 2, 1},
                           {10, 9, 8, 7, 6, 5, 4, 3, 2, 1},
                           {9, 8, 7, 6, 5, 4, 3, 2, 1, 0},
                           {9, 8, 7, 5, 4, 3, 3, 2, 1, 0},
                           {9, 8, 5, 3, 3, 2, 2, 2, 1, 0},
                           {9, 5, 2, 2, 2, 1, 1, 1, 1, 0},
                           {5, 1, 1, 1, 1, 0, 0, 0, 0, 0}}, s);
        frontframe(10, 10, 10, s);
    \end{diagram}
    \caption{$I$ and $J$ for the case $r=5$. These ideals attain the maximum possible distance in the flip graph on totally symmetric self-complementary ideals of $[10]\times [10]\times [10]$.}
    \label{fig:tsscdiameter}
\end{figure}

In general for the $[2r]\times [2r]\times [2r]$ case, the symmetric difference between these two TSSC ideals in $O_{001}$, which is also their distance in the flip graph, is $1^2 + 2^2 + \cdots + (r-1)^2 = \frac{1}{6}(r-1)r(2r-1)$.

\begin{theorem}\label{theorem:tsscdiameter}
    The diameter of $G_{S_3}(2r, 2r, 2r)$ is $\frac{1}{6}(r-1)r(2r-1)$.
\end{theorem}

\begin{proof}
    It remains to show that for any two TSSC ideals $I$ and $J$ of $[2r]\times [2r]\times [2r]$, we have that $\frac{1}{6}\abs{I\symdif J}\le \frac{1}{6}(r-1)r(2r-1)$. Since $\abs{I\cap J}\ge 4r^3 - 3\cdot\frac{1}{6}(r-1)r(2r-1)$ by \cref{lemma:tsscmustinclude}, we have that 
    \begin{align*}
    \abs{I\symdif J} &= \abs{I} + \abs{J} - 2\abs{I\cap J}\\
    &\le 4r^3 + 4r^3 - 2\left(4r^3 - 3\cdot\tfrac{1}{6}(r-1)r(2r-1)\right)\\
    &= (r-1)r(2r-1),
    \end{align*}
    as desired.
    
    Observe that equality holds when $|I\cap J| = 4r^3 - 3\cdot\frac{1}{6}(r-1)r(2r-1)$, i.e.\ when the intersection of $I$ and $J$ is the ideal as depicted in \cref{lemma:tsscmustinclude}. This can only occur when one ideal contains no points in $O_{110}$ and the other ideal contains the minimum possible number of points in $O_{001}$, which are the two ideals depicted in \cref{fig:tsscdiameter}.
\end{proof}

\subsection{Radius}
Only a lower bound for the radius is known, which follows directly from the diameter bound (\cref{theorem:tsscdiameter}).
\begin{conjecture}\label{conjecture:tsscradius}
    The radius of $G_{S_3}(2r, 2r, 2r)$ is $\ceil{\frac{1}{12}(r-1)r(2r-1)}$.
\end{conjecture}

From our code for $r\le 6$, it appears that the radius equals the ceiling of half of the diameter (see \cref{section:code} for a link to the source code). Our code gave the following results:

\begin{center}
    \begin{tabular}{c|c|c|c}
        $r$ & vertex count & diameter & radius \\
        \hline
        1 & 1 & 0 & 0\\
        \hline
        2 & 2 & 1 & 1\\
        \hline
        3 & 7 & 5 & 3\\
        \hline
        4 & 42 & 14 & 7\\
        \hline
        5 & 429 & 30 & 15\\
        \hline
        6 & 7436 & 55 & 28
    \end{tabular}
\end{center}

The size of the center is 1 for $r=3$ and $r=4$, but greater than 1 for $r=5$ and $r=6$ (and most likely for $r > 6$ as well). The elements in the center of the flip graph for $r\le 5$ are depicted in \cref{fig:tssccenter3}, \cref{fig:tssccenter4}, and \cref{fig:tssccenter5}.

\begin{figure}[p]
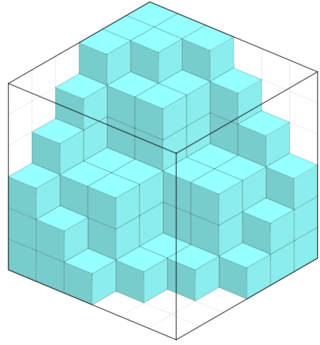

    \centering
    \begin{diagram}[0.25]
        backframe(6, 6, 6);
        ideal(new int[][] {{6, 6, 6, 5, 4, 3}, {6, 6, 5, 3, 3, 2}, {6, 5, 5, 3, 3, 1}, {5, 3, 3, 1, 1, 0}, {4, 3, 3, 1, 0, 0}, {3, 2, 1, 0, 0, 0}});
        frontframe(6, 6, 6);
    \end{diagram}
    \caption{The center for $r=3$.}
    \label{fig:tssccenter3}
\end{figure}

\begin{figure}[p]
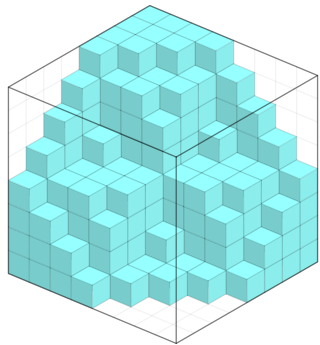

    \centering
    \begin{diagram}[0.25]
        backframe(8, 8, 8);
        ideal(new int[][] {{8, 8, 8, 8, 7, 6, 5, 4}, {8, 8, 8, 7, 5, 4, 4, 3}, {8, 8, 7, 7, 4, 4, 4, 2}, {8, 7, 7, 6, 4, 4, 3, 1}, {7, 5, 4, 4, 2, 1, 1, 0}, {6, 4, 4, 4, 1, 1, 0, 0}, {5, 4, 4, 3, 1, 0, 0, 0}, {4, 3, 2, 1, 0, 0, 0, 0}});
        frontframe(8, 8, 8);
    \end{diagram}
    \caption{The center for $r=4$.}
    \label{fig:tssccenter4}
\end{figure}

\begin{figure}[p]
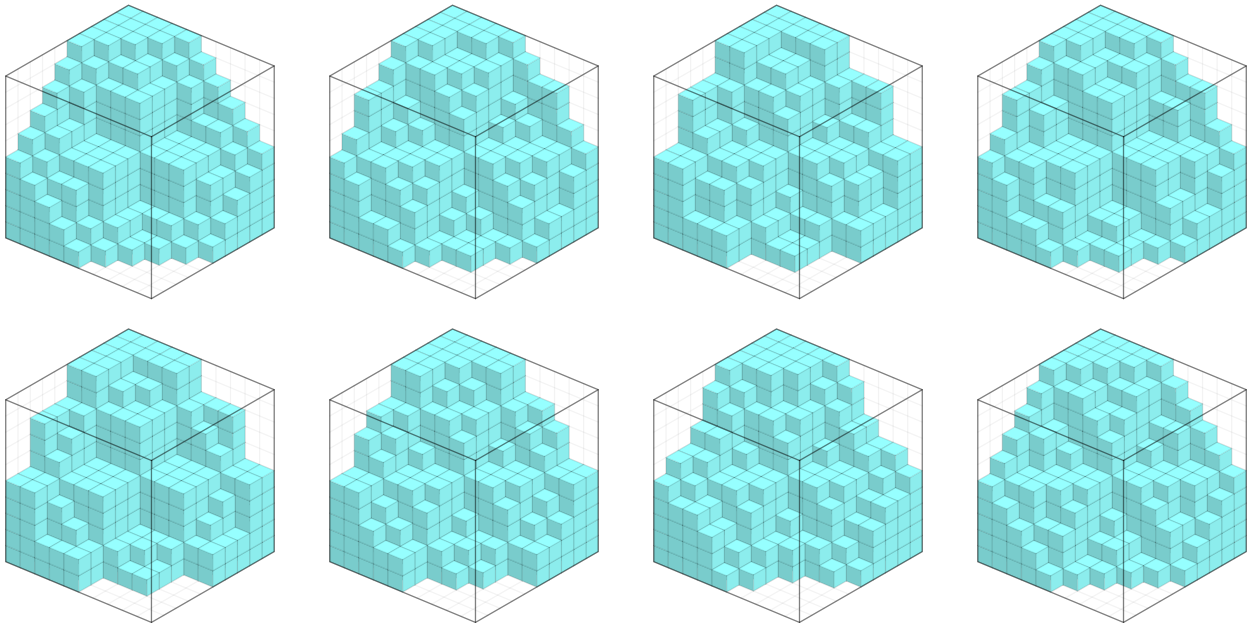

    \centering
    \begin{diagram}[1.0]
        int[][][] ideals = new int[][][]
        {
            {{10, 10, 10, 10, 10, 9, 8, 7, 6, 5}, {10, 10, 10, 10, 9, 8, 6, 6, 5, 4}, {10, 10, 10, 9, 8, 5, 5, 5, 4, 3}, {10, 10, 9, 9, 8, 5, 5, 5, 4, 2}, {10, 9, 8, 8, 8, 5, 5, 5, 2, 1}, {9, 8, 5, 5, 5, 2, 2, 2, 1, 0}, {8, 6, 5, 5, 5, 2, 1, 1, 0, 0}, {7, 6, 5, 5, 5, 2, 1, 0, 0, 0}, {6, 5, 4, 4, 2, 1, 0, 0, 0, 0}, {5, 4, 3, 2, 1, 0, 0, 0, 0, 0}},
            {{10, 10, 10, 10, 10, 9, 8, 8, 6, 5}, {10, 10, 10, 10, 9, 7, 6, 5, 5, 4}, {10, 10, 9, 9, 9, 6, 5, 5, 5, 2}, {10, 10, 9, 9, 8, 5, 5, 5, 4, 2}, {10, 9, 9, 8, 7, 5, 5, 4, 3, 1}, {9, 7, 6, 5, 5, 3, 2, 1, 1, 0}, {8, 6, 5, 5, 5, 2, 1, 1, 0, 0}, {8, 5, 5, 5, 4, 1, 1, 1, 0, 0}, {6, 5, 5, 4, 3, 1, 0, 0, 0, 0}, {5, 4, 2, 2, 1, 0, 0, 0, 0, 0}},
            {{10, 10, 10, 10, 10, 8, 8, 8, 5, 5}, {10, 10, 10, 10, 10, 8, 6, 6, 5, 5}, {10, 10, 9, 9, 8, 6, 5, 5, 4, 2}, {10, 10, 9, 9, 8, 5, 5, 5, 4, 2}, {10, 10, 8, 8, 7, 5, 5, 4, 2, 2}, {8, 8, 6, 5, 5, 3, 2, 2, 0, 0}, {8, 6, 5, 5, 5, 2, 1, 1, 0, 0}, {8, 6, 5, 5, 4, 2, 1, 1, 0, 0}, {5, 5, 4, 4, 2, 0, 0, 0, 0, 0}, {5, 5, 2, 2, 2, 0, 0, 0, 0, 0}},
            {{10, 10, 10, 10, 10, 9, 8, 8, 6, 5}, {10, 10, 10, 10, 9, 7, 7, 5, 5, 4}, {10, 10, 9, 9, 9, 5, 5, 5, 5, 2}, {10, 10, 9, 8, 8, 5, 5, 5, 3, 2}, {10, 9, 9, 8, 8, 5, 5, 5, 3, 1}, {9, 7, 5, 5, 5, 2, 2, 1, 1, 0}, {8, 7, 5, 5, 5, 2, 2, 1, 0, 0}, {8, 5, 5, 5, 5, 1, 1, 1, 0, 0}, {6, 5, 5, 3, 3, 1, 0, 0, 0, 0}, {5, 4, 2, 2, 1, 0, 0, 0, 0, 0}},
            {{10, 10, 10, 10, 10, 8, 8, 8, 5, 5}, {10, 10, 10, 10, 10, 8, 7, 6, 5, 5}, {10, 10, 9, 9, 8, 5, 5, 5, 4, 2}, {10, 10, 9, 8, 8, 5, 5, 5, 3, 2}, {10, 10, 8, 8, 8, 5, 5, 5, 2, 2}, {8, 8, 5, 5, 5, 2, 2, 2, 0, 0}, {8, 7, 5, 5, 5, 2, 2, 1, 0, 0}, {8, 6, 5, 5, 5, 2, 1, 1, 0, 0}, {5, 5, 4, 3, 2, 0, 0, 0, 0, 0}, {5, 5, 2, 2, 2, 0, 0, 0, 0, 0}},
            {{10, 10, 10, 10, 10, 8, 8, 7, 5, 5}, {10, 10, 10, 10, 10, 8, 7, 6, 5, 5}, {10, 10, 10, 9, 8, 6, 5, 5, 4, 3}, {10, 10, 9, 8, 8, 5, 5, 5, 3, 2}, {10, 10, 8, 8, 7, 5, 5, 4, 2, 2}, {8, 8, 6, 5, 5, 3, 2, 2, 0, 0}, {8, 7, 5, 5, 5, 2, 2, 1, 0, 0}, {7, 6, 5, 5, 4, 2, 1, 0, 0, 0}, {5, 5, 4, 3, 2, 0, 0, 0, 0, 0}, {5, 5, 3, 2, 2, 0, 0, 0, 0, 0}},
            {{10, 10, 10, 10, 10, 9, 7, 6, 6, 5}, {10, 10, 10, 10, 9, 8, 7, 6, 5, 4}, {10, 10, 10, 10, 8, 6, 5, 5, 4, 4}, {10, 10, 10, 8, 8, 5, 5, 5, 3, 3}, {10, 9, 8, 8, 7, 5, 5, 4, 2, 1}, {9, 8, 6, 5, 5, 3, 2, 2, 1, 0}, {7, 7, 5, 5, 5, 2, 2, 0, 0, 0}, {6, 6, 5, 5, 4, 2, 0, 0, 0, 0}, {6, 5, 4, 3, 2, 1, 0, 0, 0, 0}, {5, 4, 4, 3, 1, 0, 0, 0, 0, 0}},
            {{10, 10, 10, 10, 10, 9, 8, 7, 6, 5}, {10, 10, 10, 10, 9, 7, 7, 5, 5, 4}, {10, 10, 10, 9, 9, 6, 5, 5, 5, 3}, {10, 10, 9, 8, 8, 5, 5, 5, 3, 2}, {10, 9, 9, 8, 7, 5, 5, 4, 3, 1}, {9, 7, 6, 5, 5, 3, 2, 1, 1, 0}, {8, 7, 5, 5, 5, 2, 2, 1, 0, 0}, {7, 5, 5, 5, 4, 1, 1, 0, 0, 0}, {6, 5, 5, 3, 3, 1, 0, 0, 0, 0}, {5, 4, 3, 2, 1, 0, 0, 0, 0, 0}}
        };
        for (int i = 0; i < ideals.length; ++i) {
            pair s = 20*((i 
            backframe(10, 10, 10, s);
            ideal(ideals[i], s);
            frontframe(10, 10, 10, s);
        }
    \end{diagram}
    \caption{The center for $r=5$.}
    \label{fig:tssccenter5}
\end{figure}

An ideal in the center must contain $\floor{\frac{1}{12}(r-1)r(2r-1)}$ or $\ceil{\frac{1}{12}(r-1)r(2r-1)}$ points in $O_{110}$ by considering its distance to two ideals shown in \cref{fig:tsscdiameter}.
However, there are many such ideals, and this condition appears to be necessary but not sufficient, as the majority of TSSC ideals satisfying this condition are not in the center for $r=5$.
A proof of the upper bound for the radius may involve picking a certain ``nice" element in the center and proving that all TSSC ideals are a distance of at most $\ceil{\frac{1}{12}(r-1)r(2r-1)}$ in the flip graph.
A similar shell induction argument as in the proof of \cref{lemma:csscradiusupperbound} may not necessarily work, as the core of ideals in the center of $G_{S_3}(2r, 2r, 2r)$ do not appear to be ideals in the center of $G_{S_3}(2r-2, 2r-2, 2r-2)$ for $r\in\{4, 5\}$.

We are still hopeful that this conjecture is relatively tractable, as the radius appears to be essentially half of the diameter.

\section{Code}\label{section:code}
The code we wrote while doing research can be found at \url{https://github.com/anser0/spur}. A description of each of the folders is as follows:

\begin{itemize}
    \item \href{https://github.com/anser0/spur/tree/main/chainproducts}{\texttt{/chainproducts}} contains code to generate all self-complementary ideals of $[\ell_1]\times [\ell_2]\times [\ell_3]$ for small $\ell_i$.
    \item \href{https://github.com/anser0/spur/tree/main/cssc}{\texttt{/cssc}} contains code to generate all cyclically-symmetric ideals of $[2r]\times [2r]\times [2r]$, their eccentricities, and the ideals furthest from the unique element in the center for $r\le 4$.
    \item \href{https://github.com/anser0/spur/tree/main/data}{\texttt{/data}} contains select data from other files, such as all totally symmetric self-complementary ideals of $[2r]\times [2r]\times [2r]$ and their eccentricities for $r\le 4$.
    \item \href{https://github.com/anser0/spur/tree/main/graphvisualizer}{\texttt{/graphvisualizer}} contains code to generate and visually represent the flip graph on self-complementary ideals of chain products; six examples are shown in \cref{fig:csflipgraphs}. 
    \begin{itemize}
        \item \href{https://github.com/anser0/spur/tree/main/graphvisualizer/images}{\texttt{/graphvisualizer/images}} contains images of the flip graphs of $[\ell_1] \times \dots \times [\ell_d]$ for all $\ell_1 \cdots \ell_d \leq 64$.
    \end{itemize}
    \item \href{https://github.com/anser0/spur/tree/main/tssc}{\texttt{/tssc}} contains code to generate all totally symmetric self-complementary ideals of $[2r]\times [2r]\times [2r]$, their eccentricities for $r\le 6$, and the points that all totally symmetric self-complementary ideals must contain.
\end{itemize}

\begin{figure}[p]
    \centering
    \begin{tabular}{ll}
    \includegraphics[height=.42\linewidth]{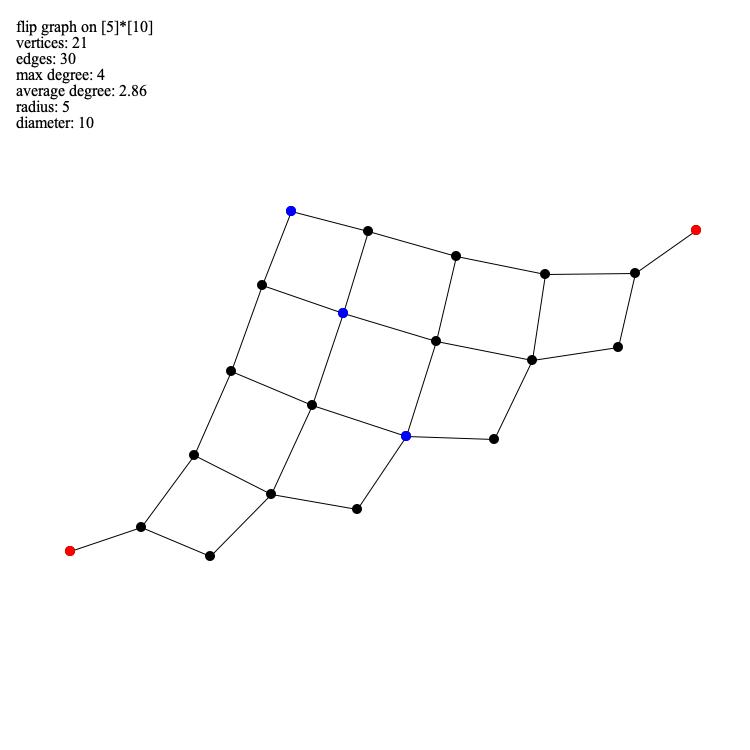} & \includegraphics[width=.42\linewidth]{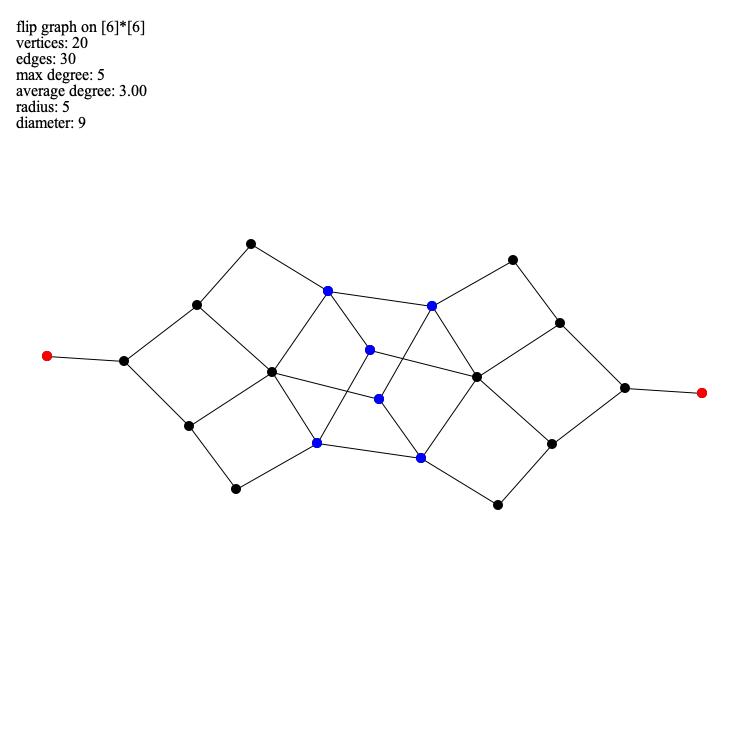} \\
    \includegraphics[width=.42\linewidth]{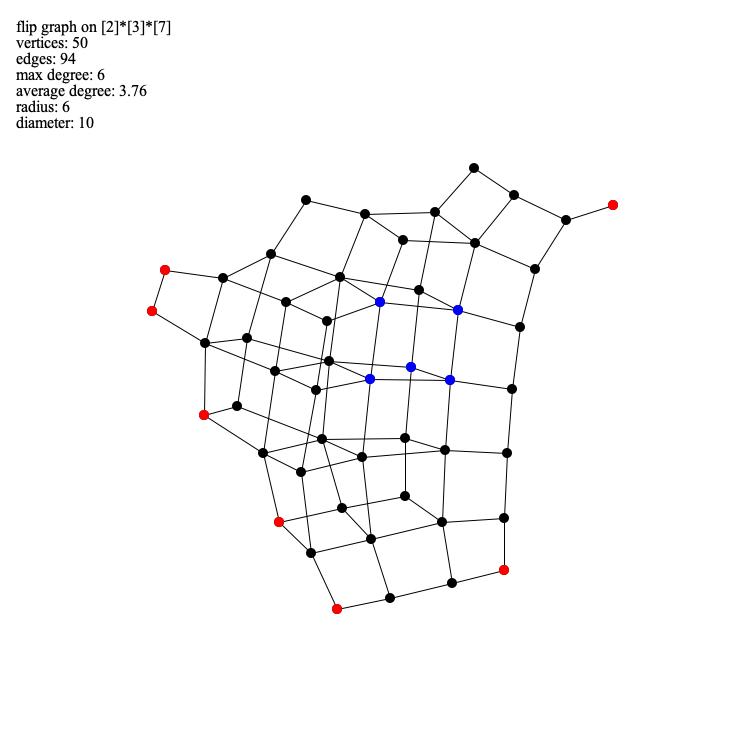} & \includegraphics[width=.42\linewidth]{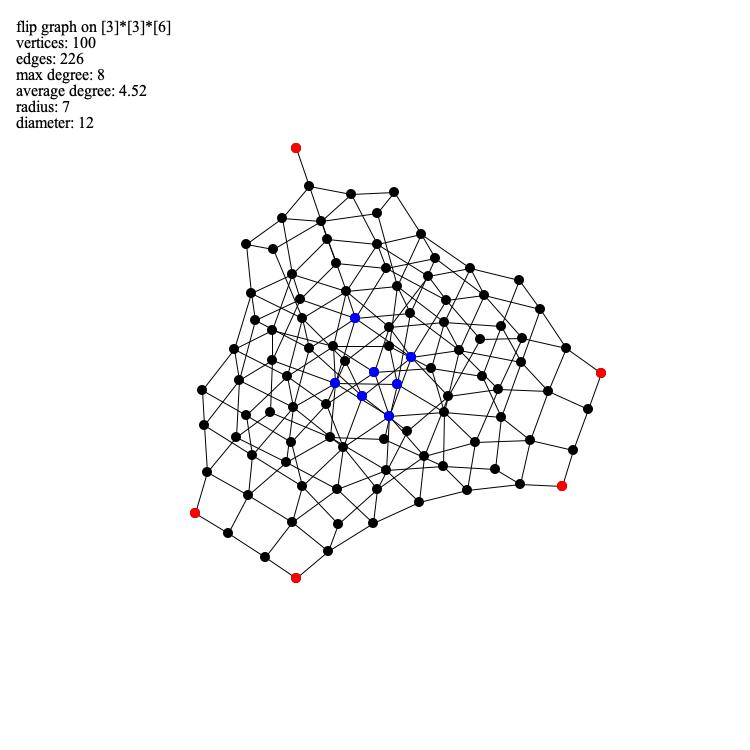} \\
    \includegraphics[width=.42\linewidth]{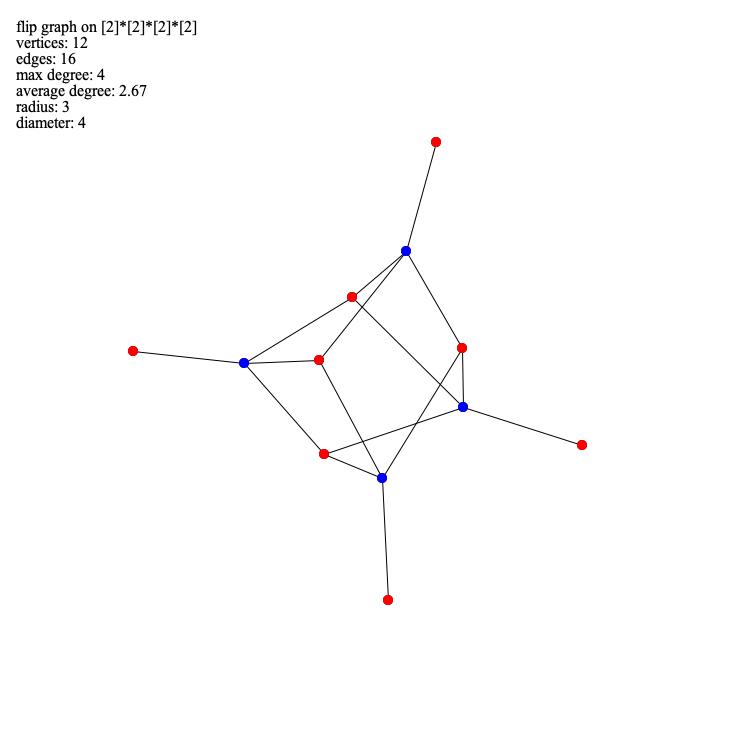} & \includegraphics[width=.42\linewidth]{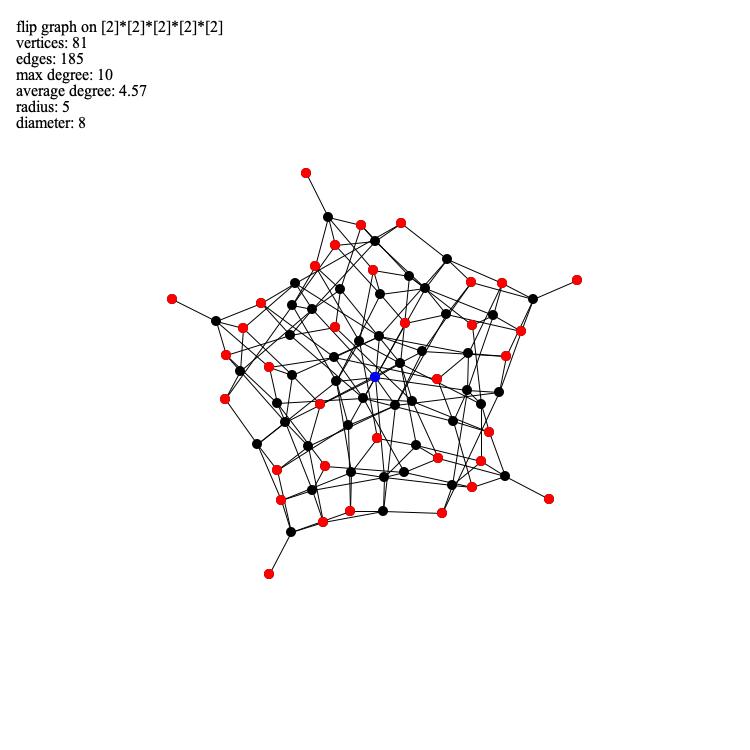} \\
    \end{tabular}
    \caption{Visual representations of sample flip graphs on self-complementary ideals of chain products. Elements in the center are colored blue, and elements in the perimeter are colored red.}
    \label{fig:csflipgraphs}
\end{figure}

\section{Future Directions}\label{section:futuredirections}
Currently, there are three main open conjectures in this paper:
\begin{itemize}
    \item \cref{conjecture:scvertex4} conjectures an asymptotic bound on the number of self-complementary ideals of $[\ell]^n$ for a fixed even positive integer $\ell$, as $n$ tends to infinity.
    \item \cref{conjecture:chvatal3} is a specific case of Chv\'atal's conjecture. If true, it would imply an exact answer for the radius of the flip graph on self-complementary ideals of chain products.
    \item \cref{conjecture:tsscradius} conjectures that the radius of the flip graph on totally symmetric self-complementary ideals of $[\ell]^3$ is given by the ceiling of half of the diameter.
\end{itemize}

In addition to the three conjectures above, there are also a number of other questions that can be asked about flip graph structures on self-dual posets. These include, but are not limited to, the following questions:

\begin{itemize}
    \item Is it possible to improve on \cref{theorem:scvertex3} to obtain more precise vertex count asymptotics for flip graphs on self-complementary ideals?
    \item In this paper, the only graph properties studied were vertex count, diameter, and radius. Other properties of interest could be:
    \begin{itemize}
        \item edge count and average degree
        \item maximal degree, which is related to Sperner's theorem
        \item characterization of the center, and bounds on its size
        \item characterization of the perimeter, and bounds on its size.
    \end{itemize}
\end{itemize}

\section{Acknowledgments}\label{section:acknowledgments}
This research was conducted through the MIT math department's Summer Program for Undergraduate Research (SPUR). First and foremost, we would like to thank our mentor Elisabeth Bullock for her continuous support and guidance throughout the program. We would also like to thank Professor Alexander Postnikov for suggesting this project. Lastly, we would like to thank Professor David Jerison for organizing SPUR and for his thoughtful comments about our research.

\bibliographystyle{alpha}
\bibliography{bibliography}

\end{document}